\documentclass[12pt,reqno]{amsart}
\usepackage{amsfonts, amssymb}
\usepackage{amsmath, amscd}
\usepackage{color}
\usepackage[colorlinks=true, pdfstartview=FitV, linkcolor=blue, citecolor=blue, urlcolor=blue]{hyperref}
\usepackage{tikz}
\usepackage{graphicx}
\usepackage{mathtools}
\usepackage{enumitem}
\usetikzlibrary{arrows, decorations.markings}
\usetikzlibrary{decorations.markings, arrows.meta}
 \usepackage[all]{xy}
\newlength{\bibitemsep}
\newlength{\bibparskip}
\let\oldthebibliography\thebibliography
\renewcommand\thebibliography[1]{%
  \oldthebibliography{#1}%
  \setlength{\parskip}{\bibitemsep}%
  \setlength{\itemsep}{\bibparskip}%
}

\theoremstyle{plain}
\newtheorem{theorem}{Theorem}[section]
\newtheorem{lemma}[theorem]{Lemma}
\newtheorem{proposition}[theorem]{Proposition}
\newtheorem{corollary}[theorem]{Corollary}

\theoremstyle{definition}
\newtheorem{remark}[theorem]{Remark}
\newtheorem{definition}[theorem]{Definition}

\newcommand{\nc}{\newcommand}
\nc{\rnc}{\renewcommand}

\nc{\on}[1]{\operatorname{#1}}
\nc{\delete}[1]{}
\nc{\onbf}[1]{\on{\bf #1}}
\nc{\onsf}[1]{\on{\sf #1}}
\nc{\red}{\textcolor{red}}
\nc{\blue}{\textcolor{blue}}
\nc{\bb}[1]{{\mathbb #1}}
\nc{\mbf}[1]{{\mathbf #1}}
\nc{\cal}[1]{{\mathcal #1}}
\nc{\mf}[1]{{\mathfrak #1}}

\font\cyr=wncyr10
\nc{\sha}{{\mbox{\cyr X}}}

\nc{\udim}{\underline{\on{dim}}}
\nc{\guvw}{g_{U,V}^{W}}

\nc{\htt}{\on{ht}}
\nc{\gr}{\on{\bold {gr}}}
\nc{\cf}{\on{cf}}
\nc{\ad}{\on{ad }}
\nc{\Ad}{\on{Ad}}
\nc{\id}{\on{Id}}
\nc{\Fr}{\on{Fr}}
\nc{\Der}{\on{Der}}
\nc{\End}{\on{End}}
\nc{\tor}{\on{Tor}}
\nc{\Ext}{\on{Ext}}
\nc{\ext}{\on{ext}}
\nc{\Fil}{\on{Fil}}
\nc{\Hom}{\on{Hom}}
\nc{\grhom}{\on{grHom}}
\nc{\ch}{\on{ch}}
\nc{\Ch}{\on{\bf Ch}}
\nc{\ind}{\on{Ind}}
\nc{\coind}{\on{Coind}}

\nc{\Mod}{\on{-Mod}}
\nc{\biMod}{\on{-biMod}}
\nc{\Poiss}{\on{-Poiss}}
\nc{\grmod}{\on{-grMod}}
\nc{\vamod}{\on{-VAMod}}
\nc{\res}{\on{Res}}
\nc{\soc}{\on{Soc}}
\nc{\rad}{\on{Rad}}
\nc{\Aut}{\on{Aut}}
\nc{\Dist}{\on{Dist}}
\nc{\Lie}{\on{Lie}}
\nc{\Ker}{\on{Ker}}
\nc{\im}{\on{Im}}
\nc{\wt}{\on{wt}}
\nc{\st}{\on{St}}
\nc{\diag}{\on{Diag}}
\nc{\rep}{\on{rep}}
\nc{\Set}{\on{\bf Set}}
\nc{\sSet}{\on{\bf sSet}}
\nc{\smCat}{\on{\bf smCat}}
\nc{\spec}{\on{spec}}
\nc{\Sym}{\on{Sym}}
\nc{\Vecs}{\on{\bf Vec}^{s}}
\nc{\colim}{\operatornamewithlimits{\underset{\longrightarrow}{lim}}}

\nc{\graphdot}{\onwithlimits{\bullet}}
\nc{\lrarrows}{\onwithlimits{\leftrightarrows}}

\nc{\interval}[1]{\mathinner{#1}}

\nc{\blist}{\begin{list}{\rom{(\roman{enumi})}}
{\setlength{\leftmargin}{0em}
\setlength{\itemindent}{7ex}
\setlength{\labelsep}{2ex}\setlength{\listparindent}{\parindent}
\usecounter{enumi}}}
\nc{\elist}{\end{list}}

\newcommand{\cc}{\ensuremath{\mathbb{C}}}

\nc{\subsub}[1]{\noindent{\bf #1}}

\usepackage{etoolbox}
\apptocmd{\thebibliography}{\setlength{\itemsep}{5pt}}{}{}

\allowdisplaybreaks
\usepackage[overload]{empheq}

\makeatletter \def\l@subsection{\@tocline{2}{0pt}{1pc}{5pc}{}} \def\l@subsection{\@tocline{2}{0pt}{2pc}{6pc}{}} \makeatother

\begin{document}

\title{Differential graded vertex Lie algebras}
\author[A. Caradot]{Antoine Caradot$^1$}
\address{$^1$School of Mathematics and Statistics\\
Henan University \\
Kaifeng, Henan, CHINA}
\email{caradot@henu.edu.cn}
\author[C. Jiang]{Cuipo Jiang$^2$}
\address{$^2$School  of Mathematics\\
Shanghai Jiao Tong University \\
Shanghai, CHINA}
\email{cpjiang@math.sjtu.edu.cn}
\author[Z. Lin]{Zongzhu Lin$^3$}
\address{$^3$Department of Mathematics\\
Kansas State University \\
Manhattan, KS 66506, USA}
\email{zlin@math.ksu.edu}


\date{\today}


\maketitle

\begin{abstract} This is the continuation of the study of differential graded (dg) vertex algebras defined in our previous paper \cite{Caradot-Jiang-Lin-4}. The goal of this paper is to construct a functor from the category of dg vertex Lie algebras to the category of dg vertex algebras which is left adjoint to the forgetful functor. This functor not only provides an abundant number of examples of dg vertex algebras, but it is also an important step in constructing a homotopy theory (\cite{Avramov-Halperin, Quillen}) in the category of vertex algebras. Vertex Lie algebras were introduced as analogues of vertex algebras, but in which we only consider the singular part of the vertex operator map and the equalities it satisfies. In this paper, we extend the definition of vertex Lie algebras to the dg setting. We construct a pair of adjoint functors between the categories of dg vertex algebras and dg vertex Lie algebras, which leads to the explicit construction of dg vertex (operator) algebras. We will give examples based on the Virasoro algebra, the Neveu-Schwarz algebra, and dg Lie algebras. 
\end{abstract}

\tableofcontents
\addtocontents{toc}


\section{Introduction}
 This paper is a continuation of our previous work \cite{Caradot-Jiang-Lin-4} on differential graded (dg) vertex algebras. A dg vertex algebra is a super vertex algebra lifted to the symmetric tensor category of differential complexes over the field $ \bb C$. In the earlier paper, we did not provide a single example of dg vertex algebras other than the standard known ones. The focus of that paper was comparing the corresponding Poisson algebra and Zhu algebra as well as their representations. In the present paper, we focus on the universal construction of dg vertex algebras from dg Lie algebras with  invariant bilinear forms (as done in \cite{Fattori-Kac} in the super setting). This construction defines a left adjoint functor to the restriction functor from the category of dg vertex algebras to the category of dg vertex Lie algebras, which are automatically dg Lie algebras (see Theorem \ref{thm:U_L(U)-}). This not only provides a large collection of dg vertex algebras, corresponding to dg vertex Lie algebras and fixed levels, but then by passing to simple quotients, one obtains examples of simple dg vertex algebras (i.e. without nontrivial dg ideals). It also opens a new study on the construction of simplicial resolutions of vertex algebra extensions analogous to Andr\'e-Quillen (co)homology theory (see \cite[Ch. 8]{Weibel}) and establishes a homotopy theory (as well as a notion of derived vertex algebra from the view point of Lurie's derived algebraic geometry). We will not deal with the coset construction of dg vertex algebras in this paper. 

 The second objective of this paper is to construct a resolution for vertex algebras  analogous to the Tate resolution of a commutative algebra in terms of free differentially graded commutative algebras (see \cite{Tate}). The goal of studying this type of resolution is to establish a homotopy theory for the category of vertex algebras. There is a natural functor from the category of the vertex algebras to the category of Poisson algebras. The category of commutative algebras has a homotopy theory established by Andr\'e and Quillen in terms of Andr\'e-Quillen (co)homomology theory.  We refer to the work of Avramov-Halperin (\cite{Avramov-Halperin}) for more details on topological and commutative-algebraic correspondences of the rational cohomology theories. The case of a local commutative algebra contains a hidden Lie algebra called homotopy Lie algebra, with the Yoneda algebra being its universal enveloping algebra. In \cite{Caradot-Jiang-Lin-1} we compared the Yoneda algebra associated to the representation category of a vertex algebra and those of the Zhu algebra and its associated Poisson algebras. For most interesting vertex algebras, the $C_2$-algebra $ R(V)$ is a local algebra and the associated Poisson variety is simply a point. However, this exactly corresponds to the local commutative algebra case. In \cite{Caradot-Jiang-Lin-2} we studied the Yoneda algebra, which suggests richer geometric invariants for the vertex algebra, and we saw a Koszul duality of vertex algebras at least in terms of their $C_2$-algebras (see also \cite{Caradot-Jiang-Lin-1}). The computation of the Yoneda algebra is based on the constructing of a ``good" resolution of the residue field as an algebra, which was first done by Tate. It turns out this construction becomes the cornerstone for a later formulation of algebraic homotopy theory (see \cite{Quillen}).  As stated in \cite{Caradot-Jiang-Lin-4}, one would like to lift the Tate resolution for commutative algebras to vertex algebras. This was one of the motivations for introducing dg vertex algebras in \cite{Caradot-Jiang-Lin-4}. The dg world is the natural one in order for the dualities to be seen much naturally. This is exactly Quillen's principle (\cite{Avramov-Halperin}): {\em Differential graded algebras ought not to be regarded merely as a tool for the calculation of (co) homology. In fact a reasonable DGA category will
also carry a ``homotopy theory" and with it a number of other invariants.}

In \cite{Caradot-Jiang-Lin-4} we defined a dg vertex algebra with the motivation that the polynomial algebra $ \bb C[t]=H^*_{\bb C^*}(pt)$ is a graded algebra with $t$ in degree 2 in mind. In the present paper, we generalise the definition with $\bb C[t]$ as a dg algebra with $ t$ in degree $2N$. The case when $ N=0$ has appeared in the literature in various constructions. For example in the BRST construction (\cite{Arakawa}), one naturally gets a dg vertex algebra with $ N=0$. Butson has also dealt with dg vertex algebras with $N=0$ in \cite{Butson1}. 

 Vertex algebras are an algebraic formulation of 2-dimensional conformal field theory, and the structure and representation theory, from an algebraic perspective, are the main focus of study. Recently, in computing other invariants in mathematical physics and algebraic geometry, vertex algebras have appeared to be a natural structure to regroup those invariants together. Thus the results on vertex algebras and their representation theory can be natural tools to compute those invariants. The cohomologies of stacks in \cite{Joyce} and \cite{BLM} are examples of such invariants. Butson (\cite{Butson2}) has also constructed vertex algebras associated to divisors on Calabi-Yau threefolds. Differential graded vertex algebras are the lifting of the cohomology to certain differential complexes. Indeed, a dg vertex algebra automatically induces a graded vertex algebra (in the category of graded vector spaces with Koszul braiding) by forgetting the differential, which then leads to a natural super vertex algebra structure with even and odd parts of the cohomologies. The construction of the universal vertex algebra from a vertex Lie algebra will reduce the construction of dg vertex algebras to constructing dg vertex Lie algebras, which should be relatively simpler. We expect that various standard complexes computing homology/cohomology theories such as resolutions also carry standard dg vertex Lie algebra structures. 

Given a dg vertex algebra and considering the singular part, one gets a dg vertex Lie algebra. This defines a functor from the category of dg vertex algebras to the category of dg vertex Lie algebras. We will show that this functor has a left adjoint. This left adjoint functor plays the role of free object in the category of dg vertex algebras relative to dg vertex Lie algebras. Analogously, the standard functor from the category of associative dg algebras to the category of dg Lie algebras admits a left adjoint functor.

\delete{On the other hand, there are two functors from the category of dg vertex algebras to the category of associative algebras. For a dg vertex algebra $V^{[*]}$, there is associated commutative dg Poisson algebra, the $C_2$-algebra $R^{[*]}(V^{[*]})$. If the vertex algebra has a conformal structure of the vertex algebra has a filtered structure, there is also another associative algebra $A(V^{[*]})$, the Zhu algebra. However, the Zhu algebra $A(V^{[*]})$ is not a dg algebra. Instead, $A(V^{[*]})$ is a filtered differential filtered algebra structure. The bifiltered structure gives rise to three graded algebra. Describing these algebras for a general dg vertex algebra is a difficult question. However, in this paper, we will show that when the dg vertex algebras are free over the category of dg Lie algebras, then the universal enveloping algebra of a Lie algebra will appear as one of the associated dg algebras of the Zhu algebra under certain conditions. The $C_2$ algebra will always be the symmetric algebra of the dg Lie algebra. These Poisson algebras are resolving commutative dg algebras in the sense of dg schemes in \cite{Behrend}. These algebras will play the role of smooth dg schemes and thus turn out to be the associated dg scheme of the dg vertex algebra. }

The motivation for studying dg vertex algebras is to construct vertex algebra objects in a more general compactly generated closed symmetric monoidal category $\bb T$.  Let $X$ be an algebraic variety with an action of an algebraic group $G$, i.e., $[X/G]$ is quotient stack.  Our constructions of vertex algebras and vertex Lie algebras in this and earlier papers should work when the symmetric tensor category $\Ch$ of differential complexes of $\bb C$-vector spaces  is replaced by $ \bb T=\Ch(\onsf{Qcoh}^G(X))$,  the category of differential complexes of quasi-coherent $G$-equivariant sheaves on $X$ (i.e., the category of differential complexes of quasi-coherent sheaves on the stack $[X/G]$).  Our goal is to describe the vertex algebra objects on a suitable triangulated category such as the derived category $D^G(\onsf{Coh}(X))$ of $G$-equivariant coherent sheaves, or $D_c^G(X)$ the derived category of $G$-equivariant constructible sheaves on the algebraic variety $X$. Both of them carry a closed symmetric tensor category structure. Then most of the cohomological constructions would become more natural in these categories. We should remark that the variety $X$ plays a completely different role to that of the curves studied by Ben-Zvi and Frenkel  (\cite{Frenkel-Ben-Zvi}). 

The paper is organised as follows. In Section \ref{sec:2}, we briefly recall some basic properties of dg vertex algebras which will be used later on. In Section \ref{sec:3}, we follow the treatment in \cite{Lepowsky-Li} to obtain the dg version of weak vertex operator structures on the endomorphism complexes $\mathcal{E}(W)=\on{Hom}(W, W((x)))$. This will be used to establish the reconstruction theorem stating that a collection of mutually local weak vertex operators automatically generates a vertex algebra, but in the dg setting. In Section \ref{sec:4}, we define what a dg vertex Lie algebra is and state some basic properties. We remark that vertex Lie algebras are also called conformal algebras by Kac (\cite{Kac}) by taking $ \lambda=x^{-1}$. The definition and properties work for conformal algebras as well. Section \ref{sec:5} is devoted to constructing the enveloping dg vertex algebra of a dg vertex Lie algebra. This construction establishes a left adjoint to the forgetful functor from the category of dg vertex algebras to the category of dg vertex Lie algebras. In Section \ref{sec:6}, we apply the left adjoint functor to several well-known dg vertex Lie algebras to construct dg vertex algebras.

\section{Preliminaries on dg vertex algebras}\label{sec:2}
In this section we will briefly modify the definitions previously given in \cite{Caradot-Jiang-Lin-4} and extend the results we will need in the later sections.

\begin{definition}\label{def:dgVA}
A \textbf{differential graded (dg) vertex algebra} in the category $\Ch$ of differential (cochain) complexes of $\cc$-vector spaces is a cochain complex $(V^{[*]}, d_V^{[*]})$ over $ \cc$ equipped with a chain map (vertex operator map) in $ \Ch$ 
\begin{align*}
\begin{array}{cccc}
Y(\cdot, x): & V^{[*]} & \longrightarrow & \Hom^{[*]}(\bb C[t, t^{-1}], \End^{[*]}(V^{[*]}))=\on{End}^{[*]}(V^{[*]})[[x,x^{-1}]] \\
           & v & \longmapsto      & Y(v, x)=\displaystyle \sum_{n \in \mathbb{Z}}v_n x^{-n-1}
\end{array}
\end{align*}
with $x$ of cohomological degree $|x|=-2N$ for $N \in \mathbb{Z}$ (i.e., $|t|=2N$), and a particular vector $\mbf{1} \in V^{[0]}$ with $d^{[0]}(\mbf 1)=0$, the vacuum vector, satisfying the following conditions:   
\begin{enumerate}[label=(\roman*).]
\item For any $u, v \in V^{[*]}$, $u_nv=0$ for $n$ sufficiently large, i.e., $Y(u, x)v \in V^{[*]}((x))$ (Truncation property).
\item $Y(\mathbf{1},x)=\on{id}_V$ (Vacuum property).
\item $Y(v, x)\mathbf{1} \in V[[x]]$ and $\displaystyle \lim_{x \to 0} Y(v, x)\mathbf{1}=v$ (Creation property).
\item The Jacobi identity 
\begin{align}\label{eq:jacobi}
\begin{split}
 x_2^{-1} \delta(\frac{x_1-x_0}{x_2})Y(Y(u,& x_0)v, x_2)=x_0^{-1}\delta(\frac{x_1-x_2}{x_0})Y(u, x_1)Y(v, x_2) \\[5pt]
& - (-1)^{|u||v|}x_0^{-1}\delta(\frac{x_2-x_1}{-x_0})Y(v, x_2)Y(u, x_1),
\end{split}
\end{align}
where $\delta(x+y)=\displaystyle \sum_{n \in \mathbb{Z}}\sum_{m=0}^{\infty}\binom{n}{m}x^{n-m}y^m$ and $v,u\in V^{[*]}$ are homogeneous.
\end{enumerate}
In particular, as $Y(\cdot, x)$ is homogeneous of degree 0, it follows that for $v \in V^{[|v|]}$, we have
\[
|v_n| =|v|-2N(n+1).
\]
\end{definition}

In the above definition and the rest of this paper, a chain map will denote a homogeneous linear map of degree $0$ that commutes with the differential.

Given a cochain complex $U^{[*]}$ and homogeneous linear maps $f, g: U^{[*]} \longrightarrow U^{[*]}$, we write $[f, g]^s$ for the super commutator $f \circ g - (-1)^{|f||g|}g \circ f$. If $f$ (or $g$) is not homogeneous, then it is first decomposed into homogeneous components $\sum_i f_i$ (or $\sum_j g_j$) and then $[f, g]^s=\sum_i[f_i, g]^s$ (or $[f, g]^s=\sum_j[f, g_j]^s$).

\begin{proposition}[Weak dg commutativity]\label{prop:3.2.1}
Let $u, v \in V^{[*]}$ be homogeneous. Then there exists $k \in \mathbb{N}$ such that
\begin{align}\label{eq:3.2.1}
(x_1-x_2)^k[Y(u, x_1), Y(v, x_2)]^s=0.
\end{align}
\end{proposition}

\begin{proof}
Take $k \geq 0$, multiply the Jacobi identity by $x_0^k$, and then take $\on{Res}_{x_0}$. We obtain
\begin{align*}
(x_1-x_2)^kY(u, x_1)Y(v, x_2)-(-1)^{|u||v|}(x_1-x_2)^kY(v, x_2)Y(u, x_1)= \\
\on{Res}_{x_0}\bigg[x_2^{-1}\delta\big(\frac{x_1-x_0}{x_2} \big) x_0^k Y(Y(u, x_0)v, x_2)\bigg].
\end{align*}
Adjust $k$ such that $u_n v=0$ for $n \geq k$. Then the right hand side becomes zero. As $Y(\cdot, x)$ is a chain map, $|u|=|Y(u, x)|$ for any homogeneous $u \in V^{[*]}$, hence we get Equation~\eqref{eq:3.2.1}.
\end{proof}

We can verify that Propositions 3.2.7, 3.2.12, 3.3.1, 3.3.5, 3.3.8, 3.3.17, 3.3.19, 3.4.1, 3.4.3 and Remark 3.3.18, 3.3.20 of \cite{Lepowsky-Li} are also true in the dg setting. The statements only need to include a coefficient $(-1)^{|u||v|}$ whenever there is a permutation between $u$ and $v$. 

As explained in \cite[Lemma 3.8]{Caradot-Jiang-Lin-4}, a dg vertex algebra satisfies the following formulas: 
\begin{align}\label{eq:Dderivative_1}
[ \mathcal D, Y(v, x)]^s=\frac{d}{dx}Y(v, x),
\end{align}
\begin{align}\label{eq:Dderivative_2}
[ \mathcal D, Y(v, x)]^s=Y(\mathcal{D}v, x),
\end{align}
where $\mathcal{D}$ is $2N$-cocycle in $Z^{[2N]}\on{End}^{[*]}(V^{[*]})$ defined by $\mathcal{D}(v)=v_{-2}\mathbf{1}$ for all $v \in V^{[*]}$. It is also proved in \cite{Caradot-Jiang-Lin-4} that a dg vertex algebra satisfies a graded version of the skew-symmetry, i.e., for all $u, v \in V^{[*]}$ homogeneous, we have
\begin{align}\label{eq:skewsym}
Y(u, x)v=(-1)^{|u||v|}e^{x\cal D}Y(v, -x)u.
\end{align}

The following theorem is proved in the same manner as in \cite[Theorem 3.5.1]{Lepowsky-Li}.

\begin{theorem}\label{thm:3.5.1}
The Jacobi identity~\eqref{eq:jacobi} for dg vertex algebras follows from the weak dg commutativity~\eqref{eq:3.2.1} in the presence axioms (i), (ii), (iii) of Definition \ref{def:dgVA}  together with the $\mathcal{D}$-bracket derivative formula~\eqref{eq:Dderivative_1}. In particular, in the definition of a dg vertex algebra, the Jacobi identity can be replaced by those properties.
\end{theorem}

\delete{
\begin{proof}
As $\mathcal{D}$ is of even degree, the $\mathcal{D}$-bracket formula becomes $[\mathcal{D}, Y(u, x)]^s=\frac{d}{dx}Y(u, x)$. We can then obtain the same formula as in \cite[(3.1.35)]{Lepowsky-Li}:
\[
e^{x\mathcal{D}}Y(u, x_0)e^{-x\mathcal{D}}=Y(x_0+x).
\]
We can then write $Y(u, x_0+x)\mathbf{1}=e^{x\mathcal{D}}Y(u, x_0)e^{-x\mathcal{D}}\mathbf{1}=e^{x\mathcal{D}}Y(u, x_0)\mathbf{1}$. Using the creation property, we can set $x_0=0$ and get $Y(u, x)\mathbf{1}=e^{x\mathcal{D}}v$.

We have assumed that the weak commutativity is satisfied, so take $u, v \in V^{[*]}$ homogeneous and $k \in \mathbb{N}$ such that $x^kY(v, x)u$ involves only non negative powers of $x$, and such that the weak commutativity relation is satisfied for $u$ and $v$. Then
\[
\begin{array}{rcl}
(x_1-x_2)^kY(u, x_1)Y(v, x_2)\mathbf{1} & = & (-1)^{|u||v|}(x_1-x_2)^kY(v, x_2)Y(u, x_1)\mathbf{1}, \\[5pt]
 & = & (-1)^{|u||v|}(x_1-x_2)^kY(v, x_2)e^{x_1\mathcal{D}}u, \\[5pt]
 & = & (-1)^{|u||v|}(x_1-x_2)^ke^{x_1\mathcal{D}}Y(v, x_2-x_1)u.
 \end{array}
\]
We can then set $x_2=0$ because $(x_1-x_2)^kY(v, x_2-x_1)u$ only involves non negative powers of $x_2-x_1$. By applying the creation property, we get
\[
x_1^kY(u, x_1)v=(-1)^{|u||v|}x_1^ke^{x_1\mathcal{D}}Y(v, -x_1)u.
\]
Multiplying both sides by $x_1^k$, we obtain the skew-symmetry relation for $(u, v)$.

Set $u, v, w \in V^{[*]}$ homogeneous. Using the weak commutativity and the skew-symmetry, we have
\[
\begin{array}{rcl}
(x_0+x_2)^kY(u, x_0+x_2)Y(v, x_2)w& = & (x_0+x_2)^kY(u, x_0+x_2)(-1)^{|v||w|}e^{x_2\mathcal{D}}Y(w, -x_2)v, \\[5pt]
 & = & (-1)^{|v||w|}e^{x_2\mathcal{D}}(x_0+x_2)^kY(u, x_0)Y(w, -x_2)v, \\[5pt]
 & = & (-1)^{(|u|+|v|)|w|}e^{x_2\mathcal{D}}(x_0+x_2)^kY(w, -x_2)Y(u, x_0)v, \\[5pt]
 & = & (-1)^{(|u|+|v|)|w|}(x_0+x_2)^kY(Y(u, x_0)v, x_2)w.
 \end{array}
\]
Hence we have the weak associativity. The result then follows from the dg version of \cite[Proposition 3.4.3]{Lepowsky-Li}.
\end{proof}
}

The definition of a module for a dg vertex algebra was also given in \cite{Caradot-Jiang-Lin-4}.

\begin{definition}\label{def:dgVAmod}
Let $(V^{[*]}, d_V, Y(\cdot, x), \mathbf{1})$ be a dg vertex algebra in $ \Ch$. A  dg module over $V^{[*]}$ is an object  $(M^{[*]}, d_M)$ in $\Ch$  equipped with a chain map (homogeneous of degree 0)
\begin{align*}
\begin{array}{cccc}
Y_M(\cdot, x): & V^{[*]} & \longrightarrow &\mathrm{End}^{[*]} (M^{[*]})[[x,x^{-1}]] \\
           & v & \longmapsto      & Y_M(v, x)=\displaystyle \sum_{n \in \mathbb{Z}}v_n x^{-n-1}
\end{array}
\end{align*}
such that for any $u,v \in V$, the following properties are verified:
\vspace{-\topsep}
\begin{enumerate}[label=(\roman*).]
\item For any $u \in V^{[*]}$, $w \in M^{[*]}$ , $u_n w=0$ for $n$ sufficiently large, i.e., $Y_M(u, x)w \in M^{[*]}((x))$. (Truncation property).
\item $Y_M(\mathbf{1},x)=\mathrm{id}_{| M}$. (Vacuum property).
\item The Jacobi identity 
 \begin{align}\label{eq:jacobi_mod} 
 \begin{split}
 x_2^{-1}\delta(\frac{x_1-x_0}{x_2}) Y_M(Y(u, &x_0)v, x_2)=  x_0^{-1}\delta(\frac{x_1-x_2}{x_0})Y_M(u, x_1)Y_M(v, x_2) \\
 & -(-1)^{|v||u|}x_0^{-1}\delta(\frac{x_2-x_1}{-x_0})Y_M(v, x_2)Y_M(u, x_1). 
  \end{split}
\end{align}
\end{enumerate}
\end{definition}

The restricted dual of a cochain complex $U^{[*]}$ will be written $(U^{[*]})'=\bigoplus_{n \in \mathbb{Z}} (U^{[n]})^*$ where $(U^{[n]})^*=\on{Hom}(U^{[n]}, \cc)$ (notice the difference in use between $^{*}$ and $^{[*]}$).

\begin{theorem}\label{thm:3.6.3} 
Let $(V^{[*]}, Y(\cdot, x), \mathbf{1})$ be a triple that satisfies axioms (i), (ii) and (iii) in Definition \ref{def:dgVA}, and in addition has the skew-symmetry~\eqref{eq:skewsym}. Let $W^{[*]}$ be a cochain complex and let $Y_W(\cdot, x): V^{[*]} \longrightarrow \on{End}^{[*]}(W^{[*]})[[x, x^{-1}]]$ be a chain map such that the conditions (i) and (ii) of Definition \ref{def:dgVAmod} are satisfied. Finally, assume that the weak associativity holds for any $u, v \in V^{[*]}$, $w \in W^{[*]}$, i.e., there exists $l \in \mathbb{N}$ (depends on $u, v, w$) such that
\begin{align}\label{eq:weakassoc}
(x_0+x_2)^lY_W(Y(u, x_0)v, x_2)w=(x_0+x_2)^lY_W(u, x_0+x_2)Y_W(v, x_2)w.
\end{align}
Then the Jacobi identity~\eqref{eq:jacobi_mod} for $Y_W(\cdot, x)$ is satisfied.
\end{theorem}

\begin{proof}
For $u \in V^{[*]}$ and $w \in W^{[*]}$, by the weak associativity assumption, there exists $l \in \mathbb{N}$ such that
\[
(x_0+x_2)^lY_W(Y(u, x_0)\mathbf{1}, x_2)w=(x_0+x_2)^lY_W(u, x_0+x_2)Y_W(\mathbf{1}, x_2)w.
\]
By the creation property and the skew-symmetry on $V^{[*]}$, we have $Y(u, x)\mathbf{1}=e^{x\mathcal{D}}u$. Hence
\[
(x_0+x_2)^lY_W(e^{x_0 \mathcal{D}}u, x_2)w=(x_0+x_2)^lY_W(u, x_0+x_2)w
\]
because $Y_W(\mathbf{1}, x_2)=\on{id}_W$. By the truncation condition on $W^{[*]}$, we can choose $l$ such that $x^l Y_W(u, x)w \in W^{[*]}[[x]]$. But then we can permute $x_0$ and $x_2$ in the right hand side of the above expression, and so
\[
(x_0+x_2)^lY_W(e^{x_0 \mathcal{D}}u, x_2)w=(x_0+x_2)^lY_W(u, x_2+x_0)w.
\]
Since all factors now involve only non negative powers of $x_0$, we can multiply by $(x_2+x_0)^{-l}$ to get
\[
Y_W(e^{x_0 \mathcal{D}}u, x_2)w=Y_W(u, x_2+x_0)w=e^{x_0\frac{d}{x_2}}Y_W(u, x_2)w.
\]
It follows that $Y_W(\mathcal{D}u, x)=\frac{d}{dx}Y_W(u, x)$.

We know from \cite[Propositions 3.3.17, 3.3.19 and Remarks 3.3.18, 3.3.20]{Lepowsky-Li} that $Y_W(\cdot, x)$ satisfies the formal associativity relation (given by \eqref{eq:3.6.17} and \eqref{eq:3.6.18}). Based on the proof of \cite[Proposition 3.4.3]{Lepowsky-Li}, in order to prove the Jacobi identity, we just have to prove the following formal commutativity relation (given by \eqref{eq:3.6.14} and \eqref{eq:3.6.15}): for any $u, v \in V^{[*]}$ homogeneous, $w \in W^{[*]}$, $w' \in (W^{[*]})'$, there exists
\begin{align}\label{eq:3.6.13}
f(x_1, x_2) \in \cc[[x_1, x_2]][x_1^{-1}, x_2^{-1}, (x_1-x_2)^{-1}]
\end{align}
such that
\begin{align}
\langle w', Y_W(u, x_1)Y_W(v, x_2)w \rangle = \iota_{12}f(x_1, x_2) \label{eq:3.6.14}, \\[5pt]
(-1)^{|u||v|}\langle w', Y_W(v, x_2)Y_W(u, x_1)w \rangle = \iota_{21}f(x_1, x_2) .\label{eq:3.6.15}
\end{align}
Using \cite[Propositions 3.3.17, 3.3.19]{Lepowsky-Li}, there exists $f$ as in Equation~\eqref{eq:3.6.13} such that Equation~\eqref{eq:3.6.14} holds and such that
\begin{align}\label{eq:3.6.16}
\langle w', Y_W(Y(u, x_0)v, x_2)w \rangle = \iota_{20}f(x_0+x_2, x_2) .
\end{align}
By reversing the roles of $u$ and $v$, of $x_1$ and $x_2$, and of $x_0$ and $-x_0$, we obtain
\begin{align*}
f' \in \cc[[x_1, x_2]][x_1^{-1}, x_2^{-1}, (x_1-x_2)^{-1}]
\end{align*}
such that
\begin{align}
\langle w', Y_W(v, x_2)Y_W(u, x_1)w \rangle &= \iota_{21}f'(x_1, x_2) \label{eq:3.6.17}, \\[5pt]
(-1)^{|u||v|}\langle w', Y_W(Y(v, -x_0)u, x_1)w \rangle &= \iota_{10}f'(x_1, x_1-x_0). \label{eq:3.6.18}
\end{align}
In order to show Equation~\eqref{eq:3.6.15}, we need to prove that $f'=(-1)^{|u||v|}f$. We have
\[
\begin{array}{rcl}
\langle w', Y_W(Y(u, x_0)v, x_2)w \rangle & = & \langle w', Y_W((-1)^{|u||v|}e^{x_0 \mathcal{D}}Y(v, -x_0)u, x_2)w \rangle  \quad (\text{by } \eqref{eq:skewsym}) \\[5pt]
& = &(-1)^{|u||v|} \langle w', Y_W(Y(v, -x_0)u, x_2+x_0)w \rangle \\[5pt]
& = &(-1)^{|u||v|} (\iota_{10}f'(x_1, x_1-x_0))_{| x_1=x_2+x_0} \hfill (\text{by } \eqref{eq:3.6.18}) \\[5pt]
& = &(-1)^{|u||v|} \iota_{20}f'(x_0+x_2, x_2) \hfill (\text{by \cite[(3.3.40)]{Lepowsky-Li}}) \\[5pt]
& = &\iota_{20}f(x_0+x_2, x_2) \hfill (\text{by } \eqref{eq:3.6.16}).
\end{array}
\]
Hence $f'=(-1)^{|u||v|}f$, which concludes the proof.
\end{proof}

\begin{proposition}\label{prop:3.9.3}
Let $S$ be a subset of $V^{[*]}$, and write $\langle S \rangle_{\on{VA}}$ for the smallest dg vertex subalgebra of $V^{[*]}$ containing $S$. Then
\[
\langle S \rangle_{\on{VA}} = \on{Span}\{u^{(1)}_{n_1} \cdots u^{(r)}_{n_r} \mathbf{1} \ | \ r \in \mathbb{N}, u^{(i)} \in S \cup d_V(S), n_i \in \mathbb{Z} \}.
\]
\end{proposition}

\begin{proof}
The proof is identical to the classical setting (see \cite[Proposition 3.9.3]{Lepowsky-Li}). The condition $u^{(i)} \in S \cup d_V(S)$ garanties that the span is a subcomplex of $V^{[*]}$.

In case the set $S$ is not $d_V$-stable, then it can be replaced by $S \cup d_V(S)$, which is $d_V$-stable.
\end{proof}

With obvious changes, there is a version of Proposition \ref{prop:3.2.1} for modules:

\begin{proposition}\label{prop:4.2.1}
Let $W^{[*]}$ be a $V^{[*]}$-module. Then, for $u, v \in V^{[*]}$  homogeneous, there exists $k \in \mathbb{N}$ such that
\begin{align}\label{eq:4.2.1}
(x_1-x_2)^k[Y_W(u, x_1), Y_W(v, x_2)]^s=0
\end{align}
\end{proposition}

\begin{theorem}\label{thm:4.4.5}
Let $V^{[*]}$ be a dg vertex algebra, $W^{[*]}$ a cochain complex, and $Y_W(\cdot, x)$ a chain map from $V^{[*]}$ to $\on{End}^{[*]}(W^{[*]})[[x, x^{-1}]]$ such that $Y_W(\mathbf{1}, x)=\on{id}_W$, $Y_W(u, x)w \in W^{[*]}((x))$, and such that the weak associativity holds, i.e., for any $u, v \in V^{[*]}$ homogeneous, $w \in W^{[*]}$, there exists $k \in \mathbb{N}$ such that
\begin{align}\label{eq:4.3.1}
(x_0+x_2)^kY_W(Y(u, x_0)v, x_2)w = (x_0+x_2)^kY_W(u, x_0+x_2)Y_W(v, x_2)w.
\end{align}
Then $Y_W(\cdot, x)$ satisfies the Jacobi identity and $(W^{[*]}, Y_W(\cdot, x))$ is a $V^{[*]}$-module. In particular, in the definition of a dg module, the Jacobi identity can be replaced by the weak associativity \eqref{eq:4.3.1} for $ Y_W(\cdot, x)$.
\end{theorem}

\begin{proof}
It is a direct consequence of Theorem \ref{thm:3.6.3}.
\end{proof}

\begin{corollary}\label{cor:4.4.7} 
Let $V^{[*]}$ be a dg vertex algebra, $W^{[*]}$ a cochain complex, and $Y_W(\cdot, x)$ a chain map from $V^{[*]}$ to $\on{End}^{[*]}(W^{[*]})[[x, x^{-1}]]$ such that $Y_W(\mathbf{1}, x)=\on{id}_W$, $Y_W(u, x)w \in W^{[*]}((x))$, and such that the iterate formula
\begin{align}\label{eq:4.4.3}
\begin{array}{rcl}
Y_W(Y(u, x_0)v, x_2) & = &\displaystyle \on{Res}_{x_1}\big[x_0^{-1}\delta\big(\frac{x_1-x_2}{x_0}\big)Y_W(u, x_1)Y_W(v, x_2)\\[5pt]
& & \displaystyle -(-1)^{|u||v|}x_0^{-1}\delta\big(\frac{x_2-x_1}{-x_0}\big)Y_W(v, x_2)Y_W(u, x_1)\big]
 \end{array}
\end{align}
is satisfied for any homogeneous $u, v \in V^{[*]}$. Then $Y_W(\cdot, x)$ satisfies the Jacobi identity and $(W^{[*]}, Y_W(\cdot, x))$ is a $V^{[*]}$-module.
\end{corollary}

\begin{proof}
The reasoning is a direct consequence of Theorem \ref{thm:4.4.5}.
\end{proof}

\begin{proposition}\label{prop:4.5.14}
Let $W^{[*]}$ be a $V^{[*]}$-module. For a $d_V$-stable subset $S$ of $V^{[*]}$, the annihilator $\on{Ann}_W(S)=\{w \in W^{[*]} \ | \ Y_W(v, x)w=0 \text{ for } v \in S \}$ is a $V^{[*]}$-submodule of $W^{[*]}$. Furthermore,
\[
\on{Ann}_W(S)=\on{Ann}_W((S))
\]
where $(S)$ is the ideal of $V^{[*]}$ generated by $S$.
\end{proposition}

\begin{proof}
Because of the equality
\[
d_W(v_n w)=(d_V(v))_n w+(-1)^{|v|}v_n(d_W(w)),
\]
the condition that $S$ is $d_V$-stable implies that $\on{Ann}_W(S)$ is a subcomplex of $W^{[*]}$. The rest of the proof is identical to the classical setting (see \cite[Proposition 4.5.14]{Lepowsky-Li}). 
\end{proof}

\begin{definition}
Let $V^{[*]}$ be a dg vertex algebra. We say that $(W^{[*]}, Y_W(\cdot, x), \mathcal{D}_W)$ is a $V^{[*]}$-module (see \cite[Remark 4.1.4]{Lepowsky-Li}) if $(W^{[*]}, Y_W(\cdot, x))$ is a dg vertex algebra module for $V^{[*]}$ and $\mathcal{D}_W:W^{[*]} \longrightarrow W^{[*]}[2N]$ is a chain map such that, for all $v \in V^{[*]}$, we have
\[
[\mathcal{D}_W, Y_W(v, x)]^s=Y_W(\mathcal{D}(v), x).
\]
\end{definition}

\begin{definition}
Let $(W^{[*]}, Y_W(\cdot, x))$ be a $V^{[*]}$-module. An element $w \in W^{[*]}$ is called vacuum-like if
\[
Y(v, x)w \in W^{[*]}[[x]]
\]
for all $v \in V^{[*]}$.
\end{definition}

We do not make any assumption on the action of $\mathcal{D}_W$ or $d_W$ on $w$ in the definition of a vacuum-like vector.

\begin{proposition}\label{prop:4.7.4}
Let $(W^{[*]}, Y_W(\cdot, x), \mathcal{D}_W)$ be a $V^{[*]}$-module and let $w \in W^{[*]}$ be such that $\mathcal{D}_W(w)=0$. Then for all $v \in V^{[*]}$, we have
\[
Y_W(v, x)w=e^{x\mathcal{D}_W}v_{-1}w.
\]
In particular, $w$ is a vacuum-like vector.
\end{proposition}

\begin{proof}
The proof is identical to the classical setting (see \cite[Proposition 4.7.4]{Lepowsky-Li}).
\end{proof}

\begin{proposition}\label{prop:4.7.7}
Let $(W^{[*]}, Y_W(\cdot, x))$ be a $V^{[*]}$-module and let $w \in W^{[*]}$ be a vacuum-like vector with $d_W(w)=0$. Then the linear map
\[
\begin{array}{rccc}
f: & V^{[*]} & \longrightarrow & W[|w|]^{[*]} \\[5pt]
& v & \longmapsto & (-1)^{|w|}v_{-1}w
\end{array}
\]
is a $V^{[*]}$-module homomorphism.
\end{proposition}

\begin{proof}
Like in the classical setting (see \cite[Proposition 4.7.7]{Lepowsky-Li}), for any $u, v \in V^{[*]}$ we have
\[
f(Y(u, x_0)v)=Y(u, x_0)f(v).
\]
Then we verify that for $v$ homogeneous, we have $|v_{-1}w|=|v|+|w|$ so $v_{-1}w \in W^{[|v|+|w|]}=W[|w|]^{[|v|]}$. Hence $f$ is of degree $0$. Furthermore, $f(d_V(v))=(d_V(v))_{-1}w=d_W(v_{-1}w)-(-1)^{|v|}v_{-1}(d_W(w))=d_W(v_{-1}w)=d_{W[|w|]}(f(v))$. It follows that $f$ is a chain map.
\end{proof}

\begin{proposition}\label{prop:4.7.9}
Let $(W^{[*]}, Y_W(\cdot, x), \mathcal{D}_W)$ be a faithful $V^{[*]}$-module and let $w \in W^{[*]}$ be such that $w$ generates $W^{[*]}$ as a $V^{[*]}$-module and $\mathcal{D}_W(w)=d_W(w)=0$. Then the linear map
\[
\begin{array}{rccc}
f: & V^{[*]} & \longrightarrow & W[|w|]^{[*]} \\[5pt]
& v & \longmapsto & (-1)^{|w|}v_{-1}w
\end{array}
\]
is a $V^{[*]}$-module isomorphism.
\end{proposition}

\begin{proof}
Using Propositions \ref{prop:4.7.4} and \ref{prop:4.7.7}, $f$ is a $V^{[*]}$-module homomorphism. Since $w$ generates $W^{[*]}$, $f$ is surjective. Let $v \in V^{[*]}$ be homogeneous such that $f(v)=0$. By Proposition \ref{prop:4.7.4}, $Y_W(v, x)w=0$, and so $w \in \on{Ann}_W(\{v\})$. We also see that, for any $n \in \mathbb{Z}$,
\[
0=d_W(v_n w)=(d_V(v))_n w+(-1)^{|v|}v_n(d_W(w))=(d_V(v))_n w,
\]
hence $w \in \on{Ann}_W(\{v, d_V(v)\})$. By Proposition \ref{prop:4.5.14}, we see that $\on{Ann}_W(\{v, d_V(v)\})$ is a $V^{[*]}$-submodule of $W^{[*]}$. As it contains $w$ and $w$ generates $W^{[*]}$, we get
\[
W^{[*]}=\on{Ann}_W(\{v, d_V(v)\}).
\]
Thus $Y_W(v, x)=0$ on $W^{[*]}$ and since $W^{[*]}$ is faithful, we get $v=0$ and $f$ injective.
\end{proof}

\section{dg fields and weak dg vertex algebras}\label{sec:3}

The goal of this section is to explain how to construct a dg vertex algebra by extending known results on weak vertex operators to the dg setting. With Theorem \ref{thm:5.7.1}, we will obtain a way to verify if a space is a dg vertex algebra without having to check the Jacobi identity \eqref{eq:jacobi}. In what follows, $W^{[*]}$ is a cochain complex. We will give the reference for the classical version of each result.

\subsection{Definition of dg fields and their actions}
\begin{definition}
Consider the complex $\mathcal{E}^{[*]}(W^{[*]})= \on{Hom}^{[*]}\big(W^{[*]}, W^{[*]}((x))\big)$ with $|x|=-2N$, $N \in \mathbb{Z}$. A homogeneous dg field of degree $|a|$ on $W^{[*]}$ is a formal series
\[
a(x)=\sum_{n \in \mathbb{Z}}a_n x^{-n-1} \in \mathcal{E}^{[*]}(W^{[*]})
\]
such that for any $n \in \mathbb{Z}$ with $a_n \neq 0$ we have
\[
|a_n x^{-n-1}|=|a_n|+2N(n+1)=|a|.
\]
\end{definition}

\delete{
By setting $|x|=-2$, we have $|a_n x^{-n-1}|=|a_n|+2n+2$ if $a_n \neq 0$. We decompose $a(x)$ into a sum of homogeneous components of degree $|a^{(1)}|$, $\dots$, $|a^{(r)}|$. We set $a^{(i)}(x)=\sum_{n \in I_i}a_n x^{-n-1}$ for $I_i=\{n \in \mathbb{Z} \ | \ |a_n|+2n+2=|a^{(i)}|\}$. We set $(a^{(i)})_n=0$ if $n \in \mathbb{Z} \backslash I_i$. Then we can write $a^{(i)}(x)=\sum_{n \in \mathbb{Z}}(a^{(i)})_n x^{-n-1}$ and if $(a^{(i)})_n \neq 0$, we have $n \in I_i$ and $|(a^{(i)})_n x^{-n-1}|=|a_n x^{-n-1}|=|a_n|+2n+2=|a^{(i)}|$. Hence
\[
a^{(i)}(x) \in \on{End}^{[*]}(W^{[*]})[[x, x^{-1}]]^{[|a^{(i)}|]}
\]
and $a(x)=\sum_{i=1}^r a^{(i)}(x)$ is a decomposition into homogeneous dg fields. It follows that for any homogeneous $w \in W^{[*]}$, we have $a^{(i)}(x)w \in W^{[*]}((x))^{[|a^{(i)}|+|w|]}$, and so
\[
a^{(i)}(x) \in \on{Hom}^{[|a^{(i)}|]}\big(W^{[*]}, W^{[*]}((x))\big).
\]
Hence a homogeneous formal series $u(x)$ is a dg field if and only if $u(x) \in  \on{Hom}^{[|u|]}\big(W^{[*]}, W^{[*]}((x))\big)$.}

We will also need the endomorphism
\[
\mathcal{D}_{\mathcal{E}}=\frac{d}{dx}: \on{End}^{[*]}(W^{[*]})[[x, x^{-1}]] \longrightarrow \on{End}^{[*]}(W^{[*]})[[x, x^{-1}]].
\]
If $a(x)$ is a homogeneous dg field and $w \in W^{[*]}$ is homogeneous, then
\[
\mathcal{D}_{\mathcal{E}}(a(x))w=-\sum_{n \in \mathbb{Z}}n a_{n-1}w x^{-n-1} \in W^{[*]}((x)),
\]
hence $\mathcal{D}_{\mathcal{E}}(\mathcal{E}^{[*]}(W^{[*]})) \subseteq \mathcal{E}^{[*]}(W^{[*]})$. Moreover, we see that $|a_{n-1} x^{-n-1}|=|a_{n-1}|+2N(n+1)=|a|+2N$. Finally, it is also clear that $\mathcal{D}_{\mathcal{E}}$ commutes with the differential, as
\[
\begin{array}{rcl}
d_{\mathcal{E}}(a(x))(w) & = & d_{W((x))}(a(x)w)-(-1)^{|a(x)|}a(x)(d_W(w)) \\[5pt]
 & = & \displaystyle \sum_{n \in \mathbb{Z}}\big(d_W(a_nw)-(-1)^{|a(x)|}a_n(d_W(w))\big)x^{-n-1},
 \end{array}
\]
so
\[
\begin{array}{rcl}
\mathcal{D}_{\mathcal{E}}(d_{\mathcal{E}}(a(x)))(w) & = & \displaystyle \sum_{n \in \mathbb{Z}}\big(d_W(a_nw)-(-1)^{|a(x)|}a_n(d_W(w))\big)(-n-1)x^{-n-2} \\[5pt]
& = & \displaystyle \sum_{n \in \mathbb{Z}}\big(d_W(a_{n-1}w)-(-1)^{|a(x)|}a_{n-1}(d_W(w))\big)(-n)x^{-n-1} \\[5pt]
& = & d_{\mathcal{E}}(\mathcal{D}_{\mathcal{E}}a(x))(w).
 \end{array}
\]
It follows that
\[
\mathcal{D}_{\mathcal{E}} \in \on{End}^{[2N]}\left(\mathcal{E}^{[*]}(W^{[*]})\right).
\]
Let $W^{[*]}$ be a faithful dg module for a dg vertex algebra $V^{[*]}$. We have a map
\[
\begin{array}{rccl}
\iota_W: & V^{[*]} & \longrightarrow & \mathcal{E}^{[*]}(W^{[*]}) \\[5pt]
 & v & \longmapsto & Y_W(v, x).
 \end{array}
\]
The faithfulness means that $\iota_W(V^{[*]})$ has a dg vertex algebra structure with vacuum $\iota_W(\mathbf{1})=\on{id}_W$.
 We want the dg vertex algebra operator $Y_\mathcal{E}(\cdot, x_0)$ on $\iota_W(V^{[*]})$ to satisfy
 \[
 Y_\mathcal{E}(\iota_W(u), x_0)\iota_W(v)=\iota_W(Y_W(u, x_0)v),
 \]
which is equivalent to
\begin{align}\label{eq:5.2.4}
Y_\mathcal{E}(Y_W(u, x), x_0)Y_W(v, x)=Y_W(Y(u, x_0)v, x).
\end{align}
By taking $\on{Res}_{x_1}$ of the Jacobi identity for $Y_W$ and using \cite[(2.3.17)]{Lepowsky-Li}, we can rewrite the right hand side of \eqref{eq:5.2.4} and obtain
\begin{align}\label{eq:5.2.5} 
\begin{split}
Y_\mathcal{E}(Y_W(u, x), x_0)Y_W(v, x)= & \on{Res}_{x_1} \bigg[ x_0^{-1}  \delta\bigg(\frac{x_1-x}{x_0}\bigg) Y_W(u, x_1) Y_W(v, x) \\[5pt] 
 &  \quad \quad  -(-1)^{|u||v|}  x_0^{-1}  \delta\bigg(\frac{x-x_1}{-x_0}\bigg) Y_W(v, x) Y_W(u, x_1) \bigg]. 
 \end{split}
\end{align}

The above equation motivates the following definition:

\begin{definition}\label{def:3.2}
Let $(W^{[*]}, d_W^{[*]})$ be a complex. We define a linear map
\[
\begin{array}{rccl}
Y_\mathcal{E}(\cdot, x_0): & \mathcal{E}^{[*]}(W^{[*]})& \longrightarrow & \on{End}^{[*]}(\mathcal{E}^{[*]}(W^{[*]}))[[x_0, x_0^{-1}]] \\[5pt]
 & a(x) & \longmapsto & \displaystyle \sum_{n \in \mathbb{Z}}a(x)_n x_0^{-n-1}
 \end{array}
\]
by writing, for any homogeneous dg fields $a(x)$, $b(x)$ on $W^{[*]}$,
\begin{align}\label{eq:5.2.6} 
\begin{split}
Y_\mathcal{E}(a(x), x_0)b(x)= & \on{Res}_{x_1} \bigg[ x_0^{-1}  \delta\bigg(\frac{x_1-x}{x_0}\bigg) a(x_1) b(x) \\[5pt] 
 &  \quad \quad  -(-1)^{|a||b|}  x_0^{-1}  \delta\bigg(\frac{x-x_1}{-x_0}\bigg) b(x) a(x_1) \bigg]. 
 \end{split}
\end{align}
\end{definition}

\begin{proposition}[{\cite[Proposition 5.2.2]{Lepowsky-Li}}]   \label{prop:5.2.2}
Let $a(x), b(x)$ be homogeneous dg fields on $W^{[*]}$. Then $a(x)_n b(x)$ is again a homogeneous dg field on $W^{[*]}$ of degree $|a|+|b|-2N(n+1)$ given by
\begin{align}\label{eq:5.2.11}
a(x)_n b(x)=\on{Res}_{x_1} \big[ (x_1-x)^na(x_1) b(x) -(-1)^{|a||b|}(-x+x_1)^nb(x)a(x_1) \big].
\end{align}
\end{proposition}

\begin{proof}
Formula~\eqref{eq:5.2.11} is obtained by multiplying Formula~\eqref{eq:5.2.6} by $x_0^n$ and taking $\on{Res}_{x_0}$. In order to show that $a(x)_n b(x)$ is a homogeneous dg field on $W^{[*]}$, we apply it to a homogeneous element $w \in W^{[*]}$ and get
\[
a(x)_n b(x)w=\bigg(\sum_{i \geq 0}\binom{n}{i}(-x)^ia_{n-i}\bigg)b(x)w -(-1)^{|a||b|}b(x) \bigg(\sum_{i \geq 0}\binom{n}{i}(-x)^{n-i}a_{i}w\bigg).
\]
We know that $b(x)w \in W^{[*]}((x))^{[|w|+|b|]}$. Furthermore,
\[
|(-x)^ia_{n-i}(b(x)w)|  =  |a|+|b|+|w|-2N(n+1)=|b(x)(-x)^{n-i}a_{i}w|.
\]
Hence $|a(x)_nb(x)w|=|a|+|b|+|w|-2N(n+1)$. As $b(x)w \in W^{[*]}((x))$ and the sum $\sum_{i \geq 0}\binom{n}{i}(-x)^{n-i}a_{i}w$ is finite, we see that
\[
a(x)_nb(x) \in \mathcal{E}^{[|a|+|b|-2N(n+1)]}(W^{[*]}).
\]
\end{proof}

\begin{remark}\label{rem:5.2.5}
Suppose $W^{[*]}$ is a faithful $V^{[*]}$-module, and write $u(x)=Y_W(u, x)$ for $u \in V^{[*]}$. By Equation~\eqref{eq:5.2.4}, for any $a, b \in V^{[*]}$, we get $Y_\mathcal{E}(a(x), x_0)b(x)=Y_W(Y(a, x_0)b, x)$. Written component-wise, we obtain
\begin{align}\label{eq:5.2.14}
a(x)_nb(x)=Y_W(a_n b, x).
\end{align}
Furthermore, we have
\begin{align}\label{eq:5.2.15}
\mathbf{1}(x)=Y_W(\mathbf{1}, x)=\on{id}_W.
\end{align}
With this, we see that $\iota_W:V^{[*]} \longrightarrow \iota_W(V^{[*]})$ is a homomorphism of dg vertex algebras (we will prove that it is indeed a chain map after the proof of Proposition \ref{prop:5.3.9}). We can iterate Equation~\eqref{eq:5.2.14} and obtain
\begin{align}\label{eq:5.2.16}
Y_W(a^{(1)}_{n_1} \cdots a^{(r)}_{n_r} b, x)=a^{(1)}(x)_{n_1} \cdots a^{(r)}(x)_{n_r} b(x),
\end{align}
and by taking $b=\mathbf{1}$ we have
\begin{align}\label{eq:5.2.17}
Y_W(a^{(1)}_{n_1} \cdots a^{(r)}_{n_r} \mathbf{1}, x)=a^{(1)}(x)_{n_1} \cdots a^{(r)}(x)_{n_r} \on{id}_W.
\end{align}
\end{remark}

\subsection{The canonical weak dg vertex algebra}

\begin{proposition}
The map $Y_\mathcal{E}(\cdot, x_0):  \mathcal{E}^{[*]}(W^{[*]}) \longrightarrow  \on{End}^{[*]}(\mathcal{E}^{[*]}(W^{[*]}))[[x_0, x_0^{-1}]]$ given in Definition \ref{def:3.2} is a chain map.
\end{proposition}

\begin{proof}

We have seen that $a(x)_n b(x) \in \mathcal{E}^{[|a|+|b|-2N(n+1)]}(W^{[*]})$. We can give a cochain complex structure to $\on{End}^{[*]}(\mathcal{E}^{[*]}(W^{[*]}))[[x_0, x_0^{-1}]]$ by setting $|x_0|=-2N$. Then we see that $|a(x)_n b(x)x_0^{-n-1}|=|a|+|b|$. It follows that $Y_\mathcal{E}(\cdot, x_0):  \mathcal{E}^{[*]}(W^{[*]}) \longrightarrow  \on{End}^{[*]}(\mathcal{E}^{[*]}(W^{[*]}))[[x_0, x_0^{-1}]]$ is a map of degree $0$.

The differential $d_{\mathcal{E}}$ on $\mathcal{E}^{[*]}(W^{[*]}) \subset \on{End}^{[*]}(W^{[*]})[[x_, x^{-1}]]$ is the same as the one on $ \on{End}^{[*]}(W^{[*]})[[x_, x^{-1}]]$. We can verify that $d_{W((x))}(x^n)=0$ for any $n \in \mathbb{Z}$, and so
\[
\begin{array}{rl}
d_{\mathcal{E}}(a(x))(w) & =d_{W((x))}(a(x)w)-(-1)^{|a|}a(x)(d_W(w)) \\[5pt]
&=\displaystyle \sum_{n \in \mathbb{Z}}\big( d_W(a_n w)-(-1)^{|a|}a_n d_W(w) \big)x^{-n-1},
\end{array}
\]
which implies that $d_{\mathcal{E}}(a(x)) \in \mathcal{E}^{[|a|+1]}(W^{[*]})$. We then can determine that
\[
d(Y_\mathcal{E}(a(x), x_0))(b(x))=\sum_{n \in \mathbb{Z}}\bigg(d_\mathcal{E}(a(x)_nb(x))-(-1)^{|a|}a(x)_n d_\mathcal{E}(b(x))  \bigg)x_0^{-n-1}.
\]
By using Proposition \ref{prop:5.2.2} several times, we can show that
\begin{align}\label{eq:d_epsilon}
d_\mathcal{E}(a(x)_nb(x))=d_\mathcal{E}(a(x))_nb(x)+(-1)^{|a|}a(x)_n d_\mathcal{E}(b(x)),
\end{align}
and so
\[
d(Y_\mathcal{E}(a(x), x_0))(b(x))=\sum_{n \in \mathbb{Z}}d_\mathcal{E}(a(x))_nb(x)x_0^{-n-1}=Y_\mathcal{E}(d_\mathcal{E}(a(x)), x_0)b(x)
\]
for any $a(x), b(x) \in  \mathcal{E}^{[*]}(W^{[*]})$.
\end{proof}

\begin{definition}
A weak dg vertex algebra is a cochain complex $V^{[*]}$ with a chain map
\[
\begin{array}{rccl}
Y(\cdot, x): & V^{[*]} & \longrightarrow & \on{End}^{[*]}(V^{[*]})[[x, x^{-1}]] \\[5pt]
 & v & \longmapsto & \displaystyle \sum_{n \in \mathbb{Z}}v_n x^{-n-1}
 \end{array}
\]
with $|x|=-2N$, $N \in \mathbb{Z}$, and an element $\mathbf{1} \in V^{[0]}$ such that
\[
Y(\mathbf{1}, x)=\on{id}_V,
\]
\[
Y(v, x)\mathbf{1} \in V^{[*]}[[x]] \text{ and } \lim_{x \to 0} Y(v, x)\mathbf{1}=v,
\]
and such that
\[
[\mathcal{D}, Y(v, x)]^s=Y(\mathcal{D}v, x)=\frac{d}{dx}Y(v, x)
\]
where the chain map $\mathcal{D}:  V^{[*]}  \longrightarrow  V[2N]^{[*]}$ is defined by $\mathcal{D}(v)= v_{-2}\mathbf{1}$.
\end{definition}

A weak dg vertex algebra can be seen as a dg vertex algebra where there is no truncation condition (i) and no Jacobi identity (iv) (cf. Definition \ref{def:dgVA}).

For a weak dg vertex algebra, we have $Y(e^{x_0\mathcal{D}}v, x)=e^{x_0 \frac{d}{dx}}Y(v, x)=Y(v, x+x_0)$ by Taylor's theorem. We then see that for any $v \in V^{[*]}$,
\begin{align}\label{eq:5.3.6}
Y(v, x)\mathbf{1}=e^{x\mathcal{D}}v. 
\end{align}

A homomorphism of weak dg vertex algebras is a chain map $\psi: U^{[*]} \longrightarrow V^{[*]}$ such that $\psi(u_n v)=\psi(u)_n\psi(v)$ for any $u, v \in U^{[*]}$, $n \in \mathbb{Z}$, and such that $\psi(\mathbf{1}_U)=\mathbf{1}_V$. Similarly, we can define weak dg vertex subalgebras and weak dg vertex algebras generated by subsets. Given a subset $S$, we write $\langle S \rangle_{\on{WVA}}$ for the smallest weak dg vertex subalgebra containing $S$.

\begin{proposition}[{\cite[Proposition 5.3.9]{Lepowsky-Li}}] \label{prop:5.3.9}
The triplet $(\mathcal{E}^{[*]}(W^{[*]}), Y_\mathcal{E}(\cdot, x_0), \on{id}_W )$ is a weak dg vertex algebra. Moreover $\mathcal{D}_{\mathcal{E}}=\mathcal{D}$. It is called the canonical weak dg vertex algebra of the complex $(W^{[*]}, d_W^{[*]})$.
\end{proposition}

\begin{proof}
By Formula~\eqref{eq:5.2.6} and \cite[(2.3.17) and (2.3.18)]{Lepowsky-Li}, we have
\[
Y_\mathcal{E}(\on{id}_W, x_0)a(x)=a(x). 
\]
Furthermore, with the same strategy, we see that
\[
\begin{array}{rcl}
Y_\mathcal{E}(a(x), x_0)\on{id}_W & = & \on{Res}_{x_1}\big[x_1^{-1}\delta\big(\frac{x+x_0}{x_1})a(x_1) \big) \big] \\[5pt]
&=& \on{Res}_{x_1}\big[x_1^{-1}\delta\big(\frac{x+x_0}{x_1})a(x+x_0) \big) \big] \quad \quad\text{by \cite[(2.3.31)]{Lepowsky-Li}}\\[5pt]
&=&a(x+x_0)\\[5pt]
&=&e^{x_0\frac{d}{dx}}a(x)\\[5pt]
&=&e^{x_0\mathcal{D}_{\mathcal{E}}}a(x).
\end{array}
\]
It follows that $\displaystyle \lim_{x_0 \to 0} Y_\mathcal{E}(a(x), x_0)\on{id}_W=a(x)$ and $a(x)_{-2}\on{id}_W=\mathcal{D}_{\mathcal{E}}a(x)$. Hence $\mathcal{D}_{\mathcal{E}}=\mathcal{D}$.

For the last two equalities, we have
\[
\begin{array}{l}
\frac{\partial}{\partial x_0}Y_\mathcal{E}(a(x), x_0)b(x) \\[5pt]
= \on{Res}_{x_1}\big[\frac{\partial}{\partial x_0}\big(x_0^{-1}\delta\big(\frac{x_1-x}{x_0}\big)\big)a(x_1)b(x) -(-1)^{|a||b|}\frac{\partial}{\partial x_0}\big(x_0^{-1}\delta\big(\frac{x-x_1}{-x_0}\big)\big)b(x)a(x_1)\big] \\[5pt]
= \on{Res}_{x_1}\big[-\frac{\partial}{\partial x_1}\big(x_0^{-1}\delta\big(\frac{x_1-x}{x_0}\big)\big)a(x_1)b(x) +(-1)^{|a||b|}\frac{\partial}{\partial x_1}\big(x_0^{-1}\delta\big(\frac{x-x_1}{-x_0}\big)\big)b(x)a(x_1)\big] \\
\quad  \text{by \cite[(2.3.20)]{Lepowsky-Li}} \\[5pt]
=\on{Res}_{x_1}\big[x_0^{-1}\delta\big(\frac{x_1-x}{x_0}\big)\frac{\partial}{\partial x_1}(a(x_1))b(x) -(-1)^{|a||b|}x_0^{-1}\delta\big(\frac{x-x_1}{-x_0}\big)b(x)\frac{\partial}{\partial x_1}(a(x_1))\big] \\[5pt]
= Y_\mathcal{E}(\mathcal{D}_{\mathcal{E}}a(x), x_0)b(x).
\end{array}
\]
Furthermore, 
\[
\begin{array}{l}
[\mathcal{D}_{\mathcal{E}}, Y_\mathcal{E}(a(x), x_0)]^sb(x)  \\[5pt]
 = \frac{\partial}{\partial x} \big(Y_\mathcal{E}(a(x), x_0)b(x)\big)-Y_\mathcal{E}(a(x), x_0)\frac{d}{dx}b(x) \\[5pt]
= \frac{\partial}{\partial x}\bigg(\on{Res}_{x_1}\big[x_0^{-1}\delta\big(\frac{x_1-x}{x_0}\big)a(x_1)b(x) -(-1)^{|a||b|}x_0^{-1}\delta\big(\frac{x-x_1}{-x_0}\big)b(x)a(x_1)\big]\bigg) \\[5pt]
\quad -\on{Res}_{x_1}\big[x_0^{-1}\delta\big(\frac{x_1-x}{x_0}\big)a(x_1)\frac{d}{dx}b(x) -(-1)^{|a|(|b|+2)}x_0^{-1}\delta\big(\frac{x-x_1}{-x_0}\big)(\frac{d}{dx}b(x))a(x_1)\big] \\[5pt]
= \on{Res}_{x_1}\big[\frac{\partial}{\partial x} \big(x_0^{-1}\delta\big(\frac{x_1-x}{x_0}\big)\big)a(x_1)b(x) -(-1)^{|a||b|}\frac{\partial}{\partial x} \big(x_0^{-1}\delta\big(\frac{x-x_1}{-x_0}\big)\big)b(x)a(x_1)\big] \\[5pt]
=\on{Res}_{x_1}\big[\frac{\partial}{\partial x_0} \big(x_0^{-1}\delta\big(\frac{x_1-x}{x_0}\big)\big)a(x_1)b(x) -(-1)^{|a||b|}\frac{\partial}{\partial x_0} \big(x_0^{-1}\delta\big(\frac{x-x_1}{-x_0}\big)\big)b(x)a(x_1)\big] \\[5pt]
=\frac{\partial}{\partial x_0} Y_\mathcal{E}(a(x), x_0) b(x).
\end{array}
\]
\end{proof}

We now introduce the notion of dg representation of a dg vertex algebra (cf. \cite[Definition 5.3.14]{Lepowsky-Li}).
\begin{definition}
A dg representation of a weak dg vertex algebra $V^{[*]}$ on a cochain complex $W^{[*]}$ is a weak dg vertex algebra homomorphism $V^{[*]} \longrightarrow \mathcal{E}^{[*]}(W^{[*]})$.
\end{definition}

\begin{theorem}[{\cite[Theorem 5.3.15]{Lepowsky-Li}}] 
Let $V^{[*]}$ be a dg vertex algebra and $W^{[*]}$ a cochain complex. Then
\smallbreak
\begin{center}
$W^{[*]}$ is a dg module for $V^{[*]}$ $\Longleftrightarrow$ there is a representation $V^{[*]} \longrightarrow \mathcal{E}^{[*]}(W^{[*]})$.
\end{center}
\end{theorem}

\begin{proof}
Let $W^{[*]}$ be a dg module for the dg vertex algebra $V^{[*]}$. We define
\[
\begin{array}{rccl}
\iota_W: & V^{[*]} & \longrightarrow & \mathcal{E}^{[*]}(W^{[*]}) \\[5pt]
 & v & \longmapsto & Y_W(v, x).
 \end{array}
\]
with $\iota_W(\mathbf{1})=\on{id}_W$. Then
\[
Y_\mathcal{E}(\iota_W(u), x_0)\iota_W(v)=Y_\mathcal{E}(Y_W(u, x), x_0)Y_W(v, x)=\iota_W(Y_W(u, x_0)v)
\]
by Formulas~\eqref{eq:5.2.4} and \eqref{eq:5.2.6}. We know that by definition $Y_W(\cdot, x)$ is a map of degree $0$. We also have
\[
\begin{array}{rcl}
d_\mathcal{E}(\iota_W(u))(w) & = & d_W Y_W(u, x)(w)-(-1)^{|u|}Y_W(u, x)d_W(w) \\[5pt]
& =& \displaystyle \sum_{n \in \mathbb{Z}}\big(d_W(u_n w)-(-1)^{|u|}u_n(d_Ww) \big)x^{-n-1},
\end{array}
\]
and
\[
\iota_W(d_V u)(w)=\sum_{n \in \mathbb{Z}}(d_Vu)_n w x^{-n-1}.
\]
But as $V^{[*]}$ is a dg vertex algebra, we know that $d_W(u_n w)=(d_Vu)_n w+(-1)^{|u|}u_n d_Ww$ (cf. \cite{Caradot-Jiang-Lin-4}), hence $d_\mathcal{E}(\iota_W(u))=\iota_W(d_V u)$ and $\iota_W$ is a chain map. We can check that it is in fact a homomorphism of weak dg vertex algebras.

Conversely, let $\phi_x: V^{[*]} \longrightarrow \mathcal{E}^{[*]}(W^{[*]}) $ be a homomorphism of weak dg vertex algebras and define $Y_W(v, x)=\phi_x(v)$ for any $v \in V^{[*]}$. Then $Y_W(v, x)w \in W^{[*]}((x))$ for any $v \in V^{[*]}$, $w \in W^{[*]}$. Furthermore,
\[
Y_W(\mathbf{1}, x)=\phi_x(\mathbf{1})=\on{id}_W
\]
and
\[
\begin{array}{l}
Y_W(Y(u, x_0)v, x_2) \\[5pt]
 =  \phi_{x_2}(Y(u, x_0)v) \\[5pt]
 =  Y_\mathcal{E}(\phi_{x_2}(u), x_0)\phi_{x_2}(v) \\[5pt]
 =  \on{Res}_{x_1}\big[x_0^{-1}\delta\big(\frac{x_1-x_2}{x_0}\big)\phi_{x_1}(u)\phi_{x_2}(v) -(-1)^{|\phi_{x_1}(u)||\phi_{x_2}(v)|}x_0^{-1}\delta\big(\frac{x-x_1}{-x_0}\big)\phi_{x_2}(v)\phi_{x_1}(u)\big] \\[5pt]
 = \on{Res}_{x_1}\big[x_0^{-1}\delta\big(\frac{x_1-x_2}{x_0}\big)Y_W(u, x_1)Y_W(v, x_2) -(-1)^{|u||v|}x_0^{-1}\delta\big(\frac{x_2-x_1}{-x_0}\big)Y_W(v, x_2)Y_W(u, x_1)\big].
 \end{array}
\]
It follows that $(W^{[*]}, Y_W(\cdot, x))$ is a $V^{[*]}$-module by Corollary \ref{cor:4.4.7}.
\end{proof}

For a chain map $\mathcal{D}_W: W^{[*]} \longrightarrow W^{[*]}[2N]$, we write $\mathcal{E}^{[*]}(W^{[*]}, \mathcal{D}_W)=\{a(x) \in \mathcal{E}^{[*]}(W^{[*]}) \ | \ [\mathcal{D}_W, a(x)]^s=\frac{d}{dx}a(x)  \}$.

\begin{proposition}[{\cite[Proposition 5.4.1]{Lepowsky-Li}}] \label{prop:5.4.1}
Let $\mathcal{D}_W: W^{[*]} \longrightarrow W^{[*]}[2N]$ be a chain map. Then $\mathcal{E}^{[*]}(W^{[*]}, \mathcal{D}_W)$ is a weak dg vertex subalgebra of $\mathcal{E}^{[*]}(W^{[*]})$.
\end{proposition}

\begin{proof}
Since $\on{id}_W \in \mathcal{E}^{[*]}(W^{[*]}, \mathcal{D}_W)$, we must show that
\[
 [\mathcal{D}_W, Y_\mathcal{E}(a(x), x_0)b(x)]^s=\frac{\partial}{\partial x}\bigg(  Y_\mathcal{E}(a(x), x_0)b(x) \bigg)
\]
for $a(x), b(x) \in \mathcal{E}^{[*]}(W^{[*]}, \mathcal{D}_W)$ in order to show that $a(x)_n b(x) \in \mathcal{E}^{[*]}(W^{[*]}, \mathcal{D}_W)$ for any $n \in \mathbb{Z}$. The rest of the proof is identical to the classical case (see \cite[Proposition 5.4.1]{Lepowsky-Li}).
\end{proof}

We immediately obtain:

\begin{theorem}[{\cite[Theorem 5.4.2]{Lepowsky-Li}}] \label{thm:5.4.2}
Let $V^{[*]}$ be a dg vertex algebra and $W^{[*]}$ a cochain complex equipped with a chain map $\mathcal{D}_W: W^{[*]} \longrightarrow W^{[*]}[2N]$. Then giving $W^{[*]}$ a structure of $V^{[*]}$-module $(W^{[*]}, Y_W(\cdot, x), \mathcal{D}_W)$ compatible with $\mathcal{D}_W$ is equivalent to giving a weak dg vertex algebra homomorphism from $V^{[*]}$ to $\mathcal{E}^{[*]}(W^{[*]}, \mathcal{D}_W)$.
\end{theorem}

\subsection{dg local subalgebras and dg vertex subalgebras of $\mathcal{E}^{[*]}(W^{[*]})$}

\begin{definition}
\begin{itemize}[leftmargin=20pt]
\item[(i)] Two dg fields $a(x)$ and $b(x)$ are mutually dg local if there exists $k \in \mathbb{N}$ such that:
\[
(x_1-x_2)^k[a(x_1), b(x_2)]^s=0.
\]

\item[(ii)] A dg field $a(x)$ is a dg vertex operator if it is dg local with itself, i.e., there exists $k \in \mathbb{N}$ such that:
\[
(x_1-x_2)^k[a(x_1), a(x_2)]^s=0.
\]
\item[(iii)] A subcomplex $A^{[*]}$ of $\mathcal{E}^{[*]}(W^{[*]})$ is dg local if any two homogeneous $a(x), b(x) \in A^{[*]}$ are mutually dg local.
\item[(iv)]  A weak dg vertex subalgebra of $\mathcal{E}^{[*]}(W^{[*]})$ is a dg local subalgebra if it is a dg local subcomplex.
\end{itemize}
\end{definition}

We say that $V^{[*]}$ is a dg vertex subalgebra of $\mathcal{E}^{[*]}(W^{[*]})$ if $(V^{[*]}, Y_V(\cdot, x_0), \on{id}_W)$ is a weak dg vertex subalgebra of $(\mathcal{E}^{[*]}(W^{[*]}), Y_\mathcal{E}(\cdot, x_0), \on{id}_W)$, i.e.,  $Y_V(a(x), x_0)=Y_\mathcal{E}(a(x), x_0)$ for any $a(x) \in V^{[*]}$, and $V^{[*]}$ is also a dg vertex algebra.

\begin{theorem}[{\cite[Theorem 5.5.11]{Lepowsky-Li}}] \label{thm:5.5.11}
Let $V^{[*]}$ be a dg vertex subalgebra of $\mathcal{E}^{[*]}(W^{[*]})$. Then $V^{[*]}$ is a dg local subalgebra. Moreover, $W^{[*]}$ is a module for $V^{[*]}$ viewed as a dg vertex algebra with
\[
Y_W(a(x), x_0)=a(x_0)
\]
for $a(x) \in V^{[*]}$. In particular, $W^{[*]}$ is a faithful $V^{[*]}$-module.
\end{theorem}

\begin{proof}
Since $V^{[*]} \subseteq \mathcal{E}^{[*]}(W^{[*]})$, we can use Formula~\eqref{eq:5.2.6}. Hence for any $w \in W^{[*]}$ and any homogeneous $a(x), b(x) \in V^{[*]}$, we have
\begin{align*}
\begin{split}
Y_\mathcal{E}(a(x), x_0)b(x)w= &  \on{Res}_{x_1} \bigg[ x_0^{-1}  \delta\bigg(\frac{x_1-x}{x_0}\bigg) a(x_1) b(x)w \\[5pt] 
 &  \quad \quad  -(-1)^{|a||b|}  x_0^{-1}  \delta\bigg(\frac{x-x_1}{-x_0}\bigg) b(x) a(x_1)w \bigg]. 
 \end{split}
\end{align*}
 By Corollary \ref{cor:4.4.7}, we know that $W^{[*]}$ is a dg module for $V^{[*]}$ with
 \[
Y_W(a(x), x_0)=a(x_0).
\]
In particular, Proposition \ref{prop:4.2.1} states that the weak commutativity holds for $Y_W$. Hence any homogeneous $a(x), b(x) \in V^{[*]}$ are mutually dg local. Therefore $V^{[*]}$ is a dg local subalgebra.
\end{proof}

Theorem \ref{thm:5.5.11} establishes the first half of a correspondance between dg vertex subalgebras of $\mathcal{E}^{[*]}(W^{[*]})$ with $W^{[*]}$ as a faithful module and dg local subalgebras of $\mathcal{E}^{[*]}(W^{[*]})$. The second half is obtained in Theorem \ref{thm:5.5.14}.

\begin{proposition}[{\cite[Proposition 5.5.12]{Lepowsky-Li}}] \label{prop:5.5.12}
Let $a(x), b(x) \in \mathcal{E}^{[*]}(W^{[*]})$ be homogeneous and mutually dg local, and set $k \in \mathbb{N}$ such that
\begin{align}\label{eq:5.5.9}
(x_1-x_2)^k[a(x_1), b(x_2)]^s=0
\end{align}
on $W^{[*]}$. Then
\[
a(x)_n b(x)=0 \text{ for } n \geq k.
\]
In particular, if $V^{[*]}$ is a dg local subalgebra of $\mathcal{E}^{[*]}(W^{[*]})$, then for any homogeneous $a(x), b(x) \in V^{[*]}$, we have
\[
a(x)_n b(x)=0 \text{ for } n \gg 0.
\]
\end{proposition}

\begin{proof}
It is clear from Equations~\eqref{eq:5.2.11} and \eqref{eq:5.5.9}.
\end{proof}

\begin{proposition}[{\cite[Proposition 5.5.13]{Lepowsky-Li}}] \label{prop:5.5.13}
Let $a(x), b(x) \in \mathcal{E}^{[*]}(W^{[*]})$ be homogeneous and mutually dg local, and set $k \in \mathbb{N}$ such that
\begin{align}\label{eq:5.5.11}
(x_1-x_2)^k[a(x_1), b(x_2)]^s=0
\end{align}
on $W^{[*]}$. Then
\begin{align}\label{eq:5.5.12}
(x_1-x_2)^k[Y_\mathcal{E}(a(x), x_1), Y_\mathcal{E}(b(x), x_2)]^s=0
\end{align}
on $\mathcal{E}^{[*]}(W^{[*]})$. In particular, if $V^{[*]}$ is a dg local subalgebra of $\mathcal{E}^{[*]}(W^{[*]})$, then the weak commutativity holds for $Y_\mathcal{E}(\cdot, x_0)$ on $V^{[*]}$.
\end{proposition}

\begin{proof}
Set $c(x) \in \mathcal{E}^{[|c|]}(W^{[*]})$. Using Equation~\eqref{eq:5.2.6}, we can show that
\[
\begin{array}{l}
Y_\mathcal{E}(a(x), x_1) Y_\mathcal{E}(b(x), x_2)c(x)  \\[5pt]
 = \displaystyle \on{Res}_{x_3}\on{Res}_{x_4} \bigg[ x_1^{-1}  \delta\bigg(\frac{x_3-x}{x_1}\bigg)  x_2^{-1}  \delta\bigg(\frac{x_4-x}{x_2}\bigg)a(x_3)b(x_4)c(x)  \bigg] \\[5pt]
\quad \displaystyle  -(-1)^{|b||c|}\on{Res}_{x_3}\on{Res}_{x_4} \bigg[ x_1^{-1}  \delta\bigg(\frac{x_3-x}{x_1}\bigg)  x_2^{-1}  \delta\bigg(\frac{x-x_4}{-x_2}\bigg)a(x_3)c(x)b(x_4)  \bigg] \\[5pt]
\quad \displaystyle  -(-1)^{|a|(|b|+|c|)}\on{Res}_{x_3}\on{Res}_{x_4} \bigg[ x_1^{-1}  \delta\bigg(\frac{x-x_3}{-x_1}\bigg)  x_2^{-1}  \delta\bigg(\frac{x_4-x}{x_2}\bigg)b(x_4)c(x)a(x_3)  \bigg] \\[5pt]
\quad \displaystyle  +(-1)^{|a|(|b|+|c|)}(-1)^{|b||c|}\on{Res}_{x_3}\on{Res}_{x_4} \bigg[ x_1^{-1}  \delta\bigg(\frac{x-x_3}{-x_1}\bigg)  x_2^{-1}  \delta\bigg(\frac{x-x_4}{-x_2}\bigg)c(x)b(x_4)a(x_3)  \bigg].
\end{array}
\]
We can do the same for $Y_\mathcal{E}(b(x), x_2) Y_\mathcal{E}(a(x), x_1)c(x)$. From \cite[(2.3.56)]{Lepowsky-Li}, it can be proved that
\begin{align}\label{eq:5.5.14'}
\begin{array}{l}
\displaystyle (x_1-x_2)^k x_1^{-1}  \delta\bigg(\frac{x_3-x}{x_1}\bigg)  x_2^{-1}  \delta\bigg(\frac{x_4-x}{x_2}\bigg) \\[5pt]
\displaystyle \quad =(x_3-x_4)^k x_1^{-1}  \delta\bigg(\frac{x_3-x}{x_1}\bigg)  x_2^{-1}  \delta\bigg(\frac{x_4-x}{x_2}\bigg).
\end{array}
\end{align}
By applying Equations~\eqref{eq:5.5.11} and \eqref{eq:5.5.14'} to the formulas we just computed, we obtain Equation \eqref{eq:5.5.12}.
\end{proof}

\begin{theorem}[{\cite[Theorem 5.5.14]{Lepowsky-Li}}] \label{thm:5.5.14}
Any dg local subalgebra $V^{[*]}$ of $\mathcal{E}^{[*]}(W^{[*]})$ is a dg vertex algebra and $W^{[*]}$ is a faithful dg module where $Y_W(a(x), x_0)=a(x_0)$ for $a(x) \in V^{[*]}$. In particular, the dg local subalgebras of $\mathcal{E}^{[*]}(W^{[*]})$ are precisely the dg vertex algebras.
\end{theorem}

\begin{proof}
By Proposition \ref{prop:5.5.13}, we know that $Y_\mathcal{E}(\cdot, x_0)$ satisfies the weak commutativity on $V^{[*]}$. By Propositions \ref{prop:5.3.9} and \ref{prop:5.5.12}, it also satisfies the $\mathcal{D}$-bracket derivative formulas, as well as the other axioms of a dg vertex algebra, minus the Jacobi identity. Then Theorem \ref{thm:3.5.1} implies the $V^{[*]}$ is a dg vertex algebra. Finally, the rest of the statement is obtained from Theorem \ref{thm:5.5.11}.
\end{proof}

There exists a dg version of Dong's lemma:

\begin{proposition}[{\cite[Proposition 5.5.15]{Lepowsky-Li}}] \label{prop:5.5.15}
Let $a(x), b(x), c(x)$ be homogeneous and pairwise mutually dg local dg fields on $W^{[*]}$. Then $a(x)_n b(x)$ and $c(x)$ are mutually dg local for any $n \in \mathbb{Z}$.
\end{proposition}

\begin{proof}
Set $n \in \mathbb{Z}$ and consider $r \in \mathbb{N}$ sucht that $r \geq -n$ and
\[
\left\{\begin{array}{l}
(x_1-x_2)^r[a(x_1), b(x_2)]^s=0, \\[5pt]
(x_1-x_2)^r[a(x_1), c(x_2)]^s=0, \\[5pt]
(x_1-x_2)^r[b(x_1), c(x_2)]^s=0.
\end{array}\right.
\]
From Equation~\eqref{eq:5.2.11}, we know that
\[
a(x)_n b(x)=\on{Res}_{x_1} \big[ (x_1-x)^na(x_1) b(x) -(-1)^{|a||b|}(-x+x_1)^nb(x)a(x_1) \big].
\]
Since $(x-x_2)^{4r}=((x-x_1)+(x_1-x_2))^{3r}(x-x_2)^r$, we have
\[
\begin{array}{l}
(x-x_2)^{4r} \big( (x_1-x)^na(x_1) b(x)c(x_2) -(-1)^{|a||b|}(-x+x_1)^nb(x)a(x_1)c(x_2) \big) \\[5pt]
= \displaystyle \sum_{s=0}^{3r}\binom{3r}{s}(x-x_1)^{3r-s}(x_1-x_2)^{s}(x-x_2)^r \\
\quad \times \bigg( (x_1-x)^na(x_1) b(x)c(x_2) -(-1)^{|a||b|}(-x+x_1)^nb(x)a(x_1)c(x_2) \bigg) \\[5pt]
= \displaystyle \sum_{s=r+1}^{3r}\binom{3r}{s}(x-x_1)^{3r-s}(x_1-x_2)^{s}(x-x_2)^r \\
\quad \times \bigg( (x_1-x)^na(x_1) b(x)c(x_2) -(-1)^{|a||b|}(-x+x_1)^nb(x)a(x_1)c(x_2) \bigg)
\end{array}
\]
because for $0 \leq s \leq r$, we have $3r-s+n \geq r$ and so
\[
(x_1-x)^{3r-s+n}[a(x_1), b(x)]^s=0.
\]
Then by using the assumption on $r$, we get
\[
\begin{array}{l}
(x-x_2)^{4r} \big( (x_1-x)^na(x_1) b(x)c(x_2) -(-1)^{|a||b|}(-x+x_1)^nb(x)a(x_1)c(x_2) \big) \\[5pt]
= \displaystyle \sum_{s=r+1}^{3r}\binom{3r}{s}(x-x_1)^{3r-s}(x_1-x_2)^{s}(x-x_2)^r \bigg( (-1)^{(|a|+|b|)|c|}(x_1-x)^nc(x_2) a(x_1) b(x) \\[5pt]
\hfill -(-1)^{|a||b|}(-1)^{|b||c|}(-1)^{|a||c|}(-x+x_1)^nc(x_2) b(x)a(x_1) \bigg) \\[5pt]
= (-1)^{(|a|+|b|)|c|}(x-x_2)^{4r} \Big( (x_1-x)^nc(x_2)a(x_1) b(x) \\[5pt]
\hfill -(-1)^{|a||b|}(-x+x_1)^nc(x_2)b(x)a(x_1) \Big).
\end{array}
\]
We then take $\on{Res}_{x_1}$ on both sides and obtain
\[
(x-x_2)^{4r} (a(x)_n b(x))c(x_2)=(-1)^{(|a|+|b|)|c|}(x-x_2)^{4r} c(x_2)(a(x)_n b(x)).
\]
But we know that $|a(x)_n b(x)|=|a|+|b|-2N(n+1)$ so $(-1)^{(|a(x)_n b(x)|)|c(x)|}=(-1)^{(|a|+|b|)|c|}$. Thus
\[
(x-x_2)^{4r}[a(x)_n b(x), c(x_2)]^s=0.
\]
\end{proof}

\begin{theorem}[{\cite[Theorem 5.5.17]{Lepowsky-Li}}] \label{thm:5.5.17}
Any maximal dg local subcomplex of $\mathcal{E}^{[*]}(W^{[*]})$ is a dg vertex algebra with $W^{[*]}$ as a faithful dg module.
\end{theorem}

\begin{proof}
Let $A^{[*]}$ be such a subcomplex. Since $\on{id}_W$ is dg local with any dg field on $W^{[*]}$ and $d_\mathcal{E}(\on{id}_W)=0$,  it follows that $A+\cc\on{id}_W$ is a dg local subcomplex of $\mathcal{E}^{[*]}(W^{[*]})$. As $A^{[*]}$ is maximal, it follows that $\on{id}_W \in A^{[*]}$.

Set $a(x), b(x) \in A^{[*]}$ homogeneous. By Proposition \ref{prop:5.5.15}, for any $n \in \mathbb{Z}$, we know that $a(x)_n b(x)$ is dg local with any dg vertex operator in $A^{[*]}$. In particular, $a(x)_n b(x)$ is dg local with $a(x)$ and $b(x)$, so it is dg local with itself. Thus the subcomplex generated by $A^{[*]}$ and $a(x)_n b(x)$ is dg local. Furthermore, Equation~\eqref{eq:d_epsilon} states that
\[
d_\mathcal{E}(a(x)_n b(x))=d_\mathcal{E}(a(x))_nb(x)+(-1)^{|a|}a(x)_n d_\mathcal{E}(b(x)).
\]
As $A^{[*]}$ is a complex, $d_\mathcal{E}(a(x))$ and $d_\mathcal{E}(b(x))$ are also in $A^{[*]}$. With the same reasoning as above, we see that $d_\mathcal{E}(a(x))_n b(x)$ is dg local with $A^{[*]}$ and itself. Hence $d_\mathcal{E}(a(x))_n b(x)$ is dg local with $a(x)$ and $d_\mathcal{E}(b(x))$, so it is dg local with $a(x)_n d_\mathcal{E}(b(x))$. Thus $d_\mathcal{E}(a(x))_n b(x)$ and $a(x)_n d_\mathcal{E}(b(x))$ are mutually dg local dg vertex operators of degree $|a|+|b|+1-2N(n+1)$. It follows that $d_\mathcal{E}(a(x)_n b(x))$ is a dg vertex operator and it is dg local with $A^{[*]}$. Therefore it is dg local with $a(x)_n b(x)$. This means that the subcomplex generated by $A^{[*]}$ and $a(x)_n b(x)$ is dg local. As $A^{[*]}$ is maximal, we get that $a(x)_n b(x) \in A^{[*]}$.

Because of the previous paragraph, $A^{[*]}$ is a weak dg vertex algebra. So it is a dg local subalgebra of $\mathcal{E}^{[*]}(W^{[*]})$, and by Theorem \ref{thm:5.5.14}, it is then a dg vertex algebra with $W^{[*]}$ as a dg module.
\end{proof} 

\begin{theorem}[{\cite[Theorem 5.5.18]{Lepowsky-Li}}] \label{thm:5.5.18}
Let $S$ be a $d_\mathcal{E}$-stable set of pairwise mutually dg local dg vertex operators on $W^{[*]}$. Then $S$ can be embedded in a dg vertex subalgebra of $\mathcal{E}^{[*]}(W^{[*]})$, and in fact, the weak dg vertex algebra $\langle S \rangle_{\on{WVA}}$ generated by $S$ is a dg vertex algebra with $W^{[*]}$ as a natural faithful module. Furthermore, $\langle S \rangle_{\on{WVA}}$ is the linear span of the elements of the form
\[
a^{(1)}(x)_{n_1} \cdots a^{(r)}(x)_{n_r} \on{id}_W
\]
for $a^{(i)}(x) \in S$, $n_i \in \mathbb{Z}$ and $r \geq 0$.
\end{theorem}

\begin{proof}
By Zorn's lemma, there exists a maximal dg local subcomplex $V^{[*]}$ containing $S$. By Theorem \ref{thm:5.5.17}, we know that $V^{[*]}$ is a dg vertex algebra with $W^{[*]}$ as a natural module. Since $\langle S \rangle_{\on{WVA}}$ is a weak dg vertex subalgebra of $V^{[*]}$, it is necessarily a dg vertex algebra, i.e. $\langle S \rangle_{\on{WVA}}=\langle S \rangle_{\on{VA}}$,  and $W^{[*]}$ is an $\langle S \rangle_{\on{WVA}}$-module. The rest is a consequence of Proposition \ref{prop:3.9.3}.
\end{proof}

\subsection{Construction of dg vertex algebras and their modules}

\begin{theorem}[{\cite[Theorem 5.6.1]{Lepowsky-Li}}] \label{thm:5.6.1}
Let $S$ be a $d_\mathcal{E}$-stable dg local subset of the weak dg vertex algebra $\mathcal{E}^{[*]}(W^{[*]}, \mathcal{D}_W)$. Then the weak dg vertex subalgebra $\langle S \rangle_{\on{WVA}}$ of $\mathcal{E}^{[*]}(W^{[*]})$ generated by $S$ is a dg vertex subalgebra of $\mathcal{E}^{[*]}(W^{[*]}, \mathcal{D}_W)$, and the tuple $(W^{[*]}, Y_W(\cdot, x_0), \mathcal{D}_W)$ is a faithful module for $\langle S \rangle_{\on{WVA}}$ with $Y_W(a(x), x_0)=a(x_0)$ for $a(x) \in \langle S \rangle_{\on{WVA}}$. In particular,
\[
Y_W(\mathcal{D}a(x), x_0)=[\mathcal{D}_W, Y_W(a(x), x_0)]^s \text{ for all } a(x) \in \langle S \rangle_{\on{WVA}}.
\]
\end{theorem}

\begin{proof}
By Theorem \ref{thm:5.5.18}, we know that $\langle S \rangle_{\on{WVA}}=\langle S \rangle_{\on{VA}}$ is a dg vertex subalgebra of $\mathcal{E}^{[*]}(W^{[*]})$ and $W^{[*]}$ is a faithful module for $\langle S \rangle_{\on{VA}}$. Since $S \subseteq \mathcal{E}^{[*]}(W^{[*]}, \mathcal{D}_W)$ and $\mathcal{E}^{[*]}(W^{[*]}, \mathcal{D}_W)$ is a weak dg vertex subalgebra of $\mathcal{E}^{[*]}(W^{[*]})$ by Proposition \ref{prop:5.4.1}, it follows that $\langle S \rangle_{\on{VA}}$ is a dg vertex subalgebra of $\mathcal{E}^{[*]}(W^{[*]}, \mathcal{D}_W)$. The rest follows from Theorem \ref{thm:5.4.2}.
\end{proof}

In order to state the next result, we first need to introduce some notations. Let $T$ be a $d_V$-stable subset of a cochain complex $V^{[*]}$ with a distinguished vector $\mathbf{1} \in V^{[0]}$, an endomorphism $\mathcal{D}_V$ of degree $2N$ such that $\mathcal{D}_V(\mathbf{1})=0$, and a chain map $Y_0(\cdot, x):  T  \longrightarrow \mathcal{E}^{[*]}(V^{[*]})$. Such a subset may satisfy different properties, which we name below:
\begin{enumerate}[leftmargin=*, label=Property (\alph*):]
\item For any $a \in T$,
\begin{align}\label{eq:5.7.2}
Y_0(a, x)\mathbf{1} \in V^{[*]}[[x]] \text{ and } \lim_{x \to 0}Y_0(a, x)\mathbf{1}=a.
\end{align}

\item For any $a \in T$,
\begin{align}\label{eq:5.7.3}
[\mathcal{D}_V, Y_0(a, x)]^s=\frac{d}{dx}Y_0(a, x).
\end{align}

\item For $a, b \in T$, there exists $k \in \mathbb{N}$ such that
\begin{align}\label{eq:5.7.4}
(x_1-x_2)^k[Y_0(a, x_1), Y_0(b, x_2)]^s=0.
\end{align}

\item $V^{[*]}$ is linearly spanned by the vectors
\begin{align}\label{eq:5.7.5}
a_{n_1}^{(1)} \cdots a_{n_r}^{(r)}\mathbf{1}
\end{align}
where $r \geq 0$, $a^{(i)} \in T$, $n_i \in \mathbb{Z}$.
\end{enumerate}

\begin{theorem}[{\cite[Theorem 5.7.1]{Lepowsky-Li}}] \label{thm:5.7.1}
Let $V^{[*]}$ be a cochain complex with a distinguished vector $\mathbf{1} \in V^{[0]}$ such that $d_V(\mathbf{1})=0$, and with an endomorphism $\mathcal{D}_V$ of degree $2N$ such that $\mathcal{D}_V(\mathbf{1})=0$. Let $T$ be a $d_V$-stable subset of $V^{[*]}$ equipped with a chain map
\begin{align}\label{eq:5.7.1}
\begin{array}{rccl}
Y_0(\cdot, x): & T & \longrightarrow & \mathcal{E}^{[*]}(V^{[*]}) \\[5pt]
 & a & \longmapsto & Y_0(a, x)=\displaystyle \sum_{n \in \mathbb{Z}}a_n x^{-n-1}.
 \end{array}
\end{align}
Assume that $T$ satisfies Properties (a), (b), (c) and (d). Then $Y_0(\cdot, x)$ can be extended to a chain map $Y(\cdot, x)$ from $V^{[*]}$ to $\on{Hom}^{[*]}\big(V^{[*]}, V^{[*]}((x))\big)$ such that $(V^{[*]}, Y(\cdot, x), \mathbf{1})$ carries a structure of a dg vertex algebra. The vertex operator $Y(\cdot, x)$ is given by
\begin{align}\label{eq:5.7.6}
Y(a_{n_1}^{(1)} \cdots a_{n_r}^{(r)}\mathbf{1}, x)=a^{(1)}(x)_{n_1} \cdots a^{(r)}(x)_{n_r}\on{id}_V
\end{align}
where for $a \in T$ we write
\begin{align}\label{eq:5.7.7}
a(x)=Y_0(a, x).
\end{align}
The operator $\mathcal{D}_V$ on $V^{[*]}$ agrees with $\mathcal{D}$, i.e.
\begin{align}\label{eq:5.7.8}
\mathcal{D}_V(v)=\mathcal{D}(v)=v_{-2}\mathbf{1}.
\end{align}
Furthermore, by writing
\begin{align}\label{eq:5.7.9}
T(x)=\{a(x) \ | \ a \in T  \} \subseteq \mathcal{E}^{[*]}(V^{[*]}),
\end{align}
the map
\begin{align}\label{eq:5.7.10}
\begin{array}{rccl}
\psi: & \langle T(x) \rangle_{\on{WVA}} & \longrightarrow & V^{[*]} \\[5pt]
 & \alpha(x) & \longmapsto & \on{Res}_x x^{-1}\alpha(x)\mathbf{1}
  \end{array}
\end{align}
is an isomorphism of dg vertex algebras.
\end{theorem}

\begin{proof}
The uniqueness and Formula~\eqref{eq:5.7.6} follow from Remark \ref{rem:5.2.5} and the spanning hypothesis. By Formulas~\eqref{eq:5.7.3}, \eqref{eq:5.7.4} and \eqref{eq:5.7.1}, $T(x)$ is a dg local subset of $\mathcal{E}^{[*]}(V^{[*]}, \mathcal{D}_V)$. For any $a \in T$, $Y_0(\cdot, x)$ satisfies
\[
d_\mathcal{E}(Y_0(a, x))=Y_0(d_V(a), x)
\]
and $d_V(a) \in T$, so we see that $d_\mathcal{E}(Y_0(a, x)) \in T(x)$, hence $T(x)$ is stable by the differential $d_\mathcal{E}$ of $\mathcal{E}^{[*]}(V^{[*]}, \mathcal{D}_V)$. We can apply Theorem \ref{thm:5.6.1} to show that $\langle T(x) \rangle_{\on{WVA}}$ is a dg vertex subalgebra of $\mathcal{E}^{[*]}(V^{[*]}, \mathcal{D}_V)$ with $(V^{[*]}, \mathcal{D}_V)$ as a faithful module. The action is given by $Y_V(\alpha(x), x_0)=\alpha(x_0)$ for $\alpha(x) \in \langle T(x) \rangle_{\on{WVA}}$.

The map $\psi$ in Formula~\eqref{eq:5.7.10} can be expressed as
\[
\psi(\alpha(x))=\on{Res}_{x_0}(x_0^{-1}\alpha(x_0)\mathbf{1})=\on{Res}_{x_0}(x_0^{-1}Y_V(\alpha(x), x_0)\mathbf{1}).
\]
We then apply Proposition \ref{prop:4.7.9} with $w=\mathbf{1}$ and get a $\langle T(x) \rangle_{\on{WVA}}$-module isomorphism
\[
\begin{array}{rccl}
\psi: & \langle T(x) \rangle_{\on{WVA}} & \longrightarrow & V^{[*]} \\[5pt]
 & \alpha(x) & \longmapsto & (-1)^{|\mathbf{1}|}\alpha(x)_{-1}\mathbf{1}=\on{Res}_{x_0}(x_0^{-1}Y_V(\alpha(x), x_0)\mathbf{1}).
  \end{array}
\]
We then regard $\psi$ as a dg vertex algebra isomorphism, and so the dg vertex algebra structure of $\langle T(x) \rangle_{\on{WVA}}$ is transported to $V^{[*]}$. We write $Y(\cdot, x)$ for the vertex operator map on $V^{[*]}$.

We have
\[
\psi(\on{id}_V)=\on{Res}_{x_0}(x_0^{-1}Y_V(\on{id}_V, x_0)\mathbf{1})=\mathbf{1}.
\]
As $\on{id}_V$ is the vacuum vector of $\langle T(x) \rangle_{\on{WVA}}$ and $\psi$ is a dg vertex algebra isomorphism, it follows that $\mathbf{1}$ is the vacuum vector of $V^{[*]}$.

For $a \in T$, we have
\[
\psi(a(x))  =  \on{Res}_{x_0}(x_0^{-1}Y_V(a(x), x_0)\mathbf{1})= a \text{ by Formula}~\eqref{eq:5.7.2}.
\]
Moreover, we get $\psi(Y_\mathcal{E}(a(x), x_0)b(x))=Y(\psi(a(x)), x_0)\psi(b(x))$ because $\psi$ is a homomorphism of dg vertex algebras. Hence
\[
\begin{array}{rcl}
Y(a, x_0) & = &\psi Y_\mathcal{E}(\psi^{-1}(a), x_0) \psi^{-1} \quad \quad \quad \text{ because }\psi \text{ is invertible} \\[5pt]
& = & Y_V(\psi^{-1}(a), x_0)  \quad \text{ because }\psi \text{ is a homomorphism of }\langle T(x)  \rangle_{\on{WVA}} \text{-modules} \\[5pt]
& = & Y_V(a(x), x_0) \\[5pt]
& = & a(x_0).
\end{array}
\]
Thus the map $Y(\cdot, x): V^{[*]} \longrightarrow \on{Hom}^{[*]}\big(V^{[*]}, V^{[*]}((x))\big)$ is an extension of $Y_0(\cdot, x)$.

By Formula~\eqref{eq:5.7.5} and the fact the $Y(a, x)=Y_0(a, x)$ for $a \in T$, we note that $V^{[*]}$ is linearly spanned by the coefficients of all monomials in the expressions
\[
X=Y(a^{(1)}, x_1) Y(a^{(2)}, x_2) \cdots Y(a^{(r)}, x_r)\mathbf{1}
\]
for $a^{(i)} \in T$. Because of the dg vertex algebra structure of $V^{[*]}$, we know that $[\mathcal{D}, Y(v, x)]^s=\frac{d}{dx}Y(v, x)$ and $\mathcal{D}(\mathbf{1})=0$. Hence 
\[
\mathcal{D}(X)=\sum_{i=1}^r  Y(a^{(1)}, x_1) \cdots \frac{d}{d x_i} Y(a^{(i)}, x_i) \cdots Y(a^{(r)}, x_r)\mathbf{1}.
\]
using that $\mathcal{D}$ is even. But then, by using Formula~\eqref{eq:5.7.3} and $\mathcal{D}_V(\mathbf{1})=0$, we get
\[
\mathcal{D}_V(X)=\sum_{i=1}^r  Y(a^{(1)}, x_1) \cdots \frac{d}{d x_i} Y(a^{(i)}, x_i) \cdots Y(a^{(r)}, x_r)\mathbf{1}.
\]
Therefore, $\mathcal{D}$ and $\mathcal{D}_V$ are equal on the elements that linearly span $V^{[*]}$, so $\mathcal{D}=\mathcal{D}_V$.
\end{proof}

\section{Differential graded vertex Lie algebras}\label{sec:4}
We want to generalise Primc's construction (\cite{Primc}) of vertex algebras using vertex Lie algebras to the dg setting. We first need to introduce a notion of dg vertex Lie algebra.

\delete{
We will follow the notation from \cite{Caradot-Jiang-Lin-4}. We recall that for any two differential complexes $(V^{[*]}, d_V^{[*]})$ and $ (U^{[*]}, d_U^{[*]})$ in $ \Ch$, we have 
\begin{align}
\Hom^{[*]}_{\bb C}(U^{[*]}\otimes \bb C[t, t^{-1}], V^{[*]}) = \Hom^{[*]}_{\bb C}(U^{[*]}, V^{[*]})[[x, x^{-1}]]
\end{align}
under the correspondence $ f \longmapsto \sum_{n\in {\bb Z}}f(-\otimes t^n)x^{-n-1}$. 

Here $f(v\otimes t^n)\in W^{[*]}$ is a homogeneous element if $ v\in V^{[*]}$ is homogeneous element. In particular
\begin{align}
\Hom^{[*]}_{\bb C}(U^{[*]}\otimes {\bb C}[t], V^{[*]})=x^{-1}\Hom^{[*]}_{\bb C}(U^{[*]}, V^{[*]})[[ x^{-1}]].
\end{align}
We note that $ {\bb C}[t]$ is the algebra of regular functions on the affine line $ {\bb A}^1$. But we want the fix a point $0\in {\bb A}^1$  which corresponds to the maximal ideal $\mf m=\langle t \rangle$. The algebra ${\bb C}[t]$ is equipped with an $ \mf m$-adic topology. With the discrete topology on the vector spaces $V^{[*]}$ and $U^{[*]}$, then $ U^{[*]}\otimes {\bb C}[t]$ has a linear topology with neighbourhoods of $0$ of the form $ u\otimes \mf m^i$. A linear map 
 $f\in \Hom^{[*]}_{\bb C}(U^{[*]}\otimes {\bb C}[t], V^{[*]})$ is called continuous if $ f^{-1}(0)$ open in $ V\otimes {\bb C}[t]$. This is equivalent to the factor for each $ v\in V$, $f(u\otimes \mf m^r)=0$ for some $r$ (depending on $ u$. We will use $ \Hom_{{\bb C}, c}^{[*]}(U^{[*]}\otimes {\bb C}[t], V^{[*]})$ to denote the space of continuous linear maps. 

Thus $f$ is continues if any only if $ \sum_{n\in {\bb N}}f(u\otimes t^n)x^{-n-1}\in x^{-1}V[x^{-1}]$. 
}

\begin{definition}\label{def:4.1} 
A \textbf{differential graded (dg) vertex Lie algebra}, or dg VLA, in $\Ch$  is a cochain complex $(U^{[*]}, d_U^{[*]})$ over $ \cc$ equipped with a chain map in $ \Ch$ 
\begin{align*}
\begin{array}{cccc}
Y^-(\cdot, x): & U^{[*]} & \longrightarrow & x^{-1}\End^{[*]}(U^{[*]}))[[x^{-1}]] \\
           & u & \longmapsto      & Y^-(v, x)=\displaystyle \sum_{n \in \mathbb{N}}u_{(n)} x^{-n-1}
\end{array}
\end{align*}
with $x$ of degree $-2N$, $N \in \mathbb{Z}$, and with a chain map
\[
\mathcal{D}:U^{[*]} \longrightarrow U[2N]^{[*]}
\]
such that, for any $u, v \in U^{[*]}$ dg homogeneous
\begin{enumerate}[label=(\roman*).]\setlength\itemsep{5pt}
\item $u_{(n)}v=0$ for $n \gg 0$, i.e., $Y^-(u, x)v \in x^{-1}U^{[*]}[x^{-1}]$;
\item $[\mathcal{D}, Y^-(v, x)]^{s}= Y^-(\mathcal{D}v, x)=\frac{d}{dx}Y^-(v, x)$;
\item $Y^-(u, x)v \simeq (-1)^{|u||v|}e^{x\cal D}Y^-(v,-x)u$ \quad (half skew-symmetry);
\item $\begin{array}[t]{l@{\hspace{0cm}}l}
 x_2^{-1}& \delta(\frac{x_1-x_0}{x_2})Y^-(Y^-(u, x_0)v, x_2) \simeq \hfill  \text{(half Jacobi identity)} \\[7pt]
& x_0^{-1}\delta(\frac{x_1-x_2}{x_0})Y^-(u, x_1)Y^-(v, x_2)- (-1)^{|u||v|}x_0^{-1}\delta(\frac{x_2-x_1}{-x_0})Y^-(v, x_2)Y^-(u, x_1),
\end{array}$

where $\delta(x+y)=\displaystyle \sum_{n \in \mathbb{Z}}\sum_{m=0}^{\infty}\binom{n}{m}x^{n-m}y^m$.
\end{enumerate}
In the above expressions, $\simeq$ means that the singular parts are equal, i.e., the coefficients of strictly negative powers of $x$ are equal.
\end{definition}

\begin{remark}
The reader will notice that the notation for the vertex operator map of the dg vertex Lie algebra is slightly different than the one of the dg vertex algebra. Indeed, the modes of an element $u \in U^{[*]}$ are written as $u_{(n)}$, whereas for a dg vertex algebra $V^{[*]}$, the modes of $v \in V^{[*]}$ are written as $v_{n}$. The purpose of this distinction is to differentiate clearly between the action of the dg vertex Lie algebra and that of the dg vertex algebra we will construct (cf. Section \ref{sec:5}).
\end{remark}

As $Y^-(\cdot, x)$ is a chain map, we see that $u_{(n)} \in \on{End}^{[|u|-2N(n+1)]}(U^{[*]})$, so
\[
|u_{(n)}v|=|u|+|v|-2N(n+1).
\]

It is straightforward to verify that for any $u \in U^{[|u|]}$, $v \in U^{[*]}$, and $n \in \mathbb{N}$,
\[
d_U(u_{(n)}v)=d_U(u)_{(n)}v+(-1)^{|u|-2N(n+1)}u_{(n)}d_U(v).
\]

We can see directly that for a dg vertex algebra $(V^{[*]}, Y(\cdot, x), \mathbf{1})$ with translation operator $\mathcal{D}$, there is a dg vertex Lie algebra structure on $V^{[*]}$ given by
\[
Y^-(v, x)=\sum_{n \geq 0}v_n x^{-n-1}.
\]
In this case, for any $n \in \mathbb{N}$, we have $v_{(n)}=v_n$. Based on this observation and the definition above, a dg vertex Lie algebra can be seen as half of a dg vertex algebra. We write $\on{Restr}^{[*]}(-):\on{dgVA} \longrightarrow \on{dgVLA}$ for the functor from the category of dg vertex algebras to the category of dg vertex Lie algebras sending a dg vertex algebra $(V^{[*]}, Y(\cdot, x), \mathbf{1})$ to the dg vertex Lie algebra $(V^{[*]}, Y^-(\cdot, x), \mathcal{D})$.

\begin{remark}\label{rem:4.3} 
In the case of a dg vertex Lie algebra $U^{[*]}$, we do not have dg locality for the operator $Y^-(\cdot, x)$ as in Proposition \ref{prop:3.2.1}. Because we only have the half Jacobi identity, we get that, for any $u, v \in U^{[*]}$ homogeneous, there exists $k \in \mathbb{N}$ such that
\begin{align*}
(x_1-x_2)^k[Y(u, x_1), Y(v, x_2)]^s \simeq 0,
\end{align*}
i.e., we only have dg locality on the singular part.
\end{remark}

\begin{definition} 
A homomorphism $\varphi: U^{[*]} \longrightarrow W^{[*]}$ of dg VLA is a chain map such that $\varphi \circ \mathcal{D}_U=\mathcal{D}_W \circ \varphi$ and satisfying
\[
\varphi(Y_U^-(u, x)v)=Y_W^-(\varphi(u), x)\varphi(v).
\]
By identifying each coefficient, we get
\[
\varphi(u_{(n)}v)=\varphi(u)_{(n)}\varphi(v) \quad \forall n \geq 0.
\]
\end{definition}

We state the following definition given in \cite{Lu-Wang-Zhang}:

\begin{definition}\label{def:dg_Lie_algebra}
A \textbf{differential graded (dg) Lie algebra of degree }$p$ ($p \in \mathbb{Z}$) is an object $(\mathfrak{g}^{[*]}, d_{\mathfrak{g}^{[*]}}) \in \Ch$ equipped with a chain map $[\cdot, \cdot] : \mathfrak{g}^{[*]} \otimes \mathfrak{g}^{[*]} \longrightarrow \mathfrak{g}^{[*]}[p]$ such that, for all elements $a, b, c \in \mathfrak{g}^{[*]}$:
\begin{enumerate}[label=(\roman*).]\setlength\itemsep{5pt}
\item $[a, b]=-(-1)^{(|a|+p)(|b|+p)}[b, a]$  \quad (antisymmetry);
\item $[a, [b, c]]=[[a,b],c]+(-1)^{(|a|+p)(|b|+p)}[b,[a, c]]$  \quad (Jacobi identity);
\end{enumerate}
\end{definition}

\begin{remark}
The chain map condition is equivalent to the bracket being of degree $p$ and such that for any $a, b \in \mathfrak{g}^{[*]}$, we have $d_{\mathfrak{g}^{[*]}}[a, b] = [d_{\mathfrak{g}^{[*]}}(a), b]+(-1)^{(|a|+p)}[a, d_{\mathfrak{g}^{[*]}}(b)]$. Moreover, saying that $(\mathfrak{g}^{[*]}, [\cdot, \cdot])$ is a dg Lie algebra of degree $p$ is equivalent to saying that $\mathfrak{g}[-p]$ equipped with $[\cdot, \cdot]'=[\cdot, \cdot][-2p] : \mathfrak{g}^{[*]}[-p] \otimes \mathfrak{g}^{[*]}[-p] \longrightarrow \mathfrak{g}^{[*]}[-p]$ is a dg Lie algebra of degree $0$.
\end{remark}

\begin{lemma} 
Let $U^{[*]}$ be a dg VLA. Then $U/\mathcal{D}U$ is a dg Lie algebra with the Lie bracket of degree $-2N$ given by
\[
[u+\mathcal{D}U, v+\mathcal{D}U]=u_{(0)}v+\mathcal{D}U.
\]
\end{lemma}

\begin{proof}
We have $\begin{array}[t]{rcl}
Y^-(\mathcal{D}u, x) & = & \displaystyle \sum_{n \geq 0} (\mathcal{D}u)_{(n)} x^{-n-1} \\[15pt]
& = & \displaystyle\frac{d}{dx}Y^-(u, x) \\[10pt]
& = &\displaystyle  \sum_{n \geq 1} -n u_{(n-1)} x^{-n-1}.
\end{array}$ \\
So $(\mathcal{D}u)_{(0)}=0$. Thus $\mathcal{D}U$ is stable by right multiplication $(-)_{(0)} v$. Using the half skew-symmetry
\[
\begin{array}[t]{rcl}
Y^-(u, x)\mathcal{D}v & \simeq & (-1)^{|u||v|}e^{x\cal D}Y^-(\mathcal{D}v,-x)u \\[5pt]
& \simeq & (-1)^{|u||v|}e^{x\cal D}(\frac{d}{d(-x)}Y^-(v,-x)u).
\end{array}
\]
So $\begin{array}[t]{rcl}
u_0 \mathcal{D}v & = & \on{Res}_x \big [(-1)^{|u||v|}e^{x\cal D}\sum_{n \geq 1}(-1)^{n+1} (-n) v_{(n-1)}u x^{-n-1}\big ] \\[5pt]
& = & \on{Res}_x \big [(-1)^{|u||v|}\big(\sum_{m \geq 0}\frac{1}{m!}(x\mathcal{D})^m\big)\big(\sum_{n \geq 1}(-1)^n n v_{(n-1)} u x^{-n-1}\big) \big ] \\[5pt]
& = &  (-1)^{|u||v|}\big(\sum_{n \geq 1}\frac{(-1)^n n}{n!}\mathcal{D}^n(v_{(n-1)}u)\big) \\[5pt]
& \in & \mathcal{D}U.
\end{array}$ \\
Hence $\mathcal{D}U$ is stable by left multiplication by $u_{(0)}$.

We notice from the half skew-symmetry (iii) that the $0$-mode satisfies the relation $u_0v=$ $(-1)^{|u||v|}\sum_{n \geq 0}\frac{1}{n!}\mathcal{D}^n(v_{(n)}u)(-1)^{-n-1}$ $ \in -(-1)^{|u||v|}v_{(0)}u+\mathcal{D}U$. Therefore the Lie bracket is skew-commutative.

By looking at the half Jacobi identity component wise and considering the coefficient for $x_0^{-1}x_1^{-1}x_2^{-1}$, we get $u_{(0)}v_{(0)}w=(-1)^{|u||v|}v_{(0)}u_{(0)}w+(u_{(0)}v)_{(0)}w$, which can be rewritten as
\[
[u, [v, w]]=[[u, v], w]+(-1)^{(|u|-2N)(|v|-2N)}[v, [u, w]].
\]

Finally, as $\mathcal{D}$ is a chain map, we have $d_U(\mathcal{D}(U^{[q]}))=\mathcal{D}(d_U(U^{[q]})) \subset \mathcal{D}(U^{[q+1]})$, so $(\mathcal{D}(U^{[*]}), d_U)$ is a subcomplex of $(U^{[*]}, d_U)$, and $U^{[*]}/\mathcal{D}(U^{[*]})$ is also a complex. Based on what we described above, $U^{[*]}/\mathcal{D}(U^{[*]})$ is a dg Lie algebra of degree $-2N$ with Lie bracket given in the statement.
\end{proof}

Set $\on{dgLA}^{\on{even}}$ the category of dg Lie algebras with even brackets, and $\on{dgVLA}$ the category of dg vertex Lie algebras. We define a functor $\on{VLA}(-): \on{dgLA}^{\on{even}} \longrightarrow \on{dgVLA}$ such that for a dg Lie algebra $(\mathfrak{g}^{[*]}, [\cdot, \cdot])$ with even bracket, we have $\on{VLA}(\mathfrak{g}^{[*]})=\mathfrak{g}^{[*]}$ as vector spaces, and the dg vertex Lie algebra structure on $\on{VLA}(\mathfrak{g}^{[*]})$ is given by $Y^-(g, x)=[g, \cdot] x^{-1}$ and $\mathcal{D}=0$.

\begin{remark}
The functor $\on{VLA}$ starts from the category of even dg Lie algebras because we work with the even loop. Indeed, in the definition of a dg vertex algebra, we impose $|x|=-2N$, which is equivalent to considering the loop algebra $\cc[t, t^{-1}]$ with $|t|=2N$ (see \cite[Section E]{Caradot-Jiang-Lin-4}). Joyce constructed a vertex algebra structure on the cohomology of a moduli stack (see \cite{Joyce}), and it is equipped with an action of the $\cc^*$-equivariant cohomology of a point. However this equivariant cohomology is $\cc[\tau]$ with $|\tau|=2$. As the cohomology of a dg vertex algebra is itself a vertex algebra (see \cite[Theorem III.10]{Caradot-Jiang-Lin-4}), we chose our $t$ to have even degree. An important consequence of having $|t|$ even is that it naturally imposes the differential on $\cc[t, t^{-1}]$ to be zero. In Definition \ref{def:4.1}, we could consider the loop algebra with $|t|$ odd and impose its differential to be zero. All the subsequent results would still work but then the functor $\on{VLA}$ would start from the category of odd dg Lie algebras.
\end{remark}

The next result is a natural consequence of $\mathcal{D}$ commuting with any dg vertex Lie algebra homomorphism.

\begin{proposition}
The functor $U^{[*]} \mapsto U^{[*]}/\mathcal{D}U^{[*]}$ is left adjoint to $\on{VLA}(-)$, i.e., for any dg vertex Lie algebra $(U^{[*]}, Y^-(\cdot, x), \mathcal{D})$ with $|\mathcal{D}|=2N$ and any dg Lie algebra $\mathfrak{g}^{[*]}$ with bracket of degree $-2N$, we have
\[
\on{Hom}_{\on{dgLA}}(U^{[*]}/\mathcal{D}U^{[*]}, \mathfrak{g}^{[*]}) \cong \on{Hom}_{\on{dgVLA}}(U^{[*]}, \on{VLA}(\mathfrak{g}^{[*]})).
\]
\end{proposition}

Consider $\cc[t, t^{-1}]$ as a complex with $|t|=-2N$, $N \in \mathbb{Z}$, so $d=0$. Then $U^{[*]} \otimes \cc[t, t^{-1}]$ becomes a complex with
\[
d(u \otimes t^n)=d_U(u) \otimes t^n
\]
and $|u \otimes t^n|=|u|-2Nn$. Set
\[
\begin{array}{cccc}
\widehat{\mathcal{D}}: & U^{[*]} \otimes \cc[t, t^{-1}] & \longrightarrow & U^{[*]} \otimes \cc[t, t^{-1}] \\[5pt]
 & u \otimes t^n & \longmapsto & \mathcal{D}u \otimes t^n+nu \otimes t^{n-1}.
\end{array}
\]
We can verify that $\widehat{\mathcal{D}}$ is homogeneous of degree $2N$, and so $\widehat{\mathcal{D}}(U^{[*]} \otimes \cc[t, t^{-1}])=\bigoplus_{p \in \mathbb{Z}}\widehat{\mathcal{D}}((U^{[*]} \otimes \cc[t, t^{-1}])^{[p]})$. Furthermore, $\widehat{\mathcal{D}}$ commutes with $d$ on $U^{[*]} \otimes \cc[t, t^{-1}]$, making $\widehat{\mathcal{D}}(U^{[*]} \otimes \cc[t, t^{-1}])$ a subcomplex of $U^{[*]} \otimes \cc[t, t^{-1}]$. Therefore the space
\[
\mathcal{L}(U^{[*]})=U^{[*]} \otimes \cc[t, t^{-1}]/\widehat{\mathcal{D}}(U^{[*]} \otimes \cc[t, t^{-1}])
\]
is itself a complex. We write $u_n$ for the image of $u \otimes t^n$ in $\mathcal{L}(U^{[*]})$ (not to be confused with the notation $u_{(n)} \in \End^{[*]}(U^{[*]}))$). The differential is given by $d_{\mathcal{L}(U)}(u_n)=(d_U(u))_n$.

\begin{theorem}\label{thm:dg_Lie} 
Let $(U^{[*]}, Y^-(\cdot, x), \mathcal{D})$ be a dg vertex Lie algebra. The quotient $\mathcal{L}(U^{[*]})$ is a dg Lie algebra with Lie bracket of degree $-2N$ given by
\begin{align}\label{eq:dg Lie}
[u_n, v_p]=\sum_{i \geq 0}\binom{n}{i}(u_{(i)}v)_{n+p-i}
\end{align}
for $u, v \in U^{[*]}$ and $n, p \in \mathbb{Z}$. Moreover, we have
\begin{align}\label{eq:dg Lie2}
(\mathcal{D}u)_n=-nu_{n-1}
\end{align}
for all $u \in U^{[*]}$, $n \in \mathbb{Z}$, and the map $\mathcal{D}: \mathcal{L}(U^{[*]}) \longrightarrow \mathcal{L}(U^{[*]})$ sending $u_n$ to $(\mathcal{D}u)_n$ is a derivation of degree $2N$ of $\mathcal{L}(U^{[*]})$.
\end{theorem}

\begin{proof}
Formula~\eqref{eq:dg Lie} defines a bilinear map on $U^{[*]} \otimes \cc[t, t^{-1}]$, and since $|(u_{(i)}v)_{n+p-i}|=|u|+|v|-2N(n+p+1)$, it follows that
\[
\begin{array}{cccc}
\{\cdot, \cdot\}: & (U^{[*]} \otimes \cc[t, t^{-1}]) \otimes (U^{[*]} \otimes \cc[t, t^{-1}]) & \longrightarrow & U^{[*]} \otimes \cc[t, t^{-1}]\\
 & (u \otimes t^m) \otimes (v \otimes t^n) & \longmapsto & \displaystyle \sum_{i \geq 0}\binom{n}{i}(u_{(i)}v)_{n+p-i}
\end{array}
\]
is a linear map of degree $-2N$.

We can verify explicitly that $\{\widehat{\mathcal{D}}(u \otimes t^n), v \otimes t^p\}=0$ and $\{u \otimes t^n, \widehat{\mathcal{D}}(v \otimes t^p)\} \in \widehat{\mathcal{D}}(U^{[*]} \otimes \cc[t, t^{-1}])$, hence $\{\cdot, \cdot\}$ induces a bilinear map $[\cdot, \cdot]$ on the quotient $\mathcal{L}(U^{[*]})$. 

Then we have
\[
\begin{array}{rcl}
\mathcal{D}([u_n, v_p]) & = & \sum_{i \geq 0}\binom{n}{i}\mathcal{D}(u_{(i)}v)\otimes t^{n+p-i} \\[5pt]
& = & \sum_{i \geq 0}\binom{n}{i}\big ( (\mathcal{D}u)_{(i)}v+u_{(i)}\mathcal{D}v \big )\otimes t^{n+p-i} \\[5pt]
& = & [(\mathcal{D}u)_n, v_p]+[u_n, (\mathcal{D}v)_p] \\[5pt]
& = &[\mathcal{D}(u_n), v_p]+[u_n, \mathcal{D}(v_p)].
\end{array}
\]
So $\mathcal{D}: \mathcal{L}(U^{[*]}) \longrightarrow \mathcal{L}(U^{[*]})$ is a derivation of degree $2N$.

We now need to verify the dg Lie algebra relations. By setting 
\begin{align*}
\begin{array}{cccc}
Y(\cdot, x): & U^{[*]} & \longrightarrow & \mathcal{L}(U^{[*]})[[x, x^{-1}]] \\
           & u & \longmapsto      & \displaystyle \sum_{n \in \mathbb{Z}}u_{n} x^{-n-1}
\end{array}
\end{align*}
we can rewrite the relations \eqref{eq:dg Lie} and \eqref{eq:dg Lie2} as
\begin{align}
[Y(u, x_1), Y(v, x_2)] &= \displaystyle \on{Res}_{x_0} \bigg[ x_2^{-1} \delta(\frac{x_1-x_0}{x_2})Y(Y^-(u, x_0)v, x_2) \bigg], \label{eq:commut_Res} \\[5pt]
Y(\mathcal{D}u, x)&=\displaystyle \frac{d}{dx}Y(u, x). \label{eq:commut_Res2}
\end{align}
Here $\mathcal{L}(U^{[*]})[[x, x^{-1}]]$ is seen as a dg Lie algebra through the Lie bracket on $\mathcal{L}(U^{[*]})$, and the bracket in \eqref{eq:commut_Res} is the Lie bracket of this dg Lie algebra structure. 

Equation \eqref{eq:commut_Res2} is obtained as an immediate consequence of the definition of $\widehat{\mathcal{D}}$. Using \cite[(2.2.19), (2.3.17) and (2.3.56)]{Lepowsky-Li} on the right hand side of \eqref{eq:commut_Res}, we can show that $[Y(u, x_1), Y(v, x_2)]=-(-1)^{|u||v|}[Y(v, x_2), Y(u, x_1)]$, and so
\[
[u_n, v_p]=-(-1)^{(|u_n|-2N)(|v_p|-2N)}[v_p, u_n].
\]
For the Jacobi identity, we have
\[
\renewcommand{\arraystretch}{1.5}
\begin{array}{l}
[[Y(u, x_1), Y(v, x_2)], Y(w, x_3)] \\[5pt]
= [\on{Res}_{x_{12}}x_2^{-1} \delta(\frac{x_1-x_{12}}{x_2})Y(Y^-(u, x_{12})v, x_2), Y(w, x_3)] \\[5pt]
= \on{Res}_{x_{23}}x_3^{-1} \delta(\frac{x_2-x_{23}}{x_3}) \on{Res}_{x_{12}} x_2^{-1} \delta(\frac{x_1-x_{12}}{x_2})Y(Y^-(Y^-(u, x_{12})v, x_{23}w), x_3) \\[5pt]
= \on{Res}_{x_{23}}x_3^{-1} \delta(\frac{x_2-x_{23}}{x_3}) \on{Res}_{x_{12}} x_2^{-1} \delta(\frac{x_1-x_{12}}{x_2})  \\[5pt]
 \quad \quad  \times  \on{Res}_{x_{13}} x_{12}^{-1} \delta(\frac{x_{13}-x_{23}}{x_{12}})Y(Y^-(u, x_{13})Y^-(v, x_{23})w, x_3) \\[5pt]
\quad -(-1)^{|u||v|} \on{Res}_{x_{23}}x_3^{-1} \delta(\frac{x_2-x_{23}}{x_3}) \on{Res}_{x_{12}} x_2^{-1} \delta(\frac{x_1-x_{12}}{x_2}) \\[5pt]
\quad \quad   \times  \on{Res}_{x_{13}} x_{12}^{-1} \delta(\frac{x_{23}-x_{13}}{-x_{12}})Y(Y^-(v, x_{23})Y^-(u, x_{13})w, x_3).
\end{array}
\]
But the following formula is known (see \cite[page 15]{Primc}):
\begin{equation*}
\scalebox{0.98}{$ \displaystyle 
 \on{Res}_{x_{12}} x_2^{-1} \delta(\frac{x_3+x_{23}}{x_2}) x_1^{-1} \delta(\frac{x_2+x_{12}}{x_1}) x_{12}^{-1} \delta(\frac{x_{13}-x_{23}}{x_{12}}) 
 = x_2^{-1} \delta(\frac{x_3+x_{23}}{x_2}) x_1^{-1} \delta(\frac{x_3+x_{13}}{x_1}).$}
\end{equation*}
We can thus rewrite the equation as
\[
\renewcommand{\arraystretch}{1.5}
\begin{array}{l}
[[Y(u, x_1), Y(v, x_2)], Y(w, x_3)] \\[5pt]
=  \on{Res}_{x_{13}}\on{Res}_{x_{23}} x_1^{-1} \delta(\frac{x_3+x_{13}}{x_1}) x_2^{-1} \delta(\frac{x_3+x_{23}}{x_2}) Y(Y^-(u, x_{13})Y^-(v, x_{23})w, x_3) \\[5pt]
\quad -(-1)^{|u||v|} \on{Res}_{x_{23}} \on{Res}_{x_{13}}x_1^{-1} \delta(\frac{x_3+x_{13}}{x_1})  x_2^{-1} \delta(\frac{x_3+x_{23}}{x_2}) Y(Y^-(v, x_{23})Y^-(u, x_{13})w, x_3) \\[5pt]
=[Y(u, x_1), [Y(v, x_2), Y(w, x_3)]] -(-1)^{(|u|-2N)([v]-2N)}[Y(v, x_2), [Y(u, x_1), Y(w, x_3)]]. 
\end{array}
\]
Now comparing the coefficients of $x_1^{-n_1-1}x_2^{-n_2-1}x_3^{-n_3-1}$ for $n_1, n_2, n_3 \in \mathbb{Z}$ on both sides yields the Jacobi identity of $\mathcal{L}(U^{[*]})$.

Finally, we check that $d_{\mathcal{L}(U)}([u_n, v_p])=[d_{\mathcal{L}(U)}(u_n), v_p] + (-1)^{(|u_n|-2N)}[u_n, d_{\mathcal{L}(U)}(v_p)]$.
\end{proof}

\begin{proposition}\label{prop:hom-dg Lie} 
Let $\varphi: U^{[*]} \longrightarrow W^{[*]}$ be a homomorphism of dg VLA. Then
\[
\begin{array}{cccc}
\mathcal{L}(\varphi): & \mathcal{L}(U^{[*]}) & \longrightarrow & \mathcal{L}(W^{[*]}) \\[5pt]
 & u_n & \longmapsto & (\varphi(u))_n
\end{array}
\]
is a homomorphism of dg Lie algebras. Moreover, $\mathcal{L}(\varphi) \circ \mathcal{D}_{U} = \mathcal{D}_{W} \circ \mathcal{L}(\varphi)$.
\end{proposition}

\begin{proof}
We know that $\varphi \circ \mathcal{D}_U = \mathcal{D}_W \circ \varphi$. Consider the extension of $\varphi$ to $U^{[*]} \otimes \cc[t, t^{-1}]$ given by $\varphi(u \otimes t^n) \longmapsto \varphi(u) \otimes t^n$. We can verify that $\varphi(u \otimes t^n +\widehat{\mathcal{D}}(U^{[*]} \otimes \cc[t, t^{-1}]))=\varphi(u) \otimes t^n +\widehat{\mathcal{D}}(U^{[*]} \otimes \cc[t, t^{-1}])$, so $\mathcal{L}(\varphi)$ is well-defined.

Using Equation~\eqref{eq:dg Lie2}, we can check that $\mathcal{L}(\varphi) \circ \mathcal{D}_{U} = \mathcal{D}_{W} \circ \mathcal{L}(\varphi)$. Furthermore, with Formula~\eqref{eq:dg Lie} and the fact that $\varphi$ is a homomorphism of dg VLA, we can verify that $ \mathcal{L}(\varphi)([u_n, v_p])=[ \mathcal{L}(\varphi)(u_n), \mathcal{L}(\varphi)(v_p)]$.

As $\varphi$ is a degree $0$ map, then so is $\mathcal{L}(\varphi)$. Moreover, because $\varphi$ is a chain map, then $\mathcal{L}(\varphi) \circ d_{\mathcal{L}(U)} = d_{\mathcal{L}(W)} \circ \mathcal{L}(\varphi)$, so $\mathcal{L}(\varphi)$ is a chain map. As it preserves the Lie bracket, it is a homomorphism of dg Lie algebras.
\end{proof}

Set $\mathcal{L}(U^{[*]})_-=\on{Span}\{u_n \ | \ u \in U^{[*]}, \, n <0  \}$ and $\mathcal{L}(U^{[*]})_+=\on{Span}\{u_n \ | \ u \in U^{[*]}, \, n \geq 0  \}$

\begin{proposition}\label{prop:L(U)_+-} 
The spaces $\mathcal{L}(U^{[*]})_-$ and $\mathcal{L}(U^{[*]})_+$ are $\mathcal{D}$-stable dg Lie subalgebras of $\mathcal{L}(U^{[*]})$, and $\mathcal{L}(U^{[*]})=\mathcal{L}(U^{[*]})_+ \oplus \mathcal{L}(U^{[*]})_-$ as vector spaces. We also have the following equalities:
\[
\left\{\begin{array}{rcl}
\mathcal{L}(U^{[*]})_- & = &  (U^{[*]} \otimes t^{-1}\cc[t^{-1}])/\widehat{\mathcal{D}}(U^{[*]} \otimes t^{-1}\cc[t^{-1}]),\\[5pt]
\mathcal{L}(U^{[*]})_+ & = &  (U^{[*]} \otimes \cc[t])/\widehat{\mathcal{D}}(U^{[*]} \otimes \cc[t]).
\end{array}\right.
\]
\end{proposition}

\begin{proof}
By looking at the indices in Formula~\eqref{eq:dg Lie}, we can see that $\mathcal{L}(U^{[*]})_-$ and $\mathcal{L}(U^{[*]})_+$ are stable by the Lie bracket. As $\mathcal{D}(u_n)=-nu_{n-1}$, it follows that $\mathcal{L}(U^{[*]})_-$ and $\mathcal{L}(U^{[*]})_+$ are stable by $\mathcal{D}$. We then verify that
\[
\begin{array}{rcl}
\mathcal{L}(U^{[*]}) & = &  (U^{[*]} \otimes \cc[t, t^{-1}])/\widehat{\mathcal{D}}(U^{[*]} \otimes \cc[t, t^{-1}]) \\[5pt]
 & = &  (U^{[*]} \otimes \cc[t] \oplus U^{[*]} \otimes t^{-1}\cc[t^{-1}])/\widehat{\mathcal{D}}(U^{[*]} \otimes \cc[t] \oplus U^{[*]} \otimes t^{-1}\cc[t^{-1}]) \\[5pt]
 & \cong &  (U^{[*]} \otimes \cc[t])/\widehat{\mathcal{D}}(U^{[*]} \otimes \cc[t]) \oplus (U^{[*]} \otimes t^{-1}\cc[t^{-1}])/\widehat{\mathcal{D}}(U^{[*]} \otimes t^{-1}\cc[t^{-1}]) \\[5pt]
 & \cong &  \mathcal{L}(U^{[*]})_+ \oplus \mathcal{L}(U^{[*]})_-.
\end{array}
\]
Finally, it is clear that both $\mathcal{L}(U^{[*]})_+$ and $\mathcal{L}(U^{[*]})_-$ are complexes whose differentials are given by $d_{\mathcal{L}(U)}$. Thus they are dg Lie algebras.
\end{proof}

A direct consequence of Propositions \ref{prop:hom-dg Lie} and \ref{prop:L(U)_+-} is as follows:

\begin{proposition} \label{prop:hom-dg Lie-L(U)_+-}
Let $\varphi: U^{[*]} \longrightarrow W^{[*]}$ be a homomorphism of dg VLA, and write $\mathcal{L}(\varphi)_\pm$ for the restriction of $\mathcal{L}(\varphi)$ to $\mathcal{L}(U^{[*]})_\pm$. Then $\mathcal{L}(\varphi)_\pm(\mathcal{L}(U^{[*]})_\pm) \subseteq \mathcal{L}(W^{[*]})_\pm$, and
\[
\mathcal{L}(\varphi)_+: \mathcal{L}(U^{[*]})_+ \longrightarrow \mathcal{L}(W^{[*]})_+ \quad , \quad  \mathcal{L}(\varphi)_-: \mathcal{L}(U^{[*]})_- \longrightarrow \mathcal{L}(W^{[*]})_- 
\]
are homomorphisms of dg Lie algebras. Moreover, $\mathcal{L}(\varphi)_\pm \circ \mathcal{D}_{U} = \mathcal{D}_{W} \circ \mathcal{L}(\varphi)_\pm$.
\end{proposition}

\begin{theorem}\label{thm:U_L(U)-}
The map
\[
\begin{array}{rccl}
\iota_U: & U^{[*]} & \longrightarrow & \mathcal{L}(U^{[*]})[2N]_- \\[5pt]
 & u & \longmapsto & u_{-1}
 \end{array}
 \]
 is an isomorphism of complexes, and $\iota_U \circ \mathcal{D} =\mathcal{D} \circ \iota_U$. Moreover, if $\varphi: U^{[*]} \longrightarrow W^{[*]}$ is a homomorphism of dg VLA, then $\iota_W \circ  \varphi = \mathcal{L}(\varphi)_- \circ \iota_U$. In particular, any dg vertex Lie algebra is automatically a dg Lie algebra.
\end{theorem}

\begin{proof}
Using Formula~\eqref{eq:dg Lie2}, we see that $u_{-k-1}=\frac{1}{k!}(\mathcal{D}^ku)_{-1}$ for all $k \geq 0$. Hence $\iota_U$ is surjective.

Consider $u \in U^{[*]}$ such that $\iota_U(u)=0$, i.e., $u \otimes t^{-1} \in \widehat{\mathcal{D}}(U^{[*]} \otimes \cc[t, t^{-1}])$. We can write $u \otimes t^{-1}=\sum_{i=-n}^m (\mathcal{D}v^{(i)} \otimes t^i+i v^{(i)} \otimes t^{i-1})$ where $v^{(i)} \in U^{[*]}$ and $m \geq 0, n \geq 1$ minimal, hence $v^{(m)} \neq 0$ and $v^{(-n)} \neq 0$. By comparing the coefficients of $t^k$ for $k \geq 0$, we see that $\sum_{i=0}^m (\mathcal{D}v^{(i)} \otimes t^i+i v^{(i)} \otimes t^{i-1})=0$, implying that $m=0$ by minimality. But then $u \otimes t^{-1}=\sum_{i=-n}^{-1} (\mathcal{D}v^{(i)} \otimes t^i+i v^{(i)} \otimes t^{i-1})$, contradicting the minimality of $n$. Hence $v^{(i)}=0$ for all $i$, and so $u=0$, implying that the map $\iota_U$ is injective. Therefore $\iota_U$ is an isomorphism of vector spaces. As $|t|=-2N$, we see that $\iota_U : U^{[*]} \stackrel{\cong}{\longrightarrow} \mathcal{L}(U^{[*]})[2N]_-$ as graded vector spaces.

As $\iota_U(d_U(u))=(d_U(u))_{-1}=d_{\mathcal{L}(U)}(u_{-1})=d_{\mathcal{L}(U)}(\iota_U(u))$, we see that $\iota_U$ is a chain map. So $\iota_U$ is an isomorphism of complexes. The remaining statements are direct consequences of Theorem \ref{thm:dg_Lie} and Proposition \ref{prop:hom-dg Lie}.
\end{proof}

With the above theorem, we can identify $U^{[*]}$ and $\mathcal{L}(U^{[*]})[2N]_- $, so we can see $U^{[*]}$ as a dg Lie algebra with Lie bracket
\begin{align}\label{eq:U_dg Lie}
[u, v] = \sum_{i \geq 0}\frac{(-1)^i}{(i+1)!}\mathcal{D}^{i+1}(u_{(i)}v).
\end{align}
As it was the case for $\mathcal{D}$ in Theorem \ref{thm:dg_Lie}, the map $\mathcal{D} : U^{[*]} \longrightarrow U^{[*]}$ is a derivation of degree $2N$ for the dg Lie algebra structure \eqref{eq:U_dg Lie} on $U^{[*]}$. Using Proposition \ref{prop:hom-dg Lie-L(U)_+-}, we see that any homomorphism of dg VLA $\varphi: U^{[*]} \longrightarrow W^{[*]}$ gives a homomorphism of dg Lie algebras $U^{[*]} \longrightarrow W^{[*]}$ with the Lie bracket given by Formula~\eqref{eq:U_dg Lie}. This in turn extends to a homomorphism $ \mathcal{L}(U^{[*]}) \longrightarrow  \mathcal{L}(W^{[*]})$ of dg Lie algebras by Proposition \ref{prop:hom-dg Lie}.

\begin{lemma}\label{lem:Jacobi_associator}
Let $(U^{[*]}, d_U^{[*]})$ be a cochain complex equipped with an endomorphism $\mathcal{D}$ of degree $2N$, $N \in \mathbb{Z}$, and a chain map
\begin{align*}
\begin{array}{cccc}
Y^-(\cdot, x): & U^{[*]} & \longrightarrow & x^{-1}\End^{[*]}(U^{[*]}))[[x^{-1}]] \\
           & u & \longmapsto      & Y^-(u, x)=\displaystyle \sum_{n \in \mathbb{N}}u_{(n)} x^{-n-1}
\end{array}
\end{align*}
such that:
\begin{itemize}\setlength\itemsep{5pt}
\item $u_{(n)} v =0$ for $n \gg 0$,
\item $[\mathcal{D}, Y^-(u, x)]^s=Y^-(\mathcal{D}u, x)=\displaystyle \frac{d}{dx}Y^-(u, x)$.
\end{itemize}
Then $Y^-(\cdot, x)$ satisfies the half Jacobi identity if and only if it satisfies the half associator formula:
\begin{align}\label{eq:half assoc}
\begin{split}
Y^-(Y^-(u, x_0)v, x_2) & \simeq  \on{Res}_{x_1}\big[  x_0^{-1}\delta(\frac{x_1-x_2}{x_0})Y^-(u, x_1)Y^-(v, x_2) \\[5pt] 
 &  \quad \quad  -(-1)^{|u||v|}x_0^{-1}\delta(\frac{x_2-x_1}{-x_0})Y^-(v, x_2)Y^-(u, x_1) \big]. 
 \end{split} 
\end{align}
\end{lemma}

\begin{proof}
If the half Jacobi identity is satisfied, then by talking $\on{Res}_{x_1}$ of the identity we obtain the half associator. 

We assume the Equation~\eqref{eq:half assoc} is satisfied. We then apply \cite[Proposition 2.3.8 and Remark 2.3.12]{Lepowsky-Li} to see that for any $n \geq 0$
\[
\renewcommand{\arraystretch}{1.2}
\begin{array}{l}
 \on{Res}_{x_1} x_1^n  x_2^{-1}\delta(\frac{x_1-x_0}{x_2})Y^-(Y^-(u, x_0)v, x_2) \\[5pt]
  = \on{Res}_{x_1}(x_2+x_0)^n x_2^{-1}\delta(\frac{x_1-x_0}{x_2})Y^-(Y^-(u, x_0)v, x_2) \\[5pt]
= (x_2+x_0)^n Y^-(Y^-(u, x_0)v, x_2) \\[5pt]
\simeq (x_2+x_0)^n \on{Res}_{x_1}\big[  x_0^{-1}\delta(\frac{x_1-x_2}{x_0})Y^-(u, x_1)Y^-(v, x_2) \\
\hspace{3.7cm}-(-1)^{|u||v|}x_0^{-1}\delta(\frac{x_2-x_1}{-x_0})Y^-(v, x_2)Y^-(u, x_1) \big] \\[5pt]
= \on{Res}_{x_1}\big[  x_1^n x_0^{-1}\delta(\frac{x_1-x_2}{x_0})Y^-(u, x_1)Y^-(v, x_2) \\
\hspace{1.7cm} -(-1)^{|u||v|}x_1^n x_0^{-1}\delta(\frac{x_2-x_1}{-x_0})Y^-(v, x_2)Y^-(u, x_1) \big].
\end{array}
\]
This indicates that the singular part of the coefficient of $x_1^{-n-1}$ on both sides of the Jacobi identity are equal for all $n \geq 0$, i.e., the half Jacobi identity is satisfied.
\end{proof}

\begin{theorem}\label{thm:dg vertex-Lie_dg Lie}
Let $(U^{[*]}, d_U^{[*]})$ be a cochain complex equipped with an endomorphism $\mathcal{D}$ of degree $2N$ and a chain map
\begin{align*}
\begin{array}{cccc}
Y^-(\cdot, x): & U^{[*]} & \longrightarrow & x^{-1}\End^{[*]}(U^{[*]}))[[x^{-1}]] \\
           & u & \longmapsto      & Y^-(u, x)=\displaystyle \sum_{n \in \mathbb{N}}u_{(n)} x^{-n-1}
\end{array}
\end{align*}
such that:
\begin{itemize}
\item $u_{(n)} v =0$ for $n \gg 0$,
\item $[\mathcal{D}, Y^-(u, x)]^s=Y^-(\mathcal{D}u, x)=\displaystyle \frac{d}{dx}Y^-(u, x)$.
\end{itemize}
Then $U^{[*]}$ is a dg vertex Lie algebra if and only if $\mathcal{L}(U^{[*]})$ is a dg Lie algebra.
\end{theorem}

\begin{proof}
One direction is exactly Theorem \ref{thm:dg_Lie}. We now show that $U^{[*]}$ is a dg vertex Lie algebra, assuming that $\mathcal{L}(U^{[*]})$ is a dg Lie algebra. We only need to prove that the half skew-symmetry (iii) and the half Jacobi identity (iv) in Definition \ref{def:4.1} are satisfied.

By Theorem \ref{thm:U_L(U)-}, we know that $u=0$ if and only if $u_{-1}=0$, and so by using the definition of $\mathcal{L}(U^{[*]})$, we see that $u=0$ if and only if $Y(u, x)=0$. Furthermore, if we use Equation~\eqref{eq:commut_Res}, the skew-symmetry of the Lie bracket on $\mathcal{L}(U^{[*]})$ and \cite[(2.2.19) and (2.3.17)]{Lepowsky-Li}, we can show that
\[
\on{Res}_{x_0}\bigg[ x_1^{-1} \delta(\frac{x_2+x_0}{x_1})Y\big(Y^-(u, x_0)v-(-1)^{|u||v|}e^{x_0 \mathcal{D}}Y^-(v, -x_0)u, x_2\big) \bigg]=0.
\]
We write $f_i$ for the coefficient of $x_0^{-i-1}$ in $Y^-(u, x_0)v-(-1)^{|u||v|}e^{x_0 \mathcal{D}}Y^-(v, -x_0)u$. Then the above equality can be rewritten as
\[
\sum_{i \geq 0} x_1^{-i-1}\delta^{(i)}\bigg(\frac{x_2}{x_1}\bigg) \frac{1}{i!}Y(f_i, x_2)=0
\]
with $\delta^{(i)}(\frac{x_2}{x_1})=\sum_{n \in \mathbb{Z}}\frac{n!}{(n-i)!}x_2^{n-i}x_1^{-n+i}$. The sum over $i$ is finite because $f_i=0$ for $i \gg 0$. By \cite[Lemma 2.1.4]{Li}, we get that $Y(f_i, x_2)=0$ for all $i$, and so $f_i=0$ for all $i$. But this means that the singular part of $Y^-(u, x_0)v-(-1)^{|u||v|}e^{x_0 \mathcal{D}}Y^-(v, -x_0)u$ is zero, which is equivalent to the half skew-symmetry being satisfied.

As $\mathcal{L}(U^{[*]})$ is assumed to be a dg Lie algebra, the Jacobi identity for $\mathcal{L}(U^{[*]})$ is:
\begin{align*}
[[Y(u, x_1), Y(v, x_2)], Y(w, x_3)]=[Y(u, x_1), [Y(v, x_2), Y(w, x_3)]] \\[5pt]
-(-1)^{|u||v|}[Y(v, x_2], [Y(u, x_1), Y(w, x_3)]].
\end{align*}
Using Equation~\eqref{eq:commut_Res} to rewrite the Lie brackets in the above equation, it becomes an equality between power series in $x_1, x_2$, and $x_3$. We then apply \cite[Lemma 2.1.4]{Li} successively and use that fact the $Y(\cdot, x)$ is injective to show that, for all $i, j \geq 0$, 
\begin{equation*}
\scalebox{0.93}{$
\begin{split}
&(u_{(j)} v)_{(i)} w  \text{ is the coefficient of } x_0^{-j-1}x_2^{-i-1} \text{ in } \\[5pt]
& \on{Res}_{x_{1}}\big[  x_0^{-1} \delta(\frac{x_{1}-x_2}{x_0})Y^-(u, x_{1})Y^-(v, x_2)w  -(-1)^{(|u||v|} x_0^{-1} \delta(\frac{x_2-x_{1}}{-x_0})Y^-(v, x_2)Y^-(u, x_{1})w \big],
 \end{split}
 $}
 \end{equation*}
 which is the component version of the half associator formula~\eqref{eq:half assoc} applied to $w$. It then follows from Lemma \ref{lem:Jacobi_associator} that the half Jacobi identity is satisfied, which concludes the proof.
\end{proof}

\begin{remark}
We know from Theorems \ref{thm:dg_Lie} and \ref{thm:U_L(U)-} that if $(U^{[*]}, Y^-(\cdot, x), \mathcal{D})$ is a dg vertex Lie algebra, then $Y^-(\cdot, x)$ gives a dg Lie algebra structure on $U^{[*]}$ given by Equation \eqref{eq:U_dg Lie}. An interesting question raised by the correspondance in Theorem \ref{thm:dg vertex-Lie_dg Lie} between a dg vertex Lie algebra structure on $U^{[*]}$ and a dg Lie algebra structure on $\mathcal{L}(U^{[*]})$ is as follows: if given $(U^{[*]}, Y_U(\cdot, x), \mathcal{D})$ such that Equation \eqref{eq:U_dg Lie} is a Lie bracket, does $Y_U(\cdot, x)$ provide a dg vertex Lie algebra structure to $U^{[*]}$? 
\end{remark}

\begin{remark}\label{rem:IV18}
\begin{enumerate}[wide, nosep]
\item Similarly to the classical case, we can define dg ideals of dg vertex Lie algebras as well as quotients of dg vertex Lie algebras. Such quotients are once again dg vertex Lie algebras.
\item There are two functors from the category of dg vertex Lie algebras to the category of super vertex Lie algebras (resp. from the category of dg Lie algebras to the category of super Lie algebras): $\on{VLA}_{\on{cohom}}$ and $\on{VLA}_{\on{forget}}$ (resp. $\on{LA}_{\on{cohom}}$ and $\on{LA}_{\on{forget}}$). For a dg vertex Lie algebra $U^{[*]}$, the super vertex Lie algebra $\on{VLA}_{\on{cohom}}(U^{[*]})$ is obtained by taking the cohomology of the dg complex $U^{[*]}$ and then grouping the odd and even degrees, and $\on{VLA}_{\on{forget}}(U^{[*]})$ is the super vertex Lie algebra obtained from $U^{[*]}$ by ignoring the differential and grouping the odd and even degrees. We proceed similarly to define $\on{LA}_{\on{cohom}}(\mathfrak{g}^{[*]})$ and $\on{LA}_{\on{forget}}(\mathfrak{g}^{[*]})$ for a dg Lie algebra $\mathfrak{g}^{[*]}$. A dg vertex Lie algebra $U^{[*]}$ (resp. dg Lie algebra $\mathfrak{g}^{[*]}$) is called cohomologically simple if $ \on{VLA}_{\on{cohom}}(U^{[*]})$ (resp. $\on{LA}_{\on{cohom}}(\mathfrak{g}^{[*]})$) is simple as a super vertex Lie algebra (resp. super Lie algebra). If $\on{VLA}_{\on{forget}}(U^{[*]})$ (resp. $\on{LA}_{\on{forget}}(\mathfrak{g}^{[*]})$) is simple, then it implies that $U^{[*]}$ (resp. $\mathfrak{g}^{[*]}$) is simple, as dg ideals are super ideals, but we expect the reciprocal not to be true. Hence for a dg vertex Lie algebra (resp. dg Lie algebra), we have three non equivalent notions of simplicity: simple as dg vertex Lie algebra (resp. dg Lie algebra), simple as super vertex Lie algebra (resp. super Lie algebra), and cohomologically simple.
\end{enumerate}
\end{remark}

\section{The enveloping dg vertex algebra of a dg vertex Lie algebra}\label{sec:5}
Recall that for a dg vertex algebra $(V^{[*]}, Y(\cdot, x), \mathbf{1})$, the tuple $(V^{[*]}, Y^-(\cdot, x), \mathcal{D})$ is a dg vertex Lie algebra and this defines a forgetful functor $\on{Restr}^{[*]}(-): \on{dgVA} \longrightarrow \on{dgVLA}$. In this section, we construct a left adjoint functor for $\on{Restr}^{[*]}(-)$.

Let $U^{[*]}$ be a dg vertex Lie algebra with derivation $\mathcal{D}$ and let $\mathcal{L}(U^{[*]})= \mathcal{L}(U^{[*]})_+ \oplus \mathcal{L}(U^{[*]})_-$ be the corresponding dg Lie algebra with derivation $\mathcal{D}$. Consider $\cc$ as the trivial $\mathcal{L}(U^{[*]})_+$-module. Following \cite{Primc}, we want to define a space $\mathcal{V}^{[*]}(U^{[*]})$ through the universal enveloping algebra $\mathcal{U}(\mathcal{L}(U^{[*]}))$. The problem is that by definition, $\mathcal{U}(\mathcal{L}(U^{[*]}))$ is the quotient of the tensor algebra $T(\mathcal{L}(U^{[*]}))$ by the ideal $I$ generated by the elements of the form $u \otimes v -(-1)^{|u||v|}v \otimes u -[u, v]$. But as $[ \cdot, \cdot]$ is of degree $-2N$, the previous sum is not homogeneous, and so $I$ is not a subcomplex of $T(\mathcal{L}(U^{[*]}))$. But if we understood the tensor product as a map of degree $-2N$, then $\mathcal{U}(\mathcal{L}(U^{[*]}))$ becomes an associative dg algebra with a product of degree $-2N$ (see \eqref{point:2} in Lemma \ref{lem:5.1}).

The proof of the next lemma is straightforward.

\begin{lemma}\label{lem:5.1}
\begin{enumerate}
\item If $(\mathfrak{g}^{[*]}, [\cdot, \cdot])$ is a dg Lie algebra with bracket of degree $n$, then the shift $(\mathfrak{g}^{[*]}[-n], [\cdot, \cdot]_{[-n]})$ is a dg Lie algebra with bracket of degree $0$.
\item\label{point:2} If $(\mathfrak{g}^{[*]}, [\cdot, \cdot])$ is a dg Lie algebra with bracket of degree $n$, then the enveloping algebra $\mathcal{U}(\mathfrak{g}^{[*]}[-n])$ of $\mathfrak{g}^{[*]}[-n]$ is an associative dg algebra where the multiplication is a chain map. Moreover, the enveloping algebra $\mathcal{U}(\mathfrak{g}^{[*]})=\mathcal{U}(\mathfrak{g}^{[*]}[-n])[n]$ of $\mathfrak{g}^{[*]}$ is an associative dg algebra of degree $n$, i.e., the multiplication is of degree $n$.
\item If $\varphi: (\mathfrak{g}^{[*]}, [\cdot, \cdot]) \longrightarrow (\mathfrak{g}'^{[*]}, [\cdot, \cdot]')$ is a homomorphism of dg Lie algebras with $\on{deg}([\cdot, \cdot])=\on{deg}([\cdot, \cdot]')=n$, then $\mathcal{U}(\varphi): \mathcal{U}(\mathfrak{g}^{[*]}) \longrightarrow \mathcal{U}(\mathfrak{g}'^{[*]})$ is a homomorphism of associative dg algebras of degree $n$.
\end{enumerate}
\end{lemma}

To make $I$ into a subcomplex, we shift the gradation on $\mathcal{L}(U^{[*]})$ and consider 
\[
\mathcal{U}(\mathcal{L}(U^{[*]})[2N])=T(\mathcal{L}(U^{[*]})[2N])/\big\langle u \otimes v -(-1)^{|u|_{[2N]}|v|_{[2N]}}v \otimes u -[u, v]_{[2N]} \big\rangle.
\]

Based on Lemma \ref{lem:5.1}, the bracket $[\cdot, \cdot]_{[2N]}$ is a chain map and therefore $\langle u \otimes v -(-1)^{|u|_{[2N]}|v|_{[2N]}}v \otimes u -[u, v]_{[2N]} \big\rangle$ is a graded subspace of the associative dg algebra $T(\mathcal{L}(U^{[*]})[2N])$. With the same lemma, it follows that $\mathcal{U}(\mathcal{L}(U^{[*]})[2N])$ is an associative dg algebra with a multiplication of degree $0$.

We define
\[
\mathcal{V}^{[*]}(U^{[*]})=\mathcal{U}(\mathcal{L}(U^{[*]})[2N]) \otimes_{\mathcal{U}(\mathcal{L}(U^{[*]})[2N]_+)} \cc.
\]
with $\cc$ the trivial $\mathcal{U}(\mathcal{L}(U^{[*]})[2N]_+)$-module. The derivation $\mathcal{D}$ on $\mathcal{L}(U^{[*]})$ extends to a derivation on $\mathcal{U}(\mathcal{L}(U^{[*]})[2N])$, which then defines a linear map of degree $2N$ of $\mathcal{V}^{[*]}(U^{[*]})$. The degree $0$ element $\mathbf{1}=1 \otimes 1 \in \mathcal{V}^{[*]}(U^{[*]})$ satisfies $\mathcal{D}(\mathbf{1})=\mathcal{D}(1) \otimes 1=0$.

Using the identification given by $\iota_U$, we have an injective chain map
\[
\begin{array}{rccl}
\kappa_U: & U^{[*]} & \longrightarrow & \mathcal{V}^{[*]}(U^{[*]}) \\[5pt]
 & u & \longmapsto & u_{-1}\mathbf{1}.
 \end{array}
 \]

\begin{lemma}\label{lem:5.2} 
Consider $\mathcal{V}^{[*]}(U^{[*]})$ as a module over $\mathcal{L}(U^{[*]})[2N]$. Then for any $u_n \in \mathcal{L}(U^{[*]})[2N]$, we have
\[
[\mathcal{D}, u_n]^s=(\mathcal{D}u)_n = -n u_{n-1}
\]
and
\[
\mathcal{D} \circ \kappa_U=\kappa_U \circ \mathcal{D}.
\]
\end{lemma}

\begin{proof}
Set $w=w^{(1)} \cdots w^{(k)}\mathbf{1} \in \mathcal{V}^{[*]}(U^{[*]})$ with $w^{(i)} \in \mathcal{L}(U^{[*]})[2N]$. We know that $\mathcal{D}$ acts on $\mathcal{U}(\mathcal{L}(U^{[*]})[2N])$ as a derivation of degree $2N$, so $\mathcal{D}(u_n w)=(\mathcal{D}u_n)w+u_n(\mathcal{D}w)$. Hence $[\mathcal{D}, u_n]^s w=(\mathcal{D}u_n)w=(\mathcal{D}u)_n w=-nu_{n-1}w$. The last equation of the lemma is a direct consequence of $\mathcal{D}(\mathbf{1})=0$.
\end{proof}

We define the following formal Laurent series
\[
Y_{\mathcal{V}(U)}(u, x)=Y_{\mathcal{V}(U)}(u_{-1}\mathbf{1}, x)=\sum_{n \in \mathbb{Z}}u_n x^{-n-1} \in \mathcal{L}(U^{[*]})[2N][[x, x^{-1}]]
\]
for $u \in U^{[*]}$ and $|x|=-2N$. Since $u_n \in \mathcal{L}(U^{[*]})[2N]$ is acting on $\mathcal{V}^{[*]}(U^{[*]})$ by left multiplication, we can see $u_n$ as an element of $\on{End}^{[*]}(\mathcal{V}^{[*]}(U^{[*]}))$, and thus we have a map
\begin{align*}
\begin{array}{cccc}
Y_{\mathcal{V}(U)}(\cdot, x): & U^{[*]} & \longrightarrow &\on{End}^{[*]}(\mathcal{V}^{[*]}(U^{[*]}))[[x, x^{-1}]] \\
           & u & \longmapsto      & \displaystyle \sum_{n \in \mathbb{Z}}u_{n} x^{-n-1}.
\end{array}
\end{align*}

We can restate the previous lemma as
\[
[\mathcal{D}, Y_{\mathcal{V}(U)}(u, x)]^s=Y_{\mathcal{V}(U)}(\mathcal{D}u, x)= \frac{d}{dx}Y_{\mathcal{V}(U)}(u, x)
\]
because $|Y_{\mathcal{V}(U)}(u, x)|=|u_n x^{-n-1}|=|u_n|_{[2N]}+2N(n+1)=|u_n|+2Nn=|u|$.

\begin{lemma} 
The set $\{Y_{\mathcal{V}(U)}(u, x) \ | \  u \in U^{[*]} \}$ is a set of mutually dg local dg vertex operators on $\mathcal{V}^{[*]}(U^{[*]})$.
\end{lemma}

\begin{proof}
We know that $u_n \mathbf{1}=0$ for $n \geq 0$, $u \in U^{[*]}$ by construction. Set $v, w \in U^{[*]}$. We verify that, in $\mathcal{U}(\mathcal{L}(U^{[*]})[2N])$, we have the equality $[u_n, v_mw_k]^s=[u_n, v_m]^s w_k+(-1)^{|u_n|_{[2N]}|v_m]_{[2N]}}v_m[u_n, w_k]^s$. By induction, $[u_n, \cdot]^s$ is a derivation of degree $|u_n|_{[2N]}=|u_n|-2N$ of $\mathcal{U}(\mathcal{L}(U^{[*]})[2N])$. Thus, using Formula~\eqref{eq:dg Lie}, we can show that for any $w \in \mathcal{V}^{[*]}(U^{[*]})$ with a decomposition as in Lemma \ref{lem:5.2}, we have $u_n w=0$ when $n$ is large enough.

We can also rewrite Formula~\eqref{eq:dg Lie} as: for any $u, v \in U^{[*]}$, we have
\[
[Y_{\mathcal{V}(U)}(u, x_1), Y_{\mathcal{V}(U)}(v, x_2)]^s=\on{Res}_{x_0}x_2^{-1} \delta\bigg(\frac{x_1-x_0}{x_2}\bigg) Y_{\mathcal{V}(U)}(Y^-(u, x_0)v, x_2).
\]
With the same reasoning as in \cite[(3.1)]{Primc}, we get
\begin{align*}
(x_1-x_2)^k[Y_{\mathcal{V}(U)}(u, x_1),& Y_{\mathcal{V}(U)}(v, x_2)]^s \\[5pt]
&=\sum_{i \geq 0}^{N_0}\frac{(-1)^i}{i!}(x_1-x_2)^k\bigg(\frac{\partial}{\partial x_1}\bigg)^i x_2^{-1} \delta\bigg(\frac{x_1}{x_2}\bigg) Y_{\mathcal{V}(U)}(u_{(i)}v, x_2)
\end{align*}
where the sum is finite because of the truncation condition. The integer $N_0$ depends only on $u$ and $v$, not on $k$. Moreover, we know from \cite[(2.3.13)]{Lepowsky-Li} that
\[
(x_1-x_2)^k\bigg(\frac{\partial}{\partial x_1}\bigg)^i x_2^{-1} \delta\bigg(\frac{x_1}{x_2}\bigg) =0 \text{ if } k >i.
\]
By taking $k > N_0$, we get the desired result.
\end{proof}

We now apply Theorem \ref{thm:5.7.1} to construct a dg vertex algebra structure on $\mathcal{V}^{[*]}(U^{[*]})$. The roles are as follows:
\[
\begin{array}{rccl}
\text{the role of}  & V^{[*]} & \text{is played by} & \mathcal{V}^{[*]}(U^{[*]}), \\[5pt]
\text{the role of}  & \mathbf{1} & \text{is played by} & \mathbf{1}, \\[5pt]
\text{the role of}  & \mathcal{D}_V & \text{is played by} & \mathcal{D}, \\[5pt]
\text{the role of}  & T & \text{is played by} & U^{[*]}=\{u_{-1}\mathbf{1} \ | \ u \in U^{[*]} \}, \\[5pt]
\text{the role of}  & Y_0(\cdot, x) & \text{is played by} & Y_{\mathcal{V}(U)}(\cdot, x).
\end{array}
\]
We obtain the following result:

\begin{theorem}\label{thm:envelop_dg_vertex} 
The set $\{Y_{\mathcal{V}(U)}(u, x) \ | \  u \in U^{[*]} \}$ of mutually dg local dg vertex operators on $\mathcal{V}^{[*]}(U^{[*]})$ generates a dg vertex algebra structure on $\mathcal{V}^{[*]}(U^{[*]})$ with vacuum vector $\mathbf{1}$ and derivation $\mathcal{D}$.
\end{theorem}

\begin{remark}\label{rem:V5}
\begin{enumerate}[wide, nosep]
\item Similarly to the classical case, we can define dg ideals of dg vertex algebras as well as quotients of dg vertex algebras. Such quotients are once again dg vertex algebras. Hence, using Theorem \ref{thm:envelop_dg_vertex} and maximal dg ideals, we can construct simple dg vertex algebras (i.e. without nontrivial dg ideals).
\item Similarly to what we said in Remark \ref{rem:IV18}, there are two functors from the category of dg vertex algebras to the category of super vertex algebras: $\on{VA}_{\on{cohom}}$ and $\on{VA}_{\on{forget}}$. For a dg vertex algebra $V^{[*]}$, the super vertex algebra $\on{VA}_{\on{cohom}}(V^{[*]})$ is obtained by taking the cohomology of the dg complex $V^{[*]}$ and then grouping the odd and even degrees, and $\on{VA}_{\on{forget}}(V^{[*]})$ is the super vertex algebra obtained from $V^{[*]}$ by ignoring the differential and grouping the odd and even degrees. A dg vertex algebra $V^{[*]}$ is called cohomologically simple if $\on{VA}_{\on{cohom}}(V^{[*]})$ is simple as a super vertex algebra. If $\on{VA}_{\on{forget}}(V^{[*]})$ is simple, then it implies that $V^{[*]}$ is simple but we expect the reciprocal not to be true. As for dg vertex Lie algebras, a dg vertex algebra has then three non equivalent notions of simplicity.
\end{enumerate}
\end{remark}

\begin{remark}
The notion of vertex Lie algebra has an equivalent formulation as Lie conformal algebra (see \cite{Kac}). A generalisation to non-linear Lie conformal algebras has been developed (see \cite{DeSole-Kac}) and a universal construction of a vertex algebra from a non-linear Lie conformal algebra, similar to Theorem \ref{thm:envelop_dg_vertex}, has been obtained. We expect that those results can be routinely generalised to the dg setting in a natural fashion, which we will not carry out in the present paper. Moreover, Theorem \ref{thm:envelop_dg_vertex} enables us to construct a dg vertex algebra from a dg vertex Lie algebra, and from a dg Lie algebra (seen in Theorem \ref{thm:dg_conform}). This raises the following question: what are the properties that characterise the dg vertex algebras arising from such a construction? This will be explored in a future paper.
\end{remark}

The following results are obtained by generalising \cite[Proposition 5.4 and Theorem 5.5]{Primc} to the dg setting. The reasoning is similar to the classical case, but the fact that homomorphisms of dg vertex algebras are chain maps needs to be taken into account. The functor $\on{Restr}^{[*]}(-)$ was defined right before Remark \ref{rem:4.3}.

\begin{proposition}\label{prop:morphism_dgLVA} 
The map $\kappa_U:  U^{[*]}  \longrightarrow  \on{Restr}^{[*]}(\mathcal{V}^{[*]}(U^{[*]}))$ is an injective homomorphism of dg vertex Lie algebras.
\end{proposition}

\begin{theorem}\label{thm:unique_extension} 
Let $V^{[*]}$ be a dg vertex algebra and $\varphi : U^{[*]} \longrightarrow \on{Restr}^{[*]}(V^{[*]})$ a homomorphism of dg vertex Lie algebras. Then $\varphi$ extends uniquely to a dg vertex algebra homomorphism $\widetilde{\varphi} : \mathcal{V}^{[*]}(U^{[*]}) \longrightarrow V^{[*]}$.
\end{theorem}

As a direct consequence of Proposition \ref{prop:morphism_dgLVA} and Theorem \ref{thm:unique_extension}, we see that $\mathcal{V}^{[*]}(-)$ and $\on{Restr}^{[*]}(-)$ form an adjoint pair:

\begin{proposition} 
The functor $\mathcal{V}^{[*]}(-)$ is the left adjoint of $\on{Restr}^{[*]}(-)$, i.e., for any dg vertex Lie algebra $U^{[*]}$ and any dg vertex algebra $V^{[*]}$, we have
\[
\on{Hom}_{\on{dgVLA}}(U^{[*]}, \on{Restr}^{[*]}(V^{[*]})) \cong \on{Hom}_{\on{dgVA}}(\mathcal{V}^{[*]}(U^{[*]}), V^{[*]}).
\]
\end{proposition}

We finish this section with a lemma that will help in the construction of dg vertex Lie algebra. 

\begin{lemma}\label{lem:construct_dg Lie}
Let $(U^{[*]}, Y^-(\cdot, x), \mathcal{D})$ be a dg vertex Lie algebra. For any $u, v \in U^{[*]}$ and $n \geq 0$, we have
\[
\kappa_U(u_{(n)}v)= [u_n ,v_{-1}]\mathbf{1}.
\]
\end{lemma}

\begin{proof}
Set $u, v \in U^{[*]}$ homogeneous and $n \geq 0$. Then by Theorem \ref{thm:dg_Lie}, we have
\begin{align*}
[u_n, v_{-1}]\mathbf{1}=\sum_{i =0}^n\binom{n}{i}(u_{(i)}v)_{n-1-i}\mathbf{1}.
\end{align*}
But $n-1-i \geq 0$ when $0 \leq i \leq n-1$. So in $\mathcal{V}^{[*]}(U^{[*]})$ we have $(u_{(i)}v)_{n-1-i}\mathbf{1}=0$. Hence $\kappa_U(u_{(n)}v)=(u_{(n)}v)_{-1}\mathbf{1}=[u_n, v_{-1}]\mathbf{1}$.
\end{proof}

\section{Examples of dg vertex algebras} \label{sec:6}

\subsection{The Virasoro dg vertex algebra}\label{sec:6.1} 
We set $U^{[*]}=\cc[\mathcal{D}] \otimes \omega \oplus \cc \mathbf{c}$ with $|\mathbf{c}|=0$, $|\mathcal{D}|=2N$, $|\omega|=4N$ ($N \in \mathbb{Z}$) and $d_U^{[*]}=0$. We turn $\mathcal{D}$ into a chain map $U^{[*]} \longrightarrow U^{[*]}[2N]$ by setting $\mathcal{D}(\mathbf{c})=0$, and we write $\mathcal{D}^k \omega$ for $\mathcal{D}^k \otimes \omega$. We then define the map $Y^-(\cdot, x)$ on the basis of $U^{[*]}$ as follows:
\begin{itemize}\setlength\itemsep{5pt}
\item $\omega_{(n)} \omega= \left\{
\begin{array}{cl}
\mathcal{D} \omega & \text{if } n=0, \\[5pt]
2 \omega&  \text{if } n=1, \\[5pt]
\frac{1}{2} \mathbf{c} & \text{if } n=3, \\[5pt]
0   &\text{otherwise},
\end{array}
\right.$
\item $\mathbf{c}_{(n)}=0$ for all $n \geq 0$,
\item $\omega_{(n)} \mathbf{c}=0$ for all $n \geq 0$,
\item $\omega_{(n)}((-\mathcal{D})^{(k)}\omega)=\displaystyle \sum_{i=0}^k\binom{n}{i}(-\mathcal{D})^{(k-i)}(\omega_{(n-i)}\omega)$ for all $n, k \geq 0$,
\item $((-\mathcal{D})^{(k)}\omega)_{(n)}=\displaystyle \binom{n}{k}\omega_{(n-k)}$ for all $n, k \geq 0$.
\end{itemize}
where $D^{(k)}=\frac{1}{k!}\mathcal{D}^k$. With this setting, we can verify that the hypothesis in Theorem \ref{thm:dg vertex-Lie_dg Lie} are satisfied. So in order to prove that there is a dg vertex algebra structure on $U^{[*]}$, we only have to show that there is a dg Lie algebra structure on $\mathcal{L}(U^{[*]})$. Using the assumptions above, we check explicitly that the bracket in Formula~\eqref{eq:dg Lie} gives $\mathcal{L}(U^{[*]})$ a dg Lie algebra structure with $|t|=-2N$, and so $U^{[*]}$ is a dg vertex Lie algebra.

According to Theorem \ref{thm:dg_Lie}, we have $n\mathbf{c}_{n-1}=-(\mathcal{D}\mathbf{c})_n=0$, so $\mathbf{c}_n=0$ if $n \neq -1$. Then by applying Formula~\eqref{eq:dg Lie}, we can see that $\mathbf{c}_{-1}$ is central in $\mathcal{L}(U^{[*]})$. Finally, applying Formula~\eqref{eq:dg Lie} for $\omega_m$ and $\omega_n$ leads to
\[
[\omega_m, \omega_n]=(m-n)\omega_{m+n-1}+\frac{m(m-1)(m-2)}{12}\delta_{m+n-2, 0}\mathbf{c}_{-1}.
\]
We then have an isomorphism
\[
\begin{array}{ccc}
\mathcal{L}(U^{[*]})[2N] & \stackrel{\cong}{\longrightarrow} & \on{dgVir} \\[5pt]
\omega_n & \longmapsto & L_{n-1}, \\[5pt]
\mathbf{c}_{-1} & \longmapsto & \mathbf{c}.
\end{array}
\]
where $\on{dgVir}$ is the Virasoro dg Lie algebra with $|L_n|=-2Nn$, $|\mathbf{c}|=0$, $d_{\on{dgVir}}^{[*]}=0$ and Lie bracket of degree $0$.

\begin{remark}\label{rem:construct_Vir}
The relations defining $U^{[*]}$ are obtained by backtracking. For the first three equalities, we use the fact that we want the dg Lie algebra $\mathcal{L}(U^{[*]})$ to be $\on{dgVir}$. By applying Lemma \ref{lem:construct_dg Lie}, we know that $\omega_{(n)}\omega$ can be identified with $[\omega_n, \omega_{-1}]\mathbf{1}$. Finally, we compute the Lie bracket in $\mathcal{L}(U^{[*]})[2N]$ (which by Lemma \ref{lem:5.1} is of degree $0$) by using the relations of $\on{dgVir}$, and we determine what should be $\omega_{(n)}\omega$. The last two equalities will force $[\mathcal{D}, Y^-(u, x)]^s=Y^-(\mathcal{D}u, x)= \frac{d}{dx}Y^-(u, x)$ using $|x|=-2N$.
\end{remark}

By Theorem \ref{thm:envelop_dg_vertex}, we know that $\mathcal{V}^{[*]}(U^{[*]})$ is a dg vertex algebra. For any $c \in \cc$, define
\[
V^{[*]}_{\on{dgVir}}(c, 0)=\mathcal{V}^{[*]}(U^{[*]})/\langle \mathbf{c}_{-1}-c\mathbf{1}\rangle.
\]
The dg ideal $\langle \mathbf{c}_{-1}-c\mathbf{1}\rangle$ is a subcomplex of $\mathcal{V}^{[*]}(U^{[*]})$ because $|\mathbf{c}_{-1}|_{[2N]}=|\mathbf{c}_{-1}|-2N=|\mathbf{c}|=0=|\mathbf{1}|$, so the quotient is a well-defined cochain complex. Furthermore, the dg vertex algebra structure of $\mathcal{V}^{[*]}(U^{[*]})$ passes down to $V^{[*]}_{\on{dgVir}}(c, 0)$. We can thus call $V^{[*]}_{\on{dgVir}}(c, 0)$ the universal Virasoro dg vertex algebra of central charge $c$. Any element of $V^{[*]}_{\on{dgVir}}(c, 0)$ is a linear combination of elements of the form
\[
\omega_{n_1} \cdots \omega_{n_r} \mathbf{1},
\]
with $r \geq 0$, $n_i \leq -1$. The degree of such an element is
\[
|\omega_{n_1} \cdots \omega_{n_r} \mathbf{1}|_{\on{dgVir}}=\sum_{i=1}^r|\omega_{n_i}|_{[2N]}=2N(r-\sum_{i=1}^r n_i),
\]
and the differential on $V^{[*]}_{\on{dgVir}}(c, 0)$ is zero.

\subsection{The Neveu-Schwarz dg vertex algebra} 
We use \cite{Kac-Wang} for the definition of the Neveu-Schwarz Lie algebra. By changing the indices, the algebra is generated by $\mathbf{c}$, $G_m$, $L_n$, with $m, n \in \mathbb{Z}$ with the relations:
\[
\begin{cases}
[L_m, L_n]=\displaystyle (m-n)L_{m+n}+\frac{m^3-m}{12}\delta_{m+n, 0}\mathbf{c}, \\[5pt]
[G_m, L_n]=\displaystyle (m+\frac{1-n}{2})G_{m+n}, \\[5pt]
[G_m, G_n]=\displaystyle  2L_{m+n+1}+\frac{m(m+1)}{3}\delta_{m+n+1, 0}\mathbf{c}, \\[5pt]
[L_m, \mathbf{c}]=[G_m, \mathbf{c}]=0.
\end{cases}
\]
Define $U^{[*]}=\cc[\mathcal{D}] \otimes (\cc \omega \oplus \cc \tau) \oplus \cc \mathbf{c}$ with $|\mathbf{c}|=0$, $|\mathcal{D}|=2N$, $|\tau|=3N$, $|\omega|=4N$ ($N \in \mathbb{Z}$), and $d_U^{[*]}=0$. Similarly to Section \ref{sec:6.1}, we extend $\mathcal{D}$ to $U^{[*]}$ by setting $\mathcal{D}\mathbf{c}=0$. Using the same method as the one described in Remark \ref{rem:construct_Vir}, we set
\smallskip

\setlength{\tabcolsep}{12pt}
\begin{tabular}{ll}
$\bullet$ $\omega_{(n)} \omega=
\left\{\begin{array}{cl}
\mathcal{D} \omega & \text{if } n=0, \\[5pt]
2 \omega&  \text{if } n=1, \\[5pt]
\frac{1}{2} \mathbf{c} & \text{if } n=3, \\[5pt]
0   &\text{otherwise},
\end{array}\right.$  & $\bullet$ $\omega_{(n)} \tau= \left\{
\begin{array}{cl}
\mathcal{D} \tau & \text{if } n=0, \\[5pt]
\frac{3}{2} \tau &  \text{if } n=1, \\[5pt]
0   &\text{otherwise},
\end{array}
\right.$  \\[35pt]
 $\bullet$ $\tau_{(n)} \omega= \left\{
\begin{array}{cl}
\frac{1}{2} \mathcal{D} \tau & \text{if } n=0, \\[5pt]
\frac{3}{2} \tau &  \text{if } n=1, \\[5pt]
0   &\text{otherwise},
\end{array}
\right.$ & $\bullet$ $\tau_{(n)} \tau= \left\{
\begin{array}{cl}
2 \omega & \text{if } n=0, \\[5pt]
\frac{2}{3} \mathbf{c} &  \text{if } n=2, \\[5pt]
0   &\text{otherwise},
\end{array}
\right.$ \\[25pt]
$\bullet$ $\omega_{(n)} \mathbf{c}=\tau_{(n)}\mathbf{c}=0$ for all $n \geq 0$, & $\bullet$ $\mathbf{c}_{(n)}=0$ for all $n \geq 0$, \\
 \end{tabular}

 \begin{itemize}[leftmargin=33.5pt]\setlength\itemsep{5pt}\setlength\itemindent{-5pt}
\item $u_{(n)}((-\mathcal{D})^{(k)}v)=\displaystyle \sum_{i=0}^k\binom{n}{i}(-\mathcal{D})^{(k-i)}(u_{(n-i)}v)$ for all $n, k \geq 0$ and $u, v= \omega, \tau$,
\item $((-\mathcal{D})^{(k)}u)_{(n)}=\displaystyle \binom{n}{k} u_{(n-k)}$ for all $n, k \geq 0$ and $u, v= \omega, \tau$.\end{itemize}

As we did for the Virasoro case, we verify that $\mathcal{L}(U^{[*]})$ with $|t|=-2N$ is the Neveu-Schwarz dg Lie algebra, with an isomorphism
\[
\begin{array}{ccc}
\mathcal{L}(U^{[*]})[2N] & \stackrel{\cong}{\longrightarrow} & \on{dgNS} \\[5pt]
\omega_n & \longmapsto & L_{n-1}, \\[5pt]
\tau_n & \longmapsto & G_{n-1}, \\[5pt]
\mathbf{c}_{-1} & \longmapsto & \mathbf{c}.
\end{array}
\]
where $|L_n|=-2Nn$, $|G_n|=N(-1-2n)$, $|\mathbf{c}|=0$, $d_{\on{dgNS}}^{[*]}=0$, and the Lie bracket is of degree $0$.

By Theorem \ref{thm:dg vertex-Lie_dg Lie}, $U^{[*]}$ is a dg vertex Lie algebra and then by Theorem \ref{thm:envelop_dg_vertex}, $\mathcal{V}^{[*]}(U^{[*]})$ is a dg vertex algebra. Based on the relations in $U^{[*]}$, this dg vertex algebra is the linear span of elements of the form
\[
\omega_{m_1} \cdots \omega_{m_r} \tau_{n_1} \cdots \tau_{n_s} \mathbf{c}_{-1}^{k} \mathbf{1},
\]
with $r, s, k \geq 0$, $m_i, n_i \leq -1$. The degree of such an element is
\begin{align*}
|\omega_{m_1} \cdots \omega_{m_r} \tau_{n_1} & \cdots \tau_{n_s} \mathbf{c}_{-1}^{k} \mathbf{1}|_{\mathcal{V}^{[*]}(U^{[*]})} \\[5pt]
& =\sum_{i=1}^r|\omega_{n_i}|_{[2N]}+\sum_{j=1}^s|\tau_{n_j}|_{[2N]}=2Nr+Ns-2N\big(\sum_{i=1}^r m_i+\sum_{j=1}^s n_j\big),
\end{align*}
and its differential is zero.

\subsection{The affine dg vertex algebra} 
\subsubsection{Affine dg Lie algebra}

Let $(\mathfrak{g}^{[*]}, d_{\mathfrak{g}}^{[*]})$ be a dg Lie algebra with $\on{deg}([\cdot, \cdot])=p$ and assume that $(\mathfrak{g}^{[*]}, d_{\mathfrak{g}}^{[*]})$ is equipped with a homogeneous linear map
\begin{align}\label{bilinear_map}
\begin{array}{rccc}
\langle \cdot, \cdot \rangle : & \mathfrak{g}^{[*]} \otimes \mathfrak{g}^{[*]} & \longrightarrow & \cc[2p] \\[5pt]
 & a \otimes b & \longmapsto & \langle a, b \rangle
 \end{array}
 \end{align}
 such that for any homogeneous $a, b \in \mathfrak{g}^{[*]}$:
\begin{empheq}[left = \empheqlbrace]{align}
& \langle d_{\mathfrak{g}}(a), b \rangle +(-1)^{|a|+2p} \langle a, d_{\mathfrak{g}}(b) \rangle=0, \label{condition_1} \\[5pt]
&\langle a, b \rangle=(-1)^{|a|+2p}\langle b, a \rangle, \label{condition_2} \\[5pt]
&\langle [a, b], c \rangle=\langle a, [b, c] \rangle. \label{condition_3}
\end{empheq}
Here $\cc$ is seen as a cochain complex concentrated in degree $0$. 

\begin{remark}
By writing \eqref{bilinear_map} as
\begin{align*}
\begin{array}{rccc}
f : & \mathfrak{g}^{[*]}& \longrightarrow & \on{Hom}^{[*]}(\mathfrak{g}^{[*]}, \cc[2p])=(\mathfrak{g}^{[*]})'[2p] \\[5pt]
 & a & \longmapsto & f_a=\langle a, \cdot \rangle
 \end{array}
 \end{align*}
where $(\mathfrak{g}^{[*]})'=\on{Hom}(\mathfrak{g}^{[*]}, \cc)$ is the dual of $\mathfrak{g}^{[*]}$, we see that \eqref{condition_1} is equivalent to $f$ being a chain map. Moreover, \eqref{condition_2} expresses the symmetry of $\langle \cdot, \cdot \rangle$ and \eqref{condition_3} its invariance.
\end{remark}

\begin{remark}\label{rem:6.2}
Because of the chain map condition, we know that for $a$ and $b$ dg homogeneous, we have $\langle a, b \rangle \in \cc[2p]^{[|a|+|b|]}=\cc^{[|a|+|b|+2p]}$. Hence $\langle a, b \rangle \neq 0$ only if
\[
|a|+|b|+2p=0.
\] 
\end{remark}

We set
\[
\widehat{\mathfrak{g}}^{[*]}=\mathfrak{g}^{[*]} \otimes \cc[t, t^{-1}] \oplus \cc K
\]
where $\cc[t, t^{-1}]$ is seen as a cochain complex with $|t|=-2N$, $N \in \mathbb{Z}$, and $|K|=-p$. We also impose $d_{\widehat{\mathfrak{g}}}(K)=0$. Then $\widehat{\mathfrak{g}}^{[*]}$ is a cochain complex with $d_{\widehat{\mathfrak{g}}}(a \otimes t^n)=d_{\mathfrak{g}}(a) \otimes t^n$. We define a bracket on $\widehat{\mathfrak{g}}^{[*]}$ as follows:
  \begin{itemize}\setlength\itemsep{5pt}
\item $[a \otimes t^m, b \otimes t^n]=[a, b] \otimes t^{m+n}+m \langle a, b \rangle \delta_{m+n, 0} K$,
\item $1 \otimes K$ central.
\end{itemize}

\begin{proposition}\label{prop:dg g_hat}
The cochain complex $(\widehat{\mathfrak{g}}^{[*]}, d_{\widehat{\mathfrak{g}}}^{[*]})$ with the Lie bracket defined above is a dg Lie algebra and the Lie bracket is of degree $p$.
\end{proposition}

\begin{proof}
In this proof we will write $a_m$ instead of $a \otimes t^m$ to simplify notations. Then $|a_m|=|a|-2Nm$.

We have $m \langle a, b \rangle \delta_{m+n, 0} K \neq 0$ only if $m+n=0$ and $|a|+|b|+2p=0$. Then $|[a, b]_{m+n}|=-p=|m \langle a, b \rangle \delta_{m+n, 0} K|$, so $|[a_m, b_n]|=-p=|a|-2Nm+|b|-2Nn+p=|a_m|+|b_n|+p$. Otherwise $|[a_m, b_n]|=|[a, b]_{m+n}|=|a|+|b|+p-2N(m+n)=|a_m|+|b_n|+p$. We see that the bracket $[\cdot, \cdot]: \widehat{\mathfrak{g}}^{[*]} \times \widehat{\mathfrak{g}}^{[*]} \longmapsto \widehat{\mathfrak{g}}^{[*]}$ is a bilinear map of degree $p$. In order to finish the proof, we verify explicitly the relations (i), (ii) and (iii) of Definition \ref{def:dg_Lie_algebra}.
\end{proof}

\begin{remark}
We can define a chain map 
\[
\begin{array}[t]{rccc}
\langle \cdot, \cdot \rangle : &  \cc[t, t^{-1}] \otimes \cc[t, t^{-1}] & \longrightarrow & \cc \\[5pt]
 & t^m \otimes t^n & \longmapsto & \on{Res}_t(\frac{d}{dt}(t^m)t^n)
 \end{array}
 \]
 and then $m\delta_{m+n, 0} =  \langle t^m, t^n \rangle$. So we could write $[a \otimes t^m, b \otimes t^n]=[a, b] \otimes t^{m+n}+ \langle a, b \rangle \langle t^m, t^n \rangle \otimes K$.
\end{remark}

\subsubsection{Constructing the dg vertex Lie algebra}\label{subsubsec:6.3.2}

Let $(\mathfrak{g}^{[*]}, d_{\mathfrak{g}}^{[*]})$ be a dg Lie algebra with Lie bracket of degree $p=-2N$ and with a chain map $\langle \cdot, \cdot \rangle :  \mathfrak{g}^{[*]} \otimes \mathfrak{g}^{[*]}  \longrightarrow  \cc[-4N]$ satisfying the relations \eqref{condition_1}, \eqref{condition_2} and \eqref{condition_3}. Set $U^{[*]}=\cc[\mathcal{D}] \otimes \mathfrak{g}^{[*]}\oplus \cc\bold{K}$ with $|\bold{K}|=0$. We see $\cc[\mathcal{D}]$ as a cochain complex with $|\mathcal{D}|=2N$, and we write $\mathcal{D}^k a=\mathcal{D}^k \otimes a$ for $a \in \mathfrak{g}^{[*]}$ and $k \geq 0$. We verify that $d_{U}(\mathcal{D}^k a)=\mathcal{D}^k(d_{\mathfrak{g}}(a))$. We then extend $\mathcal{D}$ to $U^{[*]}$ by setting $\mathcal{D}( \bold{K})=0$ and give $U^{[*]}$ a cochain complex structure by imposing $d_U( \bold{K})=0$. For any $a, b \in \mathfrak{g}^{[*]}$, we then define the map:
\begin{align*}
\begin{array}{cccc}
Y^-(\cdot, x): & U^{[*]} & \longrightarrow & x^{-1}\End^{[*]}(U^{[*]}))[[x^{-1}]] \\
           & u & \longmapsto      & Y^-(u, x)=\displaystyle \sum_{n \in \mathbb{N}}u_{(n)} x^{-n-1}
\end{array}
\end{align*}
by:
\begin{itemize}\setlength\itemsep{5pt}
\item $a_{(n)}b= \left\{
\begin{array}{cl}
[a, b] & \text{if } n=0, \\[5pt]
\langle a, b \rangle \bold{K} &  \text{if } n=1, \\[5pt]
0   & \text{if }n \geq 2,
\end{array}
\right.$
\item $a_{(n)}(\bold{K})=0$ for all $n \geq 0$,
\item $(\bold{K})_{(n)}=0$ for all $n \geq 0$,
\item $a_{(n)}((-\mathcal{D})^{(k)}b)=\displaystyle \sum_{i=0}^k\binom{n}{i}(-\mathcal{D})^{(k-i)}(a_{(n-i)}b)$ for all $n, k \geq 0$,
\item $((-\mathcal{D})^{(k)}a)_{(n)}=\displaystyle \binom{n}{k} a_{(n-k)}$ for all $n, k \geq 0$.
\end{itemize}

\begin{proposition}\label{prop:6.6}
There is an isomorphism of dg Lie algebras
\[
\begin{array}[t]{ccc}
\mathcal{L}(U^{[*]})   &\stackrel{\cong}{\longrightarrow} & \widehat{\mathfrak{g}}^{[*]} \\[5pt]
a_n & \longmapsto & a \otimes t^n, \\[5pt]
\bold{K}_{-1} & \longmapsto & K.
 \end{array}
 \]
and $U^{[*]}$ is a dg vertex Lie algebra. 
\end{proposition}

\begin{proof}
With an approach similar to the proof of Proposition \ref{prop:dg g_hat}, we prove that $\mathcal{L}(U^{[*]})$ is a dg Lie algebra with a Lie bracket of degree $-2N$. With the relations we found for the Lie bracket in $\mathcal{L}(U^{[*]})$, the isomorphism above becomes clear.

We can verify that the hypothesis of Theorem \ref{thm:dg vertex-Lie_dg Lie} are satisfied. As we just proved that $\mathcal{L}(U^{[*]})$ is a dg Lie algebra, it follows that $(U^{[*]}, Y^-(\cdot, x), \mathcal{D})$ is a dg vertex Lie algebra.
\end{proof}

A consequence of Theorem \ref{thm:envelop_dg_vertex} is that $\mathcal{V}^{[*]}(U^{[*]})$ is a dg vertex algebra. For any $k \in \cc$, define
\[
V^{[*]}_{\widehat{\mathfrak{g}}}(k, 0)=\mathcal{V}^{[*]}(U^{[*]})/\langle \bold{K}_{-1}-k\mathbf{1}\rangle.
\]
It is a well-defined dg vertex algebra and we call it the universal affine dg vertex algebra of level $k$ based on $\mathfrak{g}^{[*]}$. The differential on $V^{[*]}_{\widehat{\mathfrak{g}}}(k, 0)$ is non trivial. For example:
\[
\begin{array}{rcl}
d_{V_{\widehat{\mathfrak{g}}}(k, 0)}(a_{-2}b) & = & d_{\mathcal{L}(U)[2N]}(a_{-2})b+(-1)^{|a_{-2}|_{[2N]}}a_{-2}d_{\mathcal{U}(U)[2N]}(b_{-1}\mathbf{1}) \\[10pt]
 & = & d_{\mathfrak{g}}(a)_{-2}b+(-1)^{|a|}a_{-2}d_{\mathfrak{g}}(b).
\end{array}
\]

In particular, if we start the construction with $\mathfrak{h}^{[*]}$ a graded commutative dg Lie algebra, then $V^{[*]}_{\widehat{\mathfrak{h}}}(k, 0)$ is a dg version of the Heisenberg vertex algebra.

In fact, the dg vertex algebra $V^{[*]}_{\widehat{\mathfrak{g}}}(k, 0)$ has some extra structure. Given a finite dimensional dg Lie algebra $\mathfrak{g}^{[*]}$ with bracket of degree $p=-2N$ and a non degenerate bilinear map $\langle \cdot, \cdot \rangle$ as in \eqref{bilinear_map}, set $\{a^{(i)}\}_{1 \leq i \leq \on{dim} \mathfrak{g}^{[*]}}$ and $\{b^{(i)}\}_{1 \leq i \leq \on{dim} \mathfrak{g}^{[*]}}$ dual bases of $\mathfrak{g}^{[*]}$ with respect to $\langle \cdot, \cdot \rangle$, i.e.,
\[
\langle b^{(i)}, a^{(j)} \rangle=\delta_{i, j}
\]
for all $1 \leq i, j \leq \on{dim} \mathfrak{g}^{[*]}$. We then write $\Omega=\displaystyle \sum_{i=1}^{\on{dim} \mathfrak{g}^{[*]}}a^{(i)}b^{(i)} \in \mathcal{U}(\mathfrak{g}^{[*]}[2N])$.

\begin{theorem}\label{thm:dg_conform}
Let $\mathfrak{g}^{[*]}$ be a finite dimensional dg Lie algebra with Lie bracket of degree $-2N$, $N \in \mathbb{Z}$, and equipped with a non degenerate bilinear form $\langle \cdot, \cdot \rangle$ satisfying \eqref{bilinear_map} and the relations \eqref{condition_1}, \eqref{condition_2}, \eqref{condition_3}. Set $\{a^{(i)}\}_{1 \leq i \leq \on{dim} \mathfrak{g}^{[*]}}$ and $\{b^{(i)}\}_{1 \leq i \leq \on{dim} \mathfrak{g}^{[*]}}$ dual bases of $\mathfrak{g}^{[*]}$ with respect to $\langle \cdot, \cdot \rangle$. Assume that $\Omega$ acts on $\mathfrak{g}^{[*]}$ as a scalar $2h^\vee \in \cc$. Then for any $k \neq -h^\vee$, the affine dg vertex algebra $V^{[*]}_{\widehat{\mathfrak{g}}}(k, 0)$ of level $k$ is a dg vertex operator algebra of central charge 
\[
c_{V^{[*]}_{\widehat{\mathfrak{g}}}(k, 0)}=\frac{k \on{sdim} \mathfrak{g}^{[*]}}{k+h^\vee}
\]
with conformal vector
\[
\omega=\frac{1}{2(k+h^\vee)}\sum_{i=1}^{\on{dim} \mathfrak{g}^{[*]}} a^{(i)}_{-1}b^{(i)}_{-1}\mathbf{1} \in V^{[4N]}_{\widehat{\mathfrak{g}}}(k, 0),
\]
where $\displaystyle \on{sdim} \mathfrak{g}^{[*]}=\sum_{n \in \mathbb{Z}} \on{dim} \mathfrak{g}^{[2n]}-\sum_{n \in \mathbb{Z}} \on{dim} \mathfrak{g}^{[2n+1]}$.
\end{theorem}

\begin{proof}
The reasoning is similar to \cite[Theorem 5.7]{Kac} and \cite[Theorem 6.2.16]{Lepowsky-Li}. For the degree of $\omega$, we see that in $V^{[*]}_{\widehat{\mathfrak{g}}}(k, 0)$, for any $1 \leq i \leq \on{dim} \mathfrak{g}^{[*]}$, we have
\[
\begin{array}{rcl}
|a^{(i)}_{-1}b^{(i)}_{-1}\mathbf{1}|_{V^{[*]}_{\widehat{\mathfrak{g}}}(k, 0)} & = & |a^{(i)}_{-1}|_{\mathcal{L}(U^{[*]})[2N]}+|b^{(i)}_{-1}|_{\mathcal{L}(U^{[*]})[2N]} \\[5pt]
 & = & |a^{(i)}|_{\mathfrak{g}^{[*]}}-(-2N)-2N+|b^{(i)}|_{\mathfrak{g}^{[*]}}-(-2N)-2N \\[5pt]
  & = &|a^{(i)}|_{\mathfrak{g}^{[*]}}+|b^{(i)}|_{\mathfrak{g}^{[*]}} \\[5pt]
   & = &4N,
\end{array}
\]
by Remark \ref{rem:6.2}. Hence $|\omega|=4N$.
\end{proof}

\begin{remark}
If the super Lie algebra $\on{LA}_{forget}(\mathfrak{g}^{[*]})$ mentionned in Remark \ref{rem:IV18} is basic and simple, then a result of Kac (see \cite[Section 3.4]{Kac2}) states that there exists a supersymmetric non-degenerate invariant bilinear form on $\on{LA}_{forget}(\mathfrak{g}^{[*]})$ (see \cite[Proposition 4.3]{Fattori-Kac} for a detailed list). However it is not clear if this form satisfies  \eqref{bilinear_map} as well as Relation \eqref{condition_1}  for $\mathfrak{g}^{[*]}$. The possibility of lifting the results in \cite[Theorem 4.2 and Proposition 4.3]{Fattori-Kac} to the dg setting is an interesting question but would require more work.
\end{remark}

\begin{remark}
We know that for a dg vertex algebra $V^{[*]}$, the vertex operator $Y(\cdot, x)$ is homogeneous of degree $0$ so $|v_n|=|v|-2N(n+1)$ for any $v \in V^{[*]}$ and $n \in \mathbb{Z}$. It follows that in $V^{[*]}_{\widehat{\mathfrak{g}}}(k, 0)$, we have $|L(n)|=|\omega_{n+1}|=|\omega|-2N(n+2)=-2Nn$, and thus $L(n) \in \on{End}^{[-2nN]}(V^{[*]}_{\widehat{\mathfrak{g}}}(k, 0))$. In particular, $L(0) \in \on{End}^{[0]}(V^{[*]}_{\widehat{\mathfrak{g}}}(k, 0))$.
\end{remark}

\begin{remark}
In Theorem \ref{thm:dg_conform}, we modified the definition of a conformal vector of a dg vertex operator algebra given in \cite[Section 3.2]{Caradot-Jiang-Lin-4} by not requiring $d_V(\omega)=0$. Furthermore, the degree of $\omega$ is not $4$ but $4N$, as we only considered the case $N=1$ in \cite{Caradot-Jiang-Lin-4}.
\end{remark}

\begin{remark}
Let $(\mathfrak{g}^{[*]}, d_{\mathfrak{g}^{[*]}})$ be a dg Lie algebra. Then its cohomology $H^{[*]}(\mathfrak{g}^{[*]}, d_{\mathfrak{g}^{[*]}})$ is a graded Lie algebra. Let $\langle \cdot, \cdot \rangle$ be a bilinear map satisfying \eqref{bilinear_map}--\eqref{condition_3}. We can verify that for any $a, b \in \on{Ker}d_{\mathfrak{g}^{[*]}}$, we have $\langle a+\on{Im}d_{\mathfrak{g}^{[*]}}, b+\on{Im}d_{\mathfrak{g}^{[*]}} \rangle=\langle a, b \rangle$, so $\langle \cdot, \cdot \rangle$ defines a bilinear map on $H^{[*]}(\mathfrak{g}^{[*]}, d_{\mathfrak{g}^{[*]}})$. For any $k \in \cc$, we can define the dg vertex algebra $V^{[*]}_{\widehat{\mathfrak{g}}}(k, 0)$ and its $C_2$-dg algebra $R(V^{[*]}_{\widehat{\mathfrak{g}}}(k, 0))$ (see \cite{Caradot-Jiang-Lin-4}), as well as the graded vertex algebra $V^{[*]}_{\widehat{H(\mathfrak{g}, d_{\mathfrak{g}})}}(k, 0)$ (the differential is zero). It then becomes natural to wonder if there are isomorphisms $V^{[*]}_{\widehat{H(\mathfrak{g}, d_{\mathfrak{g}})}}(k, 0) \cong H^{[*]}(V^{[*]}_{\widehat{\mathfrak{g}}}(k, 0))$ and $R(V^{[*]}_{\widehat{H(\mathfrak{g}, d_{\mathfrak{g}})}}(k, 0)) \cong H^{[*]}(R(V^{[*]}_{\widehat{\mathfrak{g}}}(k, 0)))$.
\end{remark}

\begin{remark}
It is likely that, just as in the classical case, we have $R(V^{[*]}_{\widehat{\mathfrak{g}}}(k, 0))=\on{Sym}(\mathfrak{g}^{[*]}, d_{\mathfrak{g}^{[*]}})$ as dg Poisson algebras. It would then lead to a differential graded version of the Koszul duality, i.e. $\on{Ext}_{R(V^{[*]}_{\widehat{\mathfrak{g}}}(k, 0))}^{[*]}(\mathbf{k}, \mathbf{k})=\Lambda^{[*]}(\mathfrak{g}^{[*]}, d_{\mathfrak{g}^{[*]}})$.
\end{remark}

\begin{remark}
For any $N \in \mathbb{Z}$, we write $\on{dgLA}_{-2N}$ for the category of dg Lie algebras with bracket of degree $-2N$, $\on{dgVLA}_{-2N}$ for the category of dg vertex Lie algebra with $|x|=-2N$, and $(\on{dgLA}_{-2N}, \langle \cdot, \cdot \rangle_{-4N})$ for the category of dg Lie algebras with Lie bracket of degree $-2N$ and with a bilinear form of degree $-4N$ satisfying Equations \eqref{condition_1}, \eqref{condition_2} and \eqref{condition_3}. In the proof of Proposition \ref{prop:6.6}, we have established a commutative diagram of functors
 \begin{center}
 \begin{tikzpicture}[scale=.9, transform shape]
\tikzset{>=stealth}
\node (1) at (0,0) []{$(\on{dgLA}_{-2N}, \langle \cdot, \cdot \rangle_{-4N})$};
\node (2) at (4,0) []{$\on{dgVLA}_{N}$};
\node (3) at (4,-2) []{$\on{dgLA}_{-2N}$};
\draw[->]  (1) -- node[] {} (2);
\draw[->]  (2) -- node[right] {$\mathcal{L}(-)$} (3);
\draw[->]  (1) -- node[left, below] {$\widehat{(-)}$} (3);
\end{tikzpicture}
\end{center}
where the horizontal functor corresponds to the construction $\mathfrak{g}^{[*]} \longrightarrow U^{[*]}$ described in Section \ref{subsubsec:6.3.2}.
\end{remark}

\section*{Acknowledgements}
The first author was supported by the NFSC Grant No. 12250410252. He also appreciated the hospitality of the Department of Mathematics at Kansas State University during his visit in the fall of 2023. 

The second author was supported by the NSFC Grant No. 12171312.



\begin{thebibliography}{99999}\frenchspacing
\label{reference}
\bibitem[\sf An]{Andre}M. Andr\'e, {\em Homologie des Alg\`{e}bres Commutatives}, Springer, Berlin 1974.

\bibitem[\sf Ar]{Arakawa} T. Arakawa. {\em  Representation theory of W-algebras}. Invent. Math. {\bf 169} (2007), no. 2, 219--320.

\bibitem[\sf AH]{Avramov-Halperin} L. Avramov and S. Halperin, {\em  Through the looking glass: a dictionary between rational homotopy theory and local algebra}, in  Algebra, algebraic topology and their interactions (Stockholm, 1983), 1--27, Lecture Notes in Math., {\bf 1183}, Springer, Berlin, 1986.

\bibitem[\sf B1]{Butson1} D. Butson, {\em Equivariant localization in factorization homology and applications in mathematical physics I: Foundations}, arXiv:2011.14988 (2000).

\bibitem[\sf B2]{Butson2} D. Butson, {\em Vertex algebras from divisors on Calabi-Yau threefolds}, arXiv:2312.03648 (2023).

\bibitem[\sf BLM]{BLM} A. Bojko, W. Lim, and M. Moreira, {\em Virasoro constraints on moduli of sheaves and vertex algebras}. arXiv:2210.05266 (2022).

\bibitem[\sf CJL1]{Caradot-Jiang-Lin-1} A. Caradot, C. Jiang, and Z. Lin, {\em Yoneda algebras of the triplet vertex operator algebra.} J. Algebra {\bf 633} (2023), 425--463.

\bibitem[\sf CJL2]{Caradot-Jiang-Lin-2} A. Caradot, C. Jiang, and Z. Lin, {\em Cohomological varieties  associated to vertex operator algebras}, Adv. in Math. {\bf 447} (2024), 109699.

\bibitem[\sf CJL3]{Caradot-Jiang-Lin-4} A. Caradot, C. Jiang, and Z. Lin, {\em Differential graded vertex operator algebras and their Poisson algebras}. J. Math. Phys. {\bf 64} (2023), no. 12, Paper No. 121702, 35 pp.


\bibitem[\sf DK]{DeSole-Kac} A. De Sole and V.G. Kac, {\em Freely generated vertex algebras and non-linear lie conformal algebras}. Comm. Math. Phys. {\bf 254} (2005), 659--694.

\bibitem[\sf FK]{Fattori-Kac} D. Fattori and V.G. Kac, {\em Classification of finite simple lie conformal superalgebras}. J. Algebra {\bf 258} (2002), no. 1, 23--59. Special Issue in Celebration of Claudio Procesi's 60th Birthday.

\bibitem[\sf FBZ]{Frenkel-Ben-Zvi} E. Frenkel and D. Ben-Zvi, {\em Vertex Algebras and Algebraic Curves},  2nd Edition, Mathematical Surveys and monographs, {\bf 88} (2004), Amer. Math. Soc. 

\bibitem[\sf J]{Joyce} D. Joyce,  {\em Ringel-Hall style vertex algebra and Lie algebra structures on the homology of moduli spaces}  (unpublished).

\bibitem[\sf K1]{Kac} V.G. Kac, {\em Vertex algebras for beginners}. University Lecture Series, {\bf 10}. American Mathematical Society, Providence, RI, 1997. viii+141 pp.

\bibitem[\sf K2]{Kac2} V.G. Kac, {\em Classification of supersymmetries}. Proceedings of the International Congress of Mathematicians
(Beijing, 2002), Higher Ed. Press, Beijing, pages 319--344, 2002.

\bibitem[\sf KW]{Kac-Wang} V.G. Kac and W. Wang, {\em Vertex operator superalgebras and their representations}. Mathematical aspects of conformal and topological field theories and quantum groups (South Hadley, MA, 1992), 161--191, Contemp. Math., {\bf 175}, Amer. Math. Soc., Providence, RI, 1994. 

\bibitem[\sf LL]{Lepowsky-Li} J. Lepowsky and H. Li, {\em Introduction to Vertex Operator Algebras and Their Representations}, Progress in Mathematics, Vol. {\bf 227}, Birkh\"user Boston, Inc., Boston, MA, 2004.

\bibitem[\sf L]{Li} H. Li, {\em Local systems of vertex operators, vertex superalgebras and modules}. J. Pure Appl. Algebra {\bf 109} (1996), no. 2, 143--195. 

\bibitem[\sf LWZ]{Lu-Wang-Zhang}J. L\"u, X. Wang, and G. Zhuang, {\em DG Poisson algebra and its universal enveloping algebra}, Sci. China Math. {\bf 59} (2016), no. 5, 849-860.

\bibitem[\sf P]{Primc} M. Primc, {\em Vertex algebras generated by Lie algebras}, J. Pure Appl. Algebra {\bf 135} (1999), no. 3, 253--293. 

\bibitem[\sf Q]{Quillen} D.G. Quillen, {\em Homotopical algebra} Lecture Notes in Math. {\bf 43} Springer-Verlag, Berlin, Heidelberg-New York, 1967.

\bibitem[\sf T]{Tate} J. Tate, {\em Homology of Noetherian rings and local rings}, Illinois J. Math. {\bf 1} (1957), 14--27. 

\bibitem[W]{Weibel} C. Weibel, {\em An Introduction to Homological Algebra}. Cambridge Studies in Advanced Mathematics, {\bf 38}. Cambridge University Press, Cambridge, 1994.
\end{thebibliography}
\end{document}